\documentclass[11pt,a4paper]{article}
\usepackage[DIV12]{typearea}

\usepackage[ngerman,english]{babel}
\usepackage[utf8]{inputenc}

\usepackage{stmaryrd}

\usepackage[affil-it]{authblk}

\usepackage[bookmarks=true]{hyperref}

\usepackage{xparse}
\usepackage{amsmath, amsfonts, amssymb, amsthm}
\usepackage{verbatim}
\usepackage{mathtools}
\usepackage{enumitem}
\usepackage{calrsfs}

\usepackage{tcolorbox}

\usepackage[noadjust]{cite}

\usepackage[capitalise]{cleveref}

\crefname{equation}{}{}

\usepackage{subcaption}

\usepackage{tikz}
\usepackage{standalone}
\usetikzlibrary{decorations.markings}
\usetikzlibrary{decorations.pathreplacing}

\usepackage{calc}

\usepackage{pgf}

\newcommand{\R}{\mathbb{R}}

\newcommand{\N}{\mathbb{N}}

\newcommand{\C}{\mathbb{C}}

\newcommand{\M}{\mathcal{M}}

\DeclareDocumentCommand\dd{ o g d() }{
	\IfNoValueTF{#2}{
		\IfNoValueTF{#3}
			{\mathrm{d}\IfNoValueTF{#1}{}{^{#1}}}
			{\mathinner{\mathrm{d}\IfNoValueTF{#1}{}{^{#1}}\argopen(#3\argclose)}}
		}
		{\mathinner{\mathrm{d}\IfNoValueTF{#1}{}{^{#1}}#2} \IfNoValueTF{#3}{}{(#3)}}
	}

\newcommand{\dx}{\dd{x}}
\newcommand{\dy}{\dd{y}}
\newcommand{\dz}{\dd{z}}

\newcommand{\dv}{\dd{v}}
\newcommand{\dw}{\dd{w}}
\newcommand{\dX}{\dd{X}}
\newcommand{\dV}{\dd{V}}
\newcommand{\deta}{\dd{\eta}}
\newcommand{\dzeta}{\dd{\zeta}}

\newcommand{\dxi}{\dd{\xi}}
\newcommand{\ds}{\dd{s}}
\newcommand{\dt}{\dd{t}}

\newcommand{\del}{\partial}
\newcommand{\eps}{\varepsilon}

\newcommand{\D}{\mathcal{D}} 
\newcommand{\Ra}{\mathcal{R}} 

\providecommand{\U}[1]{\protect\rule{.1in}{.1in}}

\renewcommand{\Re}{\operatorname{Re}}

\newcommand{\vcc}{\vcentcolon}

\DeclarePairedDelimiter\abs{\lvert}{\rvert}
\DeclarePairedDelimiter\norm{\Vert}{\rVert}

\DeclareMathOperator*{\supp}{supp}

\theoremstyle{plain}
\newtheorem{theorem}{Theorem}[section]
\newtheorem{lemma}[theorem]{Lemma}
\newtheorem{proposition}[theorem]{Proposition}

\theoremstyle{definition}
\newtheorem{definition}[theorem]{Definition}

\theoremstyle{remark}
\newtheorem{remark}[theorem]{Remark}

\numberwithin{equation}{section}

\title{Self-similar asymptotic behavior for the solutions of a linear coagulation equation}
\author{ Barbara Niethammer, 
Alessia Nota, 
Sebastian Throm, 
Juan J. L. Vel\'azquez }

\date{\today}

\begin{document}
 
\maketitle

\begin{abstract}
In this paper we consider the long time asymptotics of a linear version of the Smoluchowski equation which describes the evolution 
of a tagged particle moving at constant speed in a random
distribution of fixed particles. The volumes $v$ of the particles are independently distributed according to a probability distribution which decays asymptotically 
as a power law $v^{-\sigma}$. The validity of the equation has been rigorously proved in \cite{NoV} for values of the exponent $\sigma>3$. 
The solutions of this equation  display a rich structure of different asymptotic behaviours according to the different values of the exponent 
$\sigma$. 
Here we show that for $\frac{5}{3}<\sigma<2$ the linear Smoluchowski equation is well posed and that 
there exists a unique self-similar profile which is asymptotically stable.

\end{abstract}

\bigskip

\tableofcontents

\section{Introduction}

The goal of this paper is to study the dynamics of a tagged particle
which moves at constant speed in a distribution of scatterers with different
radii coalescing with them upon collisions.

We denote by $f=f(V,t)$ a probability measure
which describes the distribution of volumes of the tagged particle. Let us
 also assume that the scatterers are distributed in space by means of a Poisson
distribution with constant rate and their sizes $v>0$ according to a
probability measure $G(v)  dv.$ It has been derived in  \cite{NoV}
that the evolution of $f(V,t)$ is given, using a suitable time
scale, by
\begin{equation} \label{eq:lin:coag}
\partial_{t}f(V,t)  =\int_{0}^{V}\dv G(v)\left((V-v)^{\frac{1}{3}}+v^{\frac{1}{3}}\right)  ^{2}f(V-v,t)
-\int_{0}^{\infty}\dv G(v)  \left(  V^{\frac{1}{3}%
}+v^{\frac{1}{3}}\right)  ^{2}f(V,t)\,. 
\end{equation}

Equation \eqref{eq:lin:coag} has been rigorously derived in \cite{NoV} taking as
starting point the dynamics of a tagged particle among a set of scatterers
randomly distributed which fills a volume fraction $\phi$ which converges to
zero under the assumption $\int_{0}^{\infty}G(v)  v^{\gamma}%
dv<\infty$, with $\gamma>2$ (cf.  \cite{NoV}, Theorem 2.4.). This assumption for the function $G(v)$ has been
made in \cite{NoV} 
for technical reasons and it is not optimal. Indeed, we remark that to define solutions of Eq. \eqref{eq:lin:coag} it is enough
to assume that $\int_{0}^{\infty}G(v) (1+v^{\frac{2}{3}})%
dv<\infty$. 

Equation \eqref{eq:lin:coag} can be thought of as a linear version of the
Smoluchowski coagulation equation in shear flows.  
Actually, the relation between
\eqref{eq:lin:coag} and the Smoluchowski coagulation equation for particles
moving in a shear flow is similar to the relation between the kinetic
equations describing the dynamics of a tagged particle moving in a Lorentz gas
and Boltzmann's equation for interacting systems of particles. 
The derivation of suitable kinetic equations and diffusive equations starting from the Lorentz model has been extensively discussed.
See for instance \cite{BNP, BNPP, DP, G, No, Sp}.

We are interested in the dynamics of the solutions of \eqref{eq:lin:coag} for
different choices of the function $G(v)$.  In particular, we expect 
self-similar behaviour of the solution for large times 
if $G(v)  $ behaves like a power law for large $v$ or if it decays
sufficiently fast. 
More precisely, we will assume in the following that 
\begin{equation}\label{ass:G}
G\in L^{1}(0,\infty),\quad  G(v)  \leq C_0 v^{-\sigma}  \quad \text{for}\;\; v>0,
\end{equation}
for some suitable $C_0>0$. 
Moreover, $G$ is supposed to behave asymptotically as a power law, i.e., withouth loss of generality,
\begin{equation}\label{eq:behaviour:G}
 \lim_{v\to \infty}v^{\sigma} G(v)=1\,.
\end{equation}
Equation (\ref{eq:lin:coag}) under the assumptions \eqref{ass:G}, \eqref{eq:behaviour:G}
yields a rich structure of asymptotic behaviours for the solutions depending on
the parameter $\sigma.$ For all the ranges of exponents the underlying
dynamics of the tagged particle suggests that its average size must increase.
However the specific rate of growth depends strongly on  $\sigma.$

We first remark that for $\sigma\leq\frac{5}{3}$ we cannot expect to have well
defined solutions of (\ref{eq:lin:coag}) due to the divergence of the integral
$\int_{0}^{\infty} G(  v)  (  V^{\frac{1}{3}}+v^{\frac{1}{3}})  ^{2}\dv.$ This divergence suggests that there are infinitely many
coalescences of the tagged particle with large particles of size $v$ in
arbitrarily short times, and since the volume of the tagged particle and the
scatterer are added in each collision this would result in an instantaneous
explosive growth of the tagged particle to one having infinite volume.
We also notice that the question of the divergence of this integral has an interesting connection with the theory of Continuum Percolation
(see e.g. \cite{Gr},\cite{MR}).

Therefore, in this paper, we will restrict our analysis to the case
$\sigma>\frac{5}{3}.$ It turns 
that for $\frac{5}{3}<\sigma<2,$ the function $f(V,t)  $  behaves in a self-similar manner, with a characteristic width of
order $V\approx t^{\mu}$ with $\mu=\left(  \sigma-\frac{5}{3}\right)  ^{-1}.$
  If $\sigma>2$ we obtain that with probability close to one we
have that all the particles behave like $V\sim C_{0}t^{3}$ for some constant
$C_{0}>0.$ The onset of these different types of behaviour can be easily seen
by means of simple probabilistic arguments.

Indeed, a spherical particle with volume $V$ moving at a speed of order one
meets one scatterer of volume $v$ according to a Poisson process with average
$V^{-\frac{2}{3}}.$ 
In each collision
with one scatterer the volume of the scatterer is modified to
\begin{equation}
V_{n+1}=V_{n}+v_{n}\,, \label{E0}%
\end{equation}
where the values of $v_{n}$ are chosen independently from each other according
to the density $G\left(  v\right)  .$ On the other hand, the encounters can be
expected to take place at times $t_{n}$ which scale approximately as
\begin{equation}
t_{n+1}=t_{n}+\tau_{n}\approx t_{n}+\frac{1}{\left(  V_{n}\right)  ^{\frac
{2}{3}}}. \label{E1}
\end{equation}

If $\sigma<2$ the average of the variables $v_{n}$ is infinite. However, the 
law of large numbers for the sum of independent random variables
whose distribution behaves like a power law (cf. \cite{Fe}) yields the following
asymptotics for the average of $V_{n+1}:$%
\[
\left\langle V_{n}\right\rangle \approx  n^{\frac{1}{\sigma-1}}\ \ \text{as\ \ }n\rightarrow\infty.
\]
Then (\ref{E1}) yields
\[
t_{n}\approx  n^{1-\frac{2}{3\left(  \sigma-1\right)  }%
}=  n^{\frac{3\sigma-5}{3\left(  \sigma-1\right)  }}=\left[
  n^{\frac{1}{\left(  \sigma-1\right)  }}\right]  ^{\frac
{1}{\mu}}\ \ \text{as\ \ }n\rightarrow\infty
\]
and we obtain
\[
\left\langle V_{n}\right\rangle \approx\left(  t_{n}\right)  ^{\mu}\text{ as
}n\rightarrow\infty\ \ \text{if }\sigma<2.
\]

On the other hand, if $\sigma>2,$ the average of the variables $v_{n}$ is
finite. Then, the usual law of large numbers implies (cf. (\ref{E0})):%
\[
\left\langle V_{n}\right\rangle \sim n\ \ \text{as\ \ }n\rightarrow\infty
\]
which combined with (\ref{E1}) gives
\[
t_{n}\sim n ^{\frac{1}{3}}\ \ \ \text{as\ \ }n\rightarrow\infty\,,
\]
whence
\[
\left\langle V_{n}\right\rangle \sim\left(  t_{n}\right)  ^{3}%
\ \ \text{as\ \ }n\rightarrow\infty\ \ \text{if }\sigma>2.
\]

A similar argument allows to prove that when
$\sigma<\frac{5}{3}$ the collapse of the particles takes place in arbitrarily
short times. Indeed, we still have
\[
V_{n}\sim  n^{\frac{1}{\sigma-1}}\ \ \text{as\ \ }
n\rightarrow\infty.
\]
We expect to have $\tau_{n}$ of order $\frac{1}{\left(  V_{n}\right)
^{\frac{2}{3}}}\sim\frac{1}{n^{\frac{2}{3\left(
\sigma-1\right)  }}},$ whence
\[
t_{n+1}\sim t_{n}+\frac{1}{n^{\frac{2}{3\left(
\sigma-1\right)  }}}%
\]
and for $1<\sigma<\frac{5}{3}$ this gives a convergent series and indicates
the onset of an infinitely large particle in finite time. 

The previous argument is very rough and it might only capture
the power laws describing the average growth of the particles. The description
of the probability density yielding this growth requires a more sophisticated
analysis. This will be given in this paper for $\sigma<2.$ In the case
$\sigma>2,$ although it can be proved with some generality that the particle
satisfies $V\sim C_{0}t^{3},$ the description of the probability distribution
yielding the asymptotics $f\left(  V,t\right)  $ depends strongly on the
values of $\sigma$ and several critical exponents in which the behaviour of
the distribution changes arise. The case $\sigma>2$ has been partially treated in \cite{NoV} 
where it has been proved that for $\sigma>\frac 8 3$ the volume of the tagged particle behaves as $V\sim C_{0}t^{3}$ neglecting the second  order corrections (cf. Theorem 4.3 in \cite{NoV}).
In this paper we will only consider the case $\sigma<2$.

Given the linearity of the problem, the main technical tool that we use is the
classical theory of semigroups, in particular the theory of Markov semigroups,
which is particularly well suited to study the type of equations considered in
this paper. In particular we use extensively the Hille Yoshida theorem as well
as some continuity results for semigroups in terms of their generators
(Trotter-Kurtz theorem). In order to apply semigroup theory the main
difficulty is to solve the resolvent equation $h-\lambda\Omega h=g$ for a
suitable operator $\Omega.$ In the particular case considered in this paper,
we need to solve the resolvent equation in a suitable domain of functions
which are continuous in $\left[  0,\infty\right]  .$ The proof of the
continuity of the solutions at the origin of the resolvent equation is the most technical
part of the paper and we  solve this question by  a
detailed analysis of the fundamental solution associated to the mentioned
nonlocal equation. 
After proving this result, it is relatively simple to prove, using standard tools
of semigroup theory, the existence, uniqueness and stability of self-similar
solutions for  $\frac{5}{3}<\sigma<2.$

The plan of the paper is the following. In Section \ref{sec:setting} we define precisely the setting for the problem of
the long time asymptotics of the solutions of \eqref{eq:lin:coag} with $G(v)$ descreasing as a power law.
We also state the main results which are a well-posedness result for the evolution equation \eqref{eq:lin:coag} (Theorem \ref{th:wellpos}) and
the existence, uniqueness and stability of a self-similar profile (Theorems \ref{th:existencesssol} and \ref{th:stabilityssprof}).

In Section \ref{sec:functan} we present our strategy which relies on the theory of Markov semigroups.  
We look for a solution of  \eqref{eq:lin:coag} 
and for a solution of the approximating equation with $G(v)$ replaced by $v^{-\sigma}$. 
To obtain these solutions in terms of semigroups, we construct and analyze the operators associated to the corresponding adjoint equation. 
The main issue is to show that the closures of the operators are Markov generators. This is proved in Section \ref{Ss.omegainfty}. Due to the difficulties of proving that the solutions of the resolvent equations are continuous at the origin this is the most technical section of the paper. 

In Section  \ref{sec:proofs} we prove the existence of a self-similar profile using a fixed-point argument. Thanks to
the Hille-Yosida theorem, we show that the operators constructed in the previous section are generators of contractive Markov 
semigroups in the space of functions $C([0,\infty])$. This will enable us to use the Trotter-Kurtz theorem, to prove the stability result.

\section{Setting and main results}\label{sec:setting}
				
Our starting point is \eqref{eq:lin:coag}, i.e.
\begin{align}\label{eq:lin:coag0bis}
\partial_t f(V,t)&=\, \int_0^{V} \hspace{-2.5mm} dv \,G(v)((V-v)^{\frac{1}{3}}+v^{\frac{1}{3}})^2 f(V-v,t)-\int_0^{\infty}
\hspace{-2.5mm} dv \,G(v)(V^{\frac{1}{3}}+v^{\frac{1}{3}})^2 f(V,t),\\
f(V,0)&=f_0(V)\in \mathcal{M}_{+}([0,\infty])
\end{align}
which is defined for $G(v)\in L^{1}((0,\infty))$. From now on we will denote by $\mathcal{M}([0,\infty])$ the space of signed Radon measures defined on $[0,\infty]$ and by $\mathcal{M}_{+}([0,\infty])$ the space of nonnegative Radon measures. We set $G_{\infty}(v):=v^{-\sigma}$.
Since $G(v)\sim G_{\infty}(v)$ as $v\to\infty$, it is natural to approximate \eqref{eq:lin:coag0bis} by
\begin{align}\label{eq:lin:coag1}
\partial_t f(V,t)&=\, \int_0^{V} \hspace{-2.5mm} dv \,v^{-\sigma}((V-v)^{\frac{1}{3}}+v^{\frac{1}{3}})^2 f(V-v,t)
-\int_0^{\infty}\hspace{-2.5mm} dv \,v^{-\sigma}(V^{\frac{1}{3}}+v^{\frac{1}{3}})^2 f(V,t),\\
f(V,0)&=f_0(V)\in \mathcal{M}_{+}([0,\infty]).
\end{align}
Note that the integrals on the right hand side of \eqref{eq:lin:coag1} are not well-defined in general
and should be understood in a  suitable way which will be made precise later.

For fixed $M>0$ we introduce the space
\begin{equation}\label{eq:defXM}
X_{M}=\bigl\{\varphi\in C([0,\infty])\;\big|\; \varphi(x)=\text{const.\,}\text{ if }x\geq M\bigr\}.
\end{equation}
Note also that $X_{M_{1}}\subset X_{M_{2}}$ for $M_{1}\leq M_{2}$. 
Moreover, we will denote as
\begin{equation}\label{eq:defX}
\mathcal{X}=\Bigl(\bigcup_{M>0} X_{M}\Bigr).
\end{equation}
We set $C^{\infty}([0,\infty]):=\{\psi\in C^{\infty}([0,\infty))\;|\; \exists \lim_{v\to\infty}\del_{v}^{k}\psi(v) \forall k\in\N_{0}\}$. 
We now  introduce the concept of weak solutions for equation \eqref{eq:lin:coag}. 
\begin{definition}\label{def:fweaksol}
We say that $f\in C([0,T];\mathcal{M}_{+}([0,\infty]))$ is a weak solution of \eqref{eq:lin:coag0bis} or \eqref{eq:lin:coag1} if, 
for any test function $\varphi\in C^1([0,T]; C^{\infty}([0,\infty])\cap \mathcal{X})$, $f$ satisfies
\begin{align}\label{eq:fweak}
&\int_{[0,\infty]} f(V,T)\varphi(V,T) dV - \int_{[0,\infty]} f(V,0)\varphi(V,0) dV -\int_{0}^{T}\int_{[0,\infty]} f(V,t)\partial_{t}\varphi(V,t) dVdt\nonumber\\&
=\int_{0}^{T}\int_{[0,\infty]}\int_{[0,\infty]} G_{\star}(v) (V^{\frac{1}{3}}+v^{\frac{1}{3}})^2 f(V,t)\big(\varphi(V+v,t)-\varphi(V,t) \big)  dvdVdt, 
\end{align}
where $G_{\star}(v)$ denotes either $G(v)$ or $G_{\infty}(v)$ respectively.
\end{definition}

\begin{remark}
We notice that, 
 it is straightforward to show that
$$V\mapsto \int_{[0,\infty]} G_{\star}(v) (V^{\frac{1}{3}}+v^{\frac{1}{3}})^2 \big(\varphi(V+v,t)-\varphi(V,t) \big)  dv\in C([0,\infty])\cap \mathcal{X}.$$
\end{remark}

We will prove the following well-posedness result whose proof is contained in Section \ref{ss:wellposproof}.
\begin{theorem}[Well-posedness]\label{th:wellpos}
For any $f_0(\cdot)\in \mathcal{M}_{+}([0,\infty])$ there exists a unique solution of \eqref{eq:lin:coag0bis} and \eqref{eq:lin:coag1} 
respectively in the sense of Definition \ref{def:fweaksol}. Moreover, if $f_0(\{\infty\})=0$ then $f(\{\infty\},t)=0$ for any $t\geq 0$.
\end{theorem}

Our next goal is to prove the convergence towards a self-similar profile with the scaling
\begin{equation*}
 V=t^{\mu}\quad \text{and}\quad \mu=(\sigma-5/3)^{-1}.
\end{equation*}
Formally, assuming enough regularity, we can write the equation for self-similar profiles associated to \eqref{eq:lin:coag1}, with 
initial data $f_0(V)=\delta(V)$. Indeed, making the  ansatz 
$ f(V,t)=\frac{1}{t^{\mu}}F\left(\frac{V}{t^{\mu}}\right)$
and plugging this into the equation \eqref{eq:lin:coag1} we formally obtain 
\begin{equation}\label{eq:self:sim}
 -\mu F(\xi)-\mu\xi\del_{\xi}F(\xi)=\int_{0}^{\xi}\eta^{-\sigma}\bigl((\xi-\eta)^{1/3}+\eta^{1/3}\bigr)^{2}
 F(\xi-\eta)\deta-\int_{0}^{\infty}\eta^{-\sigma}(\xi^{1/3}+\eta^{1/3})^{2}\deta F(\xi)\,.
\end{equation}
which can be rewritten as
\begin{equation}\label{eq:self:sim:2}
 -\mu\del_{\xi}\bigl(\xi F(\xi)\bigr)=-\del_{\xi}\biggl(\int_{0}^{\xi}\int_{\xi-x}^{\infty} F(x)\eta^{-\sigma}(x^{1/3}+\eta^{1/3})^2\deta\dx\biggr).
\end{equation}
Under the assumption that $\lim_{\xi \to \infty} \xi F(\xi)=0$ we can integrate in $\xi$, then multiply by a test function $\varphi$ and integrate once more to obtain
\begin{align}\label{eq:weakss2}
 \mu\int_{[0,\infty]} \del_{\xi} \varphi(\xi)\xi F(\xi)\dxi  =\int_{[0,\infty]}\int_{[0,\infty]} F(\xi)\eta^{-\sigma}(\xi^{1/3}+\eta^{1/3})^2\big[\varphi(\xi+\eta)-\varphi(\xi)\big]\deta\dxi
\end{align} 
for any $\varphi\in C^{\infty}([0,\infty])\cap \mathcal{X}$.

In Section \ref{Ss.ex} we will prove the existence of self-similar solutions as stated in the  following theorem.

\begin{theorem}[Existence and uniqueness of self-similar profiles]\label{th:existencesssol}
For any $\sigma \in (5/3,2)$, 
there exists a unique self-similar profile $F\in\mathcal{M}_{+}([0,\infty])$ 
 that satisfies \eqref{eq:weakss2} for any test function $\varphi\in C^{\infty}([0,\infty])\cap \mathcal{X}$. Moreover, 
$F\in C^{\infty}(0,\infty)$ and it satisfies for any $\gamma >0$ and for any $\beta\in (0,\sigma-\frac 5 3)$
\begin{align}
&\int_0^\infty (\xi^{\beta}+\xi^{-\gamma})F(\xi)d\xi<\infty, \\&
\int_{[0,\infty]}F(\xi) d\xi =1.
\end{align}
\end{theorem}

Moreover, we will show in Section \ref{Ss.stability} that the self-similar solutions are asymptotically stable.

\begin{theorem}[Stability of self-similar profiles]\label{th:stabilityssprof}
Let $f_0(\cdot)\in \mathcal{M}_{+}([0,\infty])$, with $f_0(\{\infty\})=0$. Let $f(V,t)$ be the solution of \eqref{eq:lin:coag0bis} obtained in Theorem \ref{th:wellpos}. Then, 
\begin{equation}\label{eq:ssconv}
 {t^{\mu}}f(t^{\mu}w,t)\rightharpoonup F(w)\quad \text{as} \quad t\to \infty,
\end{equation}
 in the weak topology of $\mathcal{M}_{+}([0,\infty])$.
 \end{theorem} 
Notice that the weak topology of $\mathcal{M}_{+}([0,\infty])$ is defined by duality using that $\mathcal{M}([0,\infty])=\big(C([0,\infty])\big)^{\ast}$.

 The strategy of the respective proofs relies on the analysis of the adjoint equations of \eqref{eq:lin:coag0bis}, \eqref{eq:lin:coag1}
 which read as
\begin{align}\label{eq:dual:lincoag}
\del_{t}\varphi(V,t)&=\int_{0}^{\infty}G_{\star}(v)\bigl(V^{1/3}+v^{1/3}\bigr)^2\bigl[\varphi(V+v,t)-\varphi(V,t)\bigr]\dv, \nonumber \\
\varphi(\cdot,0)&=\varphi_0(\cdot) \in C([0,\infty]),
\end{align}
where $G_{\star}=G$ in the case \eqref{eq:lin:coag0bis} and $G_{\star}=G_{\infty}$ for \eqref{eq:lin:coag1}.

Then, formally, the following identity holds: 
\begin{equation}\label{duality_volumes}
\int_{ [0,\infty]} \varphi_0(V) f(V,t)\,dV =\int_{ [0,\infty]} \varphi(V,t) f_0(V)\, dV \qquad \forall \, t\geq 0.
\end{equation}
In the existence result, cf.  Proposition \ref{prop:sgrgen}, we will in fact use \eqref{duality_volumes} to define $f$ and then show that this function is indeed
a weak solution of  \eqref{eq:lin:coag0bis} or \eqref{eq:lin:coag1} in the sense of Definition \ref{def:fweaksol}.

The right hand side of \eqref{eq:dual:lincoag} suggests to study the semigroups generated by the following operators
\begin{equation} \label{eq:Omega}
 \mathcal{K}\varphi(V)=\int_{0}^{\infty}G(v)(V^{1/3}+v^{1/3})^2\bigl[\varphi(V+v)-\varphi(V)\bigr]\dv,
  \end{equation}  
\begin{equation}   \label{eq:infOmega}
 \mathcal{K}_{\infty}\varphi(V)=\int_{0}^{\infty}G_{\infty}(v)(V^{1/3}+v^{1/3})^2\bigl[\varphi(V+v)-\varphi(V)\bigr]\dv.
 \end{equation}
We will construct in the next subsection operators $ \Omega$ and $\Omega_{\infty}$, using \eqref{eq:Omega} and \eqref{eq:infOmega},
which are the generators of contractive semigroups in the space of functions $C([0,\infty])$. 
This will enable us to use the results from semigroup theory to prove our main results Theorem \ref{th:wellpos}, Theorem \ref{th:existencesssol} and Theorem 
\ref{th:stabilityssprof}.

We already mention here that in the proof of Theorem \ref{th:stabilityssprof}, through a rescaling argument, we are led to consider more general operators than $\mathcal{K}$, which are
\begin{equation} \label{eq:TOmega}
 \mathcal{K}_{T}\varphi(V)=\int_{0}^{\infty}G_{T}(v)(V^{1/3}+v^{1/3})^2\bigl[\varphi(V+v)-\varphi(V)\bigr]\dv\\ 
\end{equation}
with $G_{T}(v)\vcc=T^{\sigma \mu}G(T^{\mu v})$. Note that $ \mathcal{K}_{1}= \mathcal{K}$.

\bigskip

\section{\texorpdfstring{Properties of the operators $\Omega_{T}$ and $\Omega_{\infty}$ and the corresponding semigroups
}{Properties of the operators Omega\_{T} and Omega\_infty and the corresponding semigroups}}
\label{sec:functan}

In this section we will prove the main technical result of this paper, which is that the closures of operators $\Omega_T$ and $\Omega_{\infty}$ (see~\cref{def:opOmegaT,def:opOmegainf} below), which are related to 
$\mathcal{K}_T$ and $\mathcal{K}_{\infty}$ in \eqref{eq:infOmega} and  \eqref{eq:TOmega}, respectively, are indeed Markov generators. 

We observe that, due to the integrability assumption \eqref{ass:G}, $G_{T}(\cdot) \in L^{1}(0,\infty)$.
We have  that
\begin{equation*}
 \overline{\mathcal{X}}^{\norm{\cdot}_{\infty}}=C([0,\infty]),
\end{equation*}
where $\mathcal{X}$ is defined as in \eqref{eq:defX}. 
For $T<\infty$ we set
\begin{equation}\label{def:DOmT}
 \D(\Omega_{T})=\mathcal{X}
\end{equation}
as domain for $\Omega_{T}$. We define $\Omega_{T}\varphi$ as 
\begin{align}\label{def:opOmegaT}
&\Omega_{T}\varphi(V)=\mathcal{K}_{T}\varphi(V) \quad \forall\; V\in [0,\infty) \quad 
\\&
\Omega_{T}\varphi(\infty)=0,
\end{align}
for any $\varphi\in  \D(\Omega_{T})$.

On the other hand, for $T=\infty$ we set 
\begin{equation}\label{def:DOmInf}
 \D(\Omega_{\infty})=\mathcal{X} \cap \Bigl(\bigcap_{V_{\ast}  >0} W^{1,\infty}([V_{*},\infty])\Bigr)\cap \Bigl(\left\{ \exists\: \lim_{V\to  0^{+}} \Omega_{\infty}\varphi (V)\right\}\Bigr),
\end{equation}
where $W^{1,\infty}$ denotes the standard Sobolev space with the $L^{\infty}$-norm.
\medskip
We observe that the right hand side of  \eqref{eq:infOmega}  is well defined for $V>0$ and for functions $\varphi \in \mathcal{X}\cap \Bigl(\underset{{V_{\ast}>0}}{\bigcap} W^{1,\infty}([V_{*},\infty])\Bigr)$. 
We define $\Omega_{\infty}\varphi (V)$,   
as 
\begin{align}\label{def:opOmegainf}
&\Omega_{\infty}\varphi (V)=\mathcal{K}_{\infty}\varphi (V) \qquad \text{for}\quad V>0,\\&
\Omega_{\infty}\varphi (0)= \lim_{V\to 0^{+}} \mathcal{K}_{\infty}\varphi (V),\\&
\Omega_{\infty}\varphi(\infty)=0,
\end{align} 
for any $\varphi \in  \D(\Omega_{\infty})$.
 Then  $\Omega_{\infty}\varphi\in C([0,\infty])$.
 
 Note that the definitions  \eqref{def:opOmegaT} and \eqref{def:opOmegainf} yield operators $\Omega_{T}, \, \Omega_{\infty}$ that bring the respective domains to the space of functions $C([0,\infty])$. We observe that both domains are contained in $\mathcal{X}$ which implies that the value at infinity of both operators acting on $\varphi$ is zero.

\medskip

For further use we recall here the definitions of Markov pregenerator (cf. Definition~2.1 in~\cite{Li}) and Markov generator (cf. Definition~2.7 in~\cite{Li}).
\begin{definition}\label{Def:pregenerator}
  A linear operator $A$ on $C([0,\infty])$ with domain $\D(\Omega)$ is said to be a \emph{Markov pregenerator} if it satisfies the following conditions:
 \begin{enumerate}[label=(\alph*)]
  \item \label{It:pregenerator:1} $1\in \D(A)$ and $\Omega 1=0$.
  \item \label{It:pregenerator:2} $\D(A)$ is dense in $C([0,\infty])$.
  \item \label{It:pregenerator:3} If $\varphi \in\D(A)$, $\lambda\geq 0$, and $\varphi-\lambda A \varphi=g$, then 
  \begin{equation*}\label{eq:maxp}
   \min_{V\in [0,\infty]}\varphi(V)\geq \min_{V\in [0,\infty]}g(V).
  \end{equation*}
 \end{enumerate}
\end{definition}

\begin{remark}\label{Rem:apriori}
 Applying \ref{It:pregenerator:3} to both $\varphi$ and $-\varphi$ one sees that a Markov pregenerator has the following property: if $\varphi\in\D(A)$, $\lambda\geq 0$, and $\varphi-\lambda A \varphi=g$, then $\norm{\varphi}\leq \norm{g}$. 
\end{remark}

\begin{definition}\label{Def:generator}
 A \emph{Markov generator} is a closed Markov pregenerator $A$ which satisfies
 \begin{equation*}
  \Ra(I-\lambda A)=C([0,\infty])
 \end{equation*}
for all sufficiently small positive $\lambda$.
\end{definition}

We can prove the following.
\begin{lemma}
The operators  $\Omega_{\infty}$ and $ \Omega_{T}$ defined as in \eqref{eq:infOmega}, \eqref{eq:TOmega} with domains $ \D(\Omega_{T})$ and  $\D(\Omega_{\infty})$ given by \eqref{def:DOmT}, \eqref{def:DOmInf} are Markov pregenerators in the sense of Definition~\ref{Def:pregenerator}.
 \end{lemma}

\begin{proof} 
 It is straightforward to verify  \ref{It:pregenerator:1} for the operators  $ \Omega_{T}$ and $\Omega_{\infty}$ and \ref{It:pregenerator:2} for $ \Omega_{T}$. To verify \ref{It:pregenerator:2} for $ \Omega_{\infty}$ we have to show that functions $\varphi\in \mathcal{X}$ which are globally Lipschitz  are dense in $C([0,\infty])$ and they belong to $ \D(\Omega_{\infty})$. The only nontrivial point to be proved is the existence of  $\lim_{V\to  0^{+}} \Omega_{\infty}\varphi (V)$, which follows from Lebesgue's dominated convergence theorem.
In order to prove \ref{It:pregenerator:3} we use that for $\varphi\in\mathcal{X}$ we have $\min_{V\in [0,\infty]} \varphi (V)=\varphi (V_0)$ for some $V_0\in [0,\infty)$. In the case of $\Omega_{T}$ then, using \eqref{eq:TOmega} we have $\Omega_{T}\varphi (V_0) \geq 0$. Then  
\begin{equation*}
 \min_{V\in[0,\infty]}g (V)\leq g (V_0)\leq \varphi (V_0)-\lambda \Omega_{T} \varphi(V_0)\leq \varphi(V_0)=\min_{[0,\infty]}\varphi (V).
\end{equation*}
Then the result follows for $A=\Omega_{T}$.  In the case  $A=\Omega_{\infty}$ the result follows similarly if $V_0 >0$. 
When $V_0=0$ we consider a sequence $V_n \to 0$, as $n\to \infty$ so that $\varphi (V_n)\to \varphi (0)$ as $n\to \infty$. Then there exists a sequence $W_n\geq V_n$ such that $\min_{V\geq V_n}\varphi (V_n)=\varphi (W_n)$. Taking a subsequence of $W_n$ if needed (which will be denoted still as $W_n$) we obtain that $W_n\to W_{\ast}\in [0,\infty]$. Moreover $\varphi (W_n)\to \varphi (W_{\ast})= \varphi (0)$ as $n\to \infty$. We have now two possibilities. 
If $W_{\ast} >0$ then $\Omega_{\infty} \varphi (W_{\ast})\geq 0$. Then, arguing as in the case of the operator $\Omega_T$ we obtain  \ref{It:pregenerator:3}. In the other case, if $W_{\ast} =0$, we use that $\Omega_{\infty} \varphi (W_{\ast})\geq 0$ by construction. Then, 
\begin{equation*}
 \min_{V\in [0,\infty]} g(V)\leq g(W_n)= \varphi(W_n)-\lambda \Omega_{\infty} \varphi (W_n)\leq \varphi (W_n).
\end{equation*}
Taking now the limit $n\to \infty$ we obtain
\begin{equation*}
 \min_{V\in [0,\infty]} g(V) \leq \varphi (0)=\min_{V\in [0,\infty]} \varphi(V).
\end{equation*}
\end{proof}

We refer to \cite{Li} (cf. Section 2 in Chapter 1) for the following result.
\begin{proposition}\label{prop:Markgen}
If for every $\lambda > 0$ and for any $g\in \mathcal{\tilde{D}}$, with $\mathcal{\tilde{D}}$ dense in $C([0, \infty])$, there exists a solution  $\varphi\in \D(A)$ of $(I-\lambda A)\varphi =g$  then $A$ has a closure $\bar{A}$ which is a Markov generator.
\end{proposition}

Our goal is to prove the following result. 
\begin{theorem}\label{lem:Markovgen}
The operators $\bar {\Omega}_{T}$, $\bar{\Omega}_{\infty}$ are Markov generators in the sense of Definition~\ref{Def:generator}.
\end{theorem}

Proposition \ref{prop:Markgen} implies that in order to prove the theorem above it is sufficient to show that the  equations
$(\mbox{Id}-\lambda \Omega_T)\varphi=g$ and $(\mbox{Id}-\lambda \Omega_{\infty})\varphi=g$ 
can be solved for a given $g$ in a dense subset $\tilde{\mathcal{D}}$ of $C([0,\infty])$ and $\lambda>0$.

\begin{theorem}\label{Prop:OmT}
 Let $\mathcal{X}$ be as \eqref{eq:defX}. For every $\lambda>0$ and for every $g\in \mathcal{X}\cap C^{\infty}([0,\infty])$ there exists  $\varphi\in \D(\Omega_{T})$ 
  such that it holds
 \begin{equation}\label{eq:OT}
  (I-\lambda\Omega_{T})\varphi=g \quad \text{on }[0,\infty].
 \end{equation}
\end{theorem}

\begin{theorem}\label{Prop:OmINF}
 Let $\mathcal{X}$ be as \eqref{eq:defX}. For every $\lambda>0$ and for every $g\in \mathcal{X}\cap C^{\infty}([0,\infty])$  there exists  $\varphi\in \D(\Omega_{\infty})$ 
  such that it holds
 \begin{equation}\label{eq:OINF}
  (I-\lambda\Omega_{\infty})\varphi=g \quad \text{on }[0,\infty].
 \end{equation}
\end{theorem}
The proof of Theorem \ref{Prop:OmINF} is actually the hard part of this paper. Section \ref{Ss.omegainfty} will be devoted to this proof.

\subsection{Proof of Theorem \ref{Prop:OmT}}\label{Ss.omegaT}
The operators $\Omega_T$ are bounded Markov pregenerators in the Banach spaces $X_M$ for any $M<\infty$. Therefore they are Markov generators (cf. \cite{Li}, Proposition 2.8, a)). Hence, \eqref{eq:OT} admits a solution $\varphi\in \D(\Omega_{T})$ for every  $g\in \mathcal{X}$. 

\subsection{Proof of Theorem \ref{Prop:OmINF}}\label{Ss.omegainfty}

\subsubsection{Strategy of the proof}

We first prove in Lemma \ref{Lemma1:OmINF}, via a fixed point argument, that a solution $\varphi$ to $(\mbox{Id}-\lambda \mathcal{K_{\infty}} )\varphi =g$ exists on $(0,\infty)$. 
The main work to prove Theorem \ref{Prop:OmINF} is, however, to prove that this solution $\varphi$ has a limit as $V \to 0$.
Towards that aim we fix some $\bar V>0$ and  write $\varphi$ for $V<\bar{V}$ as 
\begin{equation}\label{eq:complsolffi}
\varphi(V)=\frac{1}{\lambda}\int_0^{\bar{V}} G(V,v) (g(v)-\varphi(v))\dv+ \int_{\bar{V}}^{\infty} K(V,\eta;\bar{V})\varphi(\eta) d\eta=\varphi_1+\varphi_2.
\end{equation}
Here $\varphi_1$ solves 
\begin{align}
 -\mathcal{K}_{\infty} \varphi_1 &=g-\varphi \qquad \mbox{ for } V< \bar V\label{eq:p1}\\
 \varphi_1&=0 \qquad \qquad \mbox{ for } V > \bar V,\nonumber
\end{align}
while $\varphi_2$ solves
\begin{align}
 -\mathcal{K}_{\infty} \varphi_2 &=0 \qquad \mbox{ for } V< \bar V\label{eq:p2}\\
 \varphi_2&=\varphi \qquad \mbox{ for } V > \bar V.\nonumber
\end{align}
The existence and uniqueness of a solution to \eqref{eq:p1} is given in Theorem \ref{th:genbdpb}. Here the key idea is to construct a fundamental solution (see Proposition 
\ref{prop:existG}) and prove further properties of it in Lemma \ref{lem:sol1}. Problem \eqref{eq:p2} is solved in Lemma \ref{lem:sol2}.
In order to conclude that $\lim_{V \to 0}\varphi(V)$ exists, we need to establish continuity properties of $G$, which is the content of Section \ref{Sss.continuity}.

\subsubsection{\texorpdfstring{Existence of solutions of $   (I-\lambda \mathcal{K}_{\infty})\varphi=g$ for a smooth $g$}{Existence of solutions of (I-lambda K\_infty)varphi=g for a smooth g}}\label{Sss.regularity}

\begin{lemma}\label{Lemma1:OmINF}
 Let $\mathcal{X}$ be as in \eqref{eq:defX}. For every $\lambda>0$ and for every $g\in \mathcal{X}\cap C^{\infty}([0,\infty])$  there exists  $\varphi\in C(0,\infty]\cap \mathcal{X}$  such that
 \begin{equation}\label{eq:inteqOminf}
\varphi(V)-\lambda  \mathcal{K}_{\infty}\varphi(V)=g(V),\; \text{on}\;(0,\infty),
 \end{equation}
 where $ \mathcal{K}_{\infty}$ is as in \eqref{eq:infOmega}. 
 Moreover, $\varphi\in W^{1,\infty}([V_{*},\infty])$ for every $V_{\ast}>0$ and 
 \begin{equation}\label{eq:maxpffi}
  \norm{\varphi}_{L^{\infty}(0,\infty]}\leq \norm{g}_{L^{\infty}(0,\infty]}.
 \end{equation}
\end{lemma}

\begin{proof} 

We first consider a regularised version of the problem. More precisely, for every sufficiently small $\eps>0$,
we define the corresponding regularized operator $\mathcal{K}_{\infty}^{\eps}$:
 \begin{equation}\label{eq:OmInfEps}
 \mathcal{K}_{\infty}^{\eps}\varphi(V)\vcc=\int_{0}^{\infty}(v+\eps)^{-\sigma}(V^{1/3}+v^{1/3})^2\bigl[\varphi(V+v)-\varphi(V)\bigr]\dv.
 \end{equation}
We now look for a solution $\varphi^{\eps}$ of the equation
 \begin{equation}\label{eq:OI:eps:1}
  \varphi^{\eps}(V)-\lambda \mathcal{K}_{\infty}^{\eps}\varphi^{\eps}(V)=g(V),\quad V>0\,,
 \end{equation}
 with $g\in \mathcal{X}\cap C^{\infty}([0,\infty])$ and $\varphi^{\eps}\in  \mathcal{X}\cap \Bigl(\bigcap_{V_{\ast}  >0} W^{1,\infty}([V_{*},\infty])\Bigr)$. To this end we notice that equation~\eqref{eq:OI:eps:1} can be rewritten as the following fixed-point problem:
 \begin{multline}\label{eq:OI:eps:2}
  \varphi^{\eps}(V)=\biggl(1+\lambda\int_{0}^{\infty}(v+\eps)^{-\sigma}(V^{1/3}+v^{1/3})^2\dv\biggr)^{-1}\\*
  \times\biggl[g(V)+\lambda\int_{0}^{\infty}(v+\eps)^{-\sigma}(V^{1/3}+v^{1/3})^2 \varphi^{\eps}(V+v)\dv\biggr].
 \end{multline}
 To solve this, we note that we are looking for a solution in $\mathcal{X}$. Since $g\in \mathcal{X},$ we have $g(V)=g(M)$ for all $V\geq M$, for some $M>0$. Thus, the structure of~\eqref{eq:OI:eps:1} immediately implies that $\varphi^{\eps}(V)=g(M)$ for all $V\geq M$. It then follows that each solution of \eqref{eq:OI:eps:2} with $\varphi^{\eps}\in \mathcal{X}\cap C((0,\infty])$ has to satisfy $\varphi^{\eps}(V)=\varphi^{\eps}(M)$ for any $V\geq M$. Therefore, it suffices to solve~\eqref{eq:OI:eps:2} on $(0,M]$ together with the boundary condition $\varphi^{\eps}(M)=g(M)$. This can be made  using Banach's contractive fixed-point argument. More precisely  \eqref{eq:OI:eps:2} can be rewritten as
  \begin{multline*}
  \varphi^{\eps}(V)=\biggl(1+\lambda\int_{0}^{\infty}(v+\eps)^{-\sigma}(V^{1/3}+v^{1/3})^2\dv\biggr)^{-1}\\*
  \times\biggl[g(V)+\lambda\int_{V}^{\infty}(\xi-V+\eps)^{-\sigma}(V^{1/3}+(\xi-V)^{1/3})^2 \varphi^{\eps}(\xi)\dv\biggr]
 \end{multline*}
 which is a linear Volterra equation that can be solved in $C[M-\delta,\infty]$ for a sufficiently small $\delta>0$. The solutions can be estimated for $V\in (0,M]$ using a Gronwall argument and then we obtain $\varphi^{\eps}\in C((0,\infty])$. 
 
 To obtain a solution of the original problem \eqref{eq:inteqOminf}, we now take the limit $\eps\to 0$ in \eqref{eq:OI:eps:1}. In order to obtain compactness for the function $\varphi^{\eps}$ we derive uniform estimates for the derivatives of $\varphi^{\eps}$. 
Differentiating~\eqref{eq:OI:eps:1}, we get
 \begin{equation}\label{eq:OI:eps:3}
  \frac{\del \varphi^{\eps}}{\del V}-\lambda \mathcal{K}_{\infty}^{\eps}\left(\frac{\del \varphi^{\eps}}{\del V}\right)-h^{\eps}(V,\lambda)=\frac{\del g}{\del V},
 \end{equation}
where
 \begin{align*}
  h^{\eps}(V,\lambda)&=\lambda\int_{0}^{\rho}(v+\eps)^{-\sigma}\Bigl[\frac{2}{3}V^{-1/3}+\frac{2}{3}V^{-2/3}v^{1/3}\Bigr]\bigl(\varphi^{\eps}(V+v)-\varphi^{\eps}(V)\bigr)\dv\\&
 \quad  +\lambda\int_{\rho}^{\infty}(v+\eps)^{-\sigma}\Bigl[\frac{2}{3}V^{-1/3}+\frac{2}{3}V^{-2/3}v^{1/3}\Bigr]\bigl(\varphi^{\eps}(V+v)-\varphi^{\eps}(V)\bigr)\dv\\&
 =: h_1^{\eps}(V,\lambda)+h_2^{\eps}(V,\lambda).
 \end{align*}
Here  $ \rho$ is a positive constant, which is independent of $\eps$ and will be chosen  later.

 Using that 
 $$\abs{\varphi^{\eps}(V+v)-\varphi^{\eps}(V)}\leq  \norm*{\frac{\del \varphi^{\eps}}{\del V}}_{L^{\infty}(V,V+\rho)}v,$$
  as well as  $(v+\eps)^{-\sigma}\leq v^{-\sigma}$ and $\sigma<2$, we obtain
 \begin{align}\label{eq:OI:eps:4}
  \abs*{h_1^{\eps}(V,\lambda)}&\leq \frac{2}{3}\lambda\norm{(\varphi^{\eps})'}_{L^{\infty}(V,V+\rho)}\int_{0}^{\rho}v^{-\sigma}(V^{-1/3}v+V^{-2/3}v^{4/3})\dv\nonumber \\&
  \leq C_{1}(\sigma)\lambda\norm{(\varphi^{\eps})'}_{L^{\infty}(V,V+\rho)}\bigl(\rho^{2-\sigma}V^{-1/3}+\rho^{7/3-\sigma}V^{-2/3}\bigr),
  \end{align}
  where $C_{1}(\sigma)=\frac{4}{3(2-\sigma)}$. 
We now consider $h_2^{\eps}$ and we get
 \begin{align}\label{eq:OI:eps:4bis}
  \abs*{h_2^{\eps}(V,\lambda)}&\leq \frac{4}{3}\lambda\norm{\varphi^{\eps}}_{L^{\infty}(V,\infty)}\int_{\rho}^{\infty}v^{-\sigma}(V^{-1/3}+V^{-2/3}v^{1/3})\dv\nonumber\\&
\leq   C_{2}(\sigma)\lambda\norm{\varphi^{\eps}}_{L^{\infty}(V,\infty)}\bigl(\rho^{-(\sigma-1)}V^{-1/3}+\rho^{-(\sigma-4/3)}V^{-2/3}\bigr),
 \end{align}
 where $C_{2}(\sigma)=\frac{8}{3}(\sigma-\frac{4}{3})^{-1}$.
Using a maximum principle argument, in the same spirit as the one in Remark~\ref{Rem:apriori}, 
 we obtain
 \begin{equation}\label{eq:maxp1}
  \norm{\varphi^{\eps}}_{L^{\infty}[V_{*},\infty]}\leq \norm{g}_{L^{\infty}[V_{*},\infty]}\leq  \norm{g}_{L^{\infty}(0,\infty)} ,
 \end{equation}
 and 
  \begin{equation}\label{eq:maxp2}
   \norm*{\frac{\del \varphi^{\eps}}{\del V}}_{L^{\infty}[V_{*},\infty]}\leq \norm*{\frac{\del g}{\del V}}_{L^{\infty}[V_{*},\infty]}+\norm*{h^{\eps}(\cdot,\lambda)}_{L^{\infty}[V_{*},\infty]},
   \end{equation}
   where $V_{*}>0$ is fixed and independent of $\eps$.
Using  estimates~\eqref{eq:OI:eps:4}-\eqref{eq:maxp1} we  find
 \begin{align*}
  \norm*{\frac{\del \varphi^{\eps}}{\del V}}_{L^{\infty}[V_{*},\infty]}\leq & \norm*{\frac{\del g}{\del V}}_{L^{\infty}[V_{*},\infty]}+C_{1}(\sigma)\lambda  \norm*{\frac{\del \varphi^{\eps}}{\del V}}_{L^{\infty}[V_{*},\infty]}\bigl(\rho^{2-\sigma}V_{*}^{-1/3}+\rho^{7/3-\sigma}V_{*}^{-2/3}\bigr)\nonumber \\&
  +C_{2}(\sigma)\lambda\norm{\varphi^{\eps}}_{L^{\infty}[V_{*},\infty]}\bigl(\rho^{-(\sigma-1)}V_{*}^{-1/3}+\rho^{-(\sigma-4/3)}V_{*}^{-2/3}\bigr),
\nonumber \\&
 \leq \norm*{\frac{\del g}{\del V}}_{L^{\infty}[V_{*},\infty]}+C_{1}(\sigma)\lambda  \norm*{\frac{\del \varphi^{\eps}}{\del V}}_{L^{\infty}[V_{*},\infty]}\bigl(\rho^{2-\sigma}V_{*}^{-1/3}+\rho^{7/3-\sigma}V_{*}^{-2/3}\bigr)\nonumber \\&
  +C_{2}(\sigma)\lambda\norm{g}_{L^{\infty}[V_{*},\infty]}\bigl(\rho^{-(\sigma-1)}V_{*}^{-1/3}+\rho^{-(\sigma-4/3)}V_{*}^{-2/3}\bigr).
 \end{align*}
 We choose now $\rho>0$ independent of $\eps$ such that the following inequality holds:
 \begin{equation*}
  C_{1}(\sigma)\lambda\bigl(\rho^{2-\sigma}V_{*}^{-1/3}+\rho^{7/3-\sigma}V_{*}^{-2/3}\bigr)\leq \frac{1}{2}.
 \end{equation*}
We can then absorb the term on the right hand side, containing the derivative of $\varphi^{\eps}$, into the left hand side. Therefore we have 
\begin{align}\label{eq:OI:eps:der:1}
  \norm*{\frac{\del \varphi^{\eps}}{\del V}}_{L^{\infty}[V_{*},\infty]}\leq &  \,2 \, \norm*{\frac{\del g}{\del V}}_{L^{\infty}[V_{*},\infty]}  +2\, C_{2}(\sigma)\lambda\norm{g}_{L^{\infty}[V_{*},\infty]}\bigl(\rho^{-(\sigma-1)}V_{*}^{-1/3}+\rho^{-(\sigma-4/3)}V_{*}^{-2/3}\bigr).
 \end{align}

Thus, combining \eqref{eq:maxp1} with \eqref{eq:OI:eps:der:1}  we obtain a uniform estimate for $\varphi^{\eps}\in  W^{1,\infty}([V_{*},\infty])$. Then,  using  the  Arzel{\'a}-Ascoli theorem there exists a sub-sequence $\{\eps
_{j}\}$ 
 and a function $\varphi\in W^{1,\infty}([V_{*},\infty])$ such that $\varphi^{\eps_j}\to \varphi$ in $L^{\infty}[V_{*},\infty]$. 
 
 Using \eqref{eq:maxp1} and \eqref{eq:OI:eps:der:1} we can take the limit in \eqref{eq:OI:eps:1} by means of Lebesgue's dominated convergence theorem.
 Hence \eqref{eq:inteqOminf} follows. 
 
 Finally \eqref{eq:maxpffi} follows taking the limit $\eps \to 0$ and $V_{\ast}\to 0$ in \eqref{eq:maxp1}.
 \end{proof}
\bigskip

The main difficulty that remains to prove Theorem \ref{Prop:OmINF} is to show that the function $\varphi$ obtained in Lemma \ref{Lemma1:OmINF} is continuous at $V=0$. Notice that Lemma \ref{Lemma1:OmINF} implies that $\varphi$ is continuous for $V\in (0,\infty]$ but does not guarantee uniform continuity as $V\to 0$. In order to have the desired continuity at $V=0$ we will develop a theory of fundamental solutions for the operator $-\mathcal{K}_{\infty}.$
\bigskip

\subsubsection{\texorpdfstring{On the fundamental solutions for the operator $\mathcal{K}_{\infty}$}{On fundamental solutions for the operator K\_infty}}
\label{Sss.fundamental}
We will study the boundary value problem associated to the operator $\mathcal{K}_{\infty}$
\begin{align}\label{eq:genbdpb}
-\mathcal{K}_{\infty}&u(V)=g(V) \qquad\qquad\text{for}\quad 0< V< \bar{V} 
\\&
u(V)= \Psi(V) \qquad\qquad\text{for}\quad V>\bar{V}, \label{eq:genbdpb1}
\end{align}
with $\bar{V}>0$.

\begin{remark}
We can think of the function $\Psi(V)$ as a boundary value for the problem \eqref{eq:genbdpb}. Notice that $\mathcal{K}_{\infty}u(V)$ depends on values of $u(w)$ for $w>\bar{V}$ (cf.~\eqref{eq:infOmega}).
\end{remark}

\begin{theorem}\label{th:genbdpb}
Let $\bar{V}>0$.  
Suppose that $\Psi \in W^{1,\infty}[\bar{V},\infty)$ and $g\in  \left(\bigcap_{\delta>0}W^{1,\infty}[\delta ,\bar{V}-\delta]\right)$ for any $\bar{V}>0$. 
In addition, suppose that $\|g\|_{\infty}=\|g\|_{L^{\infty}(0,\bar{V})}<\infty$. 
Then there exists a unique solution $u$ of the boundary value problem \eqref{eq:genbdpb}-\eqref{eq:genbdpb1} such that  $u\in \left(\bigcap_{\delta>0}W^{1,\infty}[\delta,\bar{V}-\delta]\right)$ and
\begin{equation}\label{eq:normunu}
\sup_{0<V<\bar{V}} (\bar{V}-V)^{-(\sigma-1)} \vert u(V)\vert <\infty.
\end{equation}
\end{theorem}

\medskip

To prove the theorem above we will use the following result.

\begin{proposition}\label{prop:existG}
For every $V_0>0$ there exists a function $G(V,V_0)$ defined for $V\in (0, \infty)\setminus\{V_0\}$ which is real analytic for $V\in (0,V_0)$, 
$G(V,V_0)=0$ for $ V> V_0$ and satisfies 
 \begin{equation}\label{eq:homogpb1}
-\mathcal{K}_{\infty}G(V,V_0)=0\qquad\quad \text{for}\quad 0< V< V_0,
\end{equation}
as well as the asymptotics
 \begin{equation}\label{eq:homogpb2}
G(V,V_0)=\frac{C_{\ast}}{V_0^{\frac 2 3}} (V_0-V)^{\sigma-2} [1+O((V_0-V)^{\frac{1}{3}})], \quad \text{for}\quad V\to V_0^{-},
 \end{equation}
with $C_{\ast}= C_{\ast}(\sigma)=\left(\int_{0}^{\infty}Q(w)dw\right)^{-1}>0$ where 
\begin{equation}\label{def:Q}
Q(w):=\left[  \frac{1}{(\sigma-1)}  \left(w\right)^{\sigma-2}-\int_0^{w} (z+1)^{-\sigma} \left(w-z\right)^{\sigma-2} \dz\right],
\end{equation}
is a positive function for $w>0$, $Q(\cdot)\in L^1(0,\infty)$. Moreover, the function $G(V,V_0)$ satisfies
\begin{equation}\label{eq:bdG}
0<G(V,V_0)\leq C_{\ast} V_0^{-\frac 2 3}(V_0-V)^{\sigma-2} \quad \text{for any} \;\; V\in(0,V_0).
\end{equation} 

\end{proposition}

\begin{remark}
The normalization factor $C_{\ast}V_0^{-\frac 2 3}$ in \eqref{eq:homogpb2} has been chosen because in that way the function $G(V,V_0)$ can be thought of,
heuristically, as the solution of 
$-\mathcal{K}_{\infty} G(V,V_0)=\delta(V-V_0)$. However, since it is not straightforward to give a meaning to this distributional identity, we preferred to state the result as in Proposition \ref{prop:existG}.
\end{remark}

In order to construct the function $G(V,V_0)$ 
 in Proposition \ref{prop:existG} it is convenient to reformulate the problem \eqref{eq:homogpb1} in terms of a new set of variables to simplify the dependence on the parameters. To this end we define $\Lambda(\xi)$ by means of  
\begin{equation}\label{def:changevar}
 G(V,V_0)=C_{\ast}V_0^{\sigma-\frac{8}{3}}\Lambda(\xi) \quad \text{with}\quad \xi=\frac{V_0-V}{V_0},
\end{equation} 
where $C_{\ast}$ is as in Proposition \ref{prop:existG}. 
Then, using that 
$\Lambda(\xi)=0$ for $\xi<0$, we can rewrite \eqref{eq:homogpb1} as
 \begin{align}\label{eq:OmLam_1}
 &
 -\int_{0}^{\xi}\eta^{-\sigma}\left(1+\frac{2}{(1-\xi)^{1/3}}\eta^{1/3}+\frac{\eta^{2/3}}{(1-\xi)^{2/3}}\right)\bigl[\Lambda(\xi-\eta)-\Lambda(\xi)\bigr]\deta\nonumber \\&
 \quad +\Lambda(\xi)\int_{\xi}^{\infty}\eta^{-\sigma}\left(1+\frac{2}{(1-\xi)^{1/3}}\eta^{1/3}+\frac{\eta^{2/3}}{(1-\xi)^{2/3}}\right)\deta 
 =0,
\end{align}
where we used the change of variable $ \eta=\frac{v}{V_0}.$
\medskip

In order to simplify the writing we introduce the following notation: 
 \begin{equation}\label{def:Dalpha}
D^{\alpha}_{+} G(\xi):=\int_{0}^{\xi} \eta^{-(\alpha+1)}\bigl[G(\xi-\eta)-G(\xi)\bigr]d\eta
\end{equation}
for $0<\alpha<1$. Notice that this operator is well defined for any function $G \in \underset{\xi_{\ast}>0}{\bigcap}W^{1,\infty}(\xi_{\ast},1)$. 

Moreover, we set
\begin{equation}\label{eq:defF12}
F_1(\xi)=\frac{2}{(1-\xi)^{1/3}}\quad \text{and}\quad F_2(\xi)=\frac{1}{(1-\xi)^{2/3}},\;\; 0<\xi<1.
\end{equation}

Then, \eqref{eq:OmLam_1} can be rewritten as
 \begin{equation}\label{eq:eqLambda}
 \mathcal{L}(\Lambda)(\xi) =0 \quad\,   \text{for}\quad 0<\xi<1 
 \end{equation}
 where 
 \begin{equation}\label{eq:defLLambda}
 \mathcal{L}(\Lambda)(\xi)=-D^{\sigma-1}_{+}\Lambda(\xi)+ \frac{\xi^{-(\sigma-1)}}{(\sigma-1)}\Lambda(\xi)-  \sum_{k=1}^{2}F_{k}(\xi) \big[D^{\sigma-1-\frac{k}{3}}_{+}\Lambda(\xi)-\frac{\xi^{-(\sigma-1-\frac{k}{3})}}{(\sigma-1-\frac{k}{3})}\Lambda(\xi)\big].
\end{equation}
We recall that we are looking for a solution $G(V,V_0)$ satisfying the asymptotics \eqref{eq:homogpb2}.
Then, due to the definition \eqref{def:changevar}
we require that 
\begin{equation}\label{eq:astmLam}
\Lambda(\xi)= \xi^{\sigma-2}\big[1+O(\xi^{\frac{1}{3}})\big] \quad \text{for}\quad \xi\to 0^{+}.
\end{equation}
Our plan is to construct a solution of \eqref{eq:eqLambda}-\eqref{eq:astmLam} of the form

\begin{equation}\label{eq:solLambda}
\Lambda(\xi)=\sum_{j=0}^{2}\sum_{m=0}^{\infty} a_{m,j} \,\xi^{\sigma-2+\frac{j}{3}+m}, \quad a_{0,0}=1.
\end{equation}
We remark that this method to construct a solution of \eqref{eq:eqLambda}-\eqref{eq:astmLam} is reminiscent of the classical Frobenius method for ordinary differential equations. 

We first state two computational lemmas 
that will be used in the construction of the solution of the form \eqref{eq:solLambda}. 

\begin{lemma}\label{lem:phiab}
Let $ \Phi_{\alpha}(\beta)$ be the function defined by
\begin{equation}\label{def:phiab}
\Phi_{\alpha}(\beta)=\int_{0}^{1} \eta^{-(\alpha+1)}\bigl[(1-\eta)^{\beta}-1\bigr]d\eta\quad \text{for any}\;\; \alpha\in (0,1)\;\; \; \text{Re}(\beta)>-1.
\end{equation}
Then, we have
\begin{equation}\label{eq:betafc0}
\Phi_{\alpha}(\beta)= 
 \left[\frac{1}{\alpha} -\frac{\beta}{\alpha}\frac{\Gamma(1-\alpha)\Gamma(\beta)}{\Gamma(1-\alpha+\beta)}\right],
\end{equation}
where $\Gamma(\cdot)$ denotes the Gamma function. 
In particular, if $\sigma\in (1,2)$, the following identity holds
\begin{equation}\label{eq:phisigma}
\Phi_{\sigma-1}(\sigma-2)=\int_{0}^{1}y^{-\sigma}\bigl[(1-y)^{\sigma-2}-1\bigr]\dy=\frac{1}{\sigma-1}.
\end{equation} 
\end{lemma}

\medskip

\begin{lemma}\label{lem:omega}
Let $\omega(\ell,j,m;\sigma)$ be defined as 
\begin{equation}\label{eq:omega}
\omega(\ell,j,m;\sigma)\equiv \frac{\big(\sigma-2+\frac{j}{3}+m\big)}{\big(\sigma-1-\frac{\ell}{3}\big)} \,
B\Big(2-\sigma+\frac{\ell}{3},\sigma-2+\frac{j}{3}+m\Big),
\end{equation}
where $B$ denotes the Beta function.
 Suppose that $\ell\in\{0,1,2\}$, $j\in\{0,1,2\}$. There exist functions $C_1(\sigma),\, C_2(\sigma)$ satisfying $0<C_1(\sigma)\leq C_2(\sigma)<\infty$ for any $ \sigma\in \big(\frac{5}{3}, 2\big) $ such that 
\begin{equation}\label{eq:boundomega}
C_1(\sigma)\, (1+m)^{\sigma-1-\frac{\ell}{3}}\leq \vert \omega(\ell,j,m;\sigma) \vert \leq C_2(\sigma)\, (1+m)^{\sigma-1-\frac{\ell}{3}},\quad  \ell+j+m\geq 1.
\end{equation}
The functions $C_1(\sigma),\, C_2(\sigma)$ are continuous and uniformly bounded in every compact sub-interval of 
$\big(\frac{5}{3}, 2\big) $.
\end{lemma}
The proofs of \cref{lem:phiab,lem:omega} are postponed to Appendix \ref{appendix1}.

We now provide a solution of equation \eqref{eq:eqLambda}.
\begin{lemma}\label{lem:Frobth}
There exists a unique solution of  \eqref{eq:eqLambda} of the form 
\eqref{eq:solLambda}  where the series is absolutely convergent in the interval $0<\xi<1$. 
 Moreover, the solution is unique in the class of functions of the form \eqref{eq:solLambda} for which the series is absolutely convergent in the interval $0<\xi<\rho$ with $\rho>0$.
The  asymptotics \eqref{eq:astmLam} holds.
\end{lemma}

\begin{proof}[Proof of Lemma \ref{lem:Frobth}]

Using the definition \eqref{def:Dalpha} we have $ D^{\alpha}_{+}(\xi^{\beta})= \xi^{\beta-\alpha}\Phi_{\alpha}(\beta)$, for $\alpha\in (0,1)$ and $\text{Re}(\beta)>-1$ 
where $\Phi_{\alpha}(\beta)$ is as in \eqref{def:phiab}. Using Lemma \ref{lem:phiab} we can write 
\begin{equation}\label{eq:betafc}
D^{\alpha}_{+}(\xi^{\beta})=\frac{ \xi^{\beta-\alpha}}{\alpha}\left(1-\beta B\big(1-\alpha,\beta\big)\right)
\end{equation}
where $B$ is the Beta function.

Suppose that \eqref{eq:eqLambda} has a solution of the form \eqref{eq:solLambda} which is convergent in an interval $0<\xi<\rho$ with $\rho>0$. Then, using \eqref{eq:betafc} we get 
\begin{align}\label{eq:eqLambda2}
&\sum_{j=0}^{2}\sum_{m=0}^{\infty} \frac{\big(\sigma-2+\frac{j}{3}+m\big)}{\big(\sigma-1\big)} B\Big(2-\sigma,\sigma-2+\frac{j}{3}+m\Big) a_{m,j}\,\xi^{-1+\frac{j}{3}+m} 
\nonumber
\\&=-\sum_{\ell=1}^{2}F_\ell(\xi) \sum_{j=0}^{2}\sum_{m=0}^{\infty} \omega(\ell,j,m;\sigma) a_{m,j}\,\xi^{-1+\frac{j}{3}+\frac{\ell}{3}+m}
\end{align}
where $\omega(\ell,j,m;\sigma)$ is as in \eqref{eq:omega}.

Notice that in the derivation of \eqref{eq:eqLambda2} there is a cancellation of the leading order term in the difference $-D^{\sigma-1}_{+}\Lambda(\xi)+ \frac{\xi^{-(\sigma-1)}}{(\sigma-1)}\Lambda(\xi)$ due to \eqref{eq:phisigma}.

Since the functions $F_{\ell}$ defined as in \eqref{eq:defF12} are analytic in the unit disk, they can be written as  $F_{\ell}(\xi)=\sum_{k=0}^{\infty} b_{k,\ell}\xi^k$, $\ell=1,2$, where the series converges for $\vert \xi\vert<1$.  
Therefore, the classical Cauchy estimates imply that for any $\delta>0$ there exists a constant $C_\delta$ such that 
\begin{equation}\label{eq:boundb}
\vert b_{k,\ell}\vert \leq C_{\delta}\big(1+\frac{\delta}{3}\big)^k.
\end{equation}
In what follows we will always assume that $\delta<1$.

Equating coefficients, we then obtain
\begin{align}\label{eq:iteration_akm}
&\omega(0,s,k;\sigma)a_{k,s}\nonumber \\&= -\sum_{\{(\ell,j):\,\rho=s,\, \tau=0\}}  \sum_{m=0}^{k} b_{k-m,\ell}\, \omega (\ell,j,m;\sigma)\, a_{m,j} -\sum_{\{(\ell,j):\,\rho=s,\, \tau=1\}}  \sum_{m=0}^{k-1} b_{k-1-m,\ell}\,\omega (\ell,j,m;\sigma)\, a_{m,j},
\end{align}
for $k=0,1,2,\dots$ and $s\in\{0,1,2\}$. In the sums $j\in\{0,1,2\}$ and $\ell\in\{1,2\}$. In \eqref{eq:iteration_akm} we denote as 
$\rho=\rho(j,\ell)=3 \big\{\frac{j+\ell }{3}\big\}$ and  $\tau= \big[\frac{j+\ell }{3}\big]$ where $\big\{\cdot \big\}$ denotes the fractional part, 
and $\big[\cdot \big]$ denotes the integer part. Note that  $\rho\in \{0,1,2\}$ and $\tau\in [0,1]$.

We observe that \eqref{eq:iteration_akm} is an iterative formula. Given the initial condition $a_{0,0}=1$ we can determine the coefficients $\{a_{k,s}\;:\; k= 0,1,2,\dots,\text{ and } s= 0,1, 2\}$ in a recursive way following the order indicated in Figure \ref{fig1} .

\begin{figure}[th]
\centering
\includegraphics [scale=0.3]{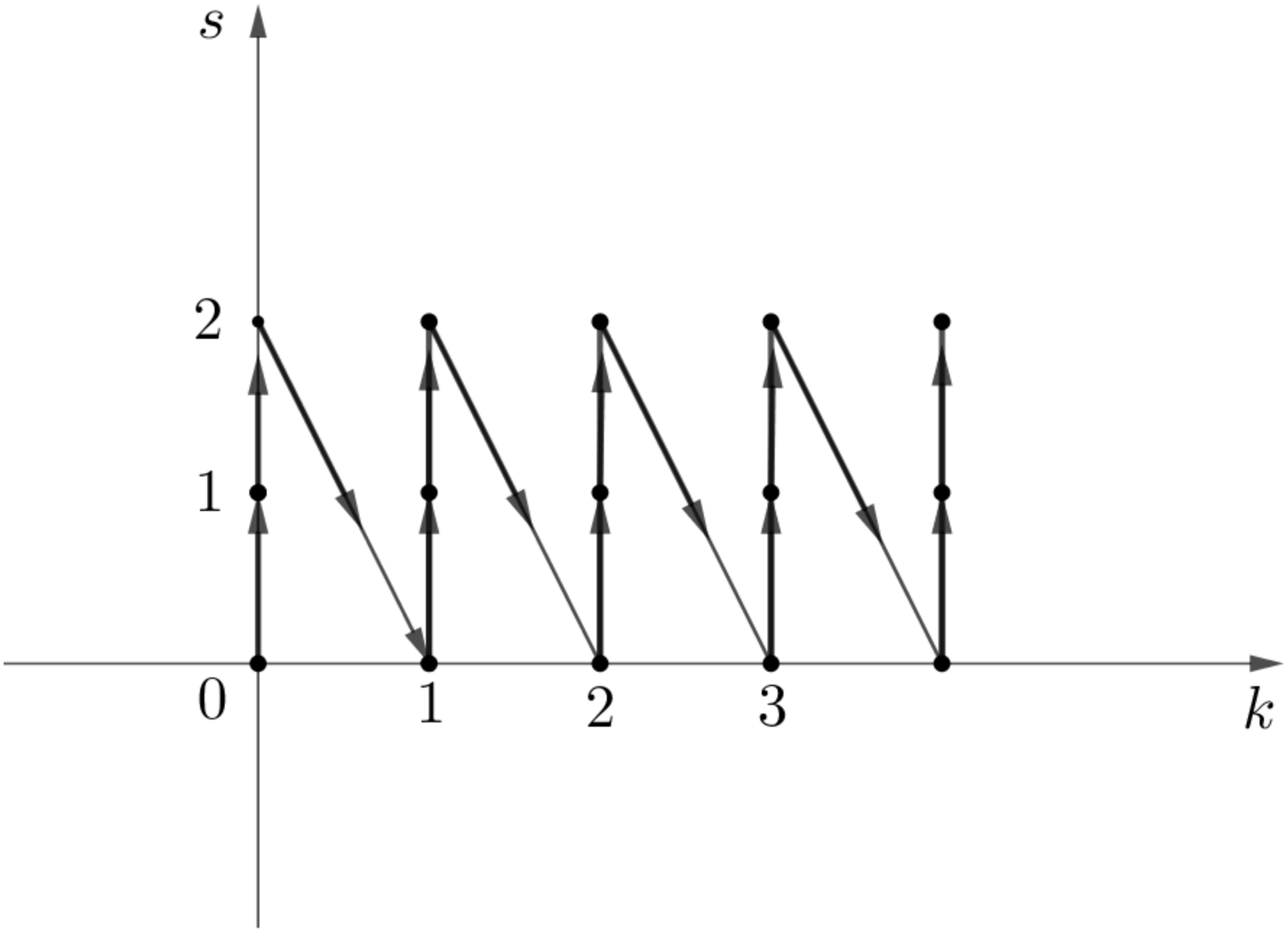}\caption{Recursive procedure to determine $a_{k,s}$ for $k=0,1,2,\dots$ \text{ and } $s\in\{0,1,2\}$.  The values of $a_{k,s}$ are determined by the values of $a_{m,j}$ with $(m,j)$ preceding $(k,s)$ in the graph above. \label{fig1} }%
\end{figure}

We now estimate the coefficients $a_{k,s}$. To this end notice that from \eqref{eq:iteration_akm} we have 
\begin{align*}
\vert a_{k,s}\vert \leq& \frac{1}{\vert \omega(0,s,k;\sigma)\vert}\sum_{\{(\ell,j):\,\rho=s,\, \tau=0\}}  \sum_{m=0}^{k}\vert  b_{k-m,\ell}\vert \, \vert \omega (\ell,j,m;\sigma)\vert \,\vert  a_{m,j}\vert   \nonumber \\&
+\frac{1}{\vert \omega(0,s,k;\sigma)\vert}\sum_{\{(\ell,j):\,\rho=s,\, \tau=1\}}  \sum_{m=0}^{k-1}\vert  b_{k-1-m,\ell}\vert \,\vert \omega (\ell,j,m;\sigma)\vert \, \vert a_{m,j}\vert ,
\end{align*}

Using \eqref{eq:boundb}, Lemma \ref{lem:omega} as well as the change of variables 
\begin{equation}\label{eq:atotheta}
\theta_{k,s}=\frac{a_{k,s}}{\big(1+\frac{\delta}{3}\big)^{k}},
\end{equation} 
we get
\begin{align}\label{eq:theta}
\vert \theta_{k,s}\vert \leq C_{\delta}(\sigma)\sum_{ \{j<s\}}  \frac{1}{(1+k)^{\frac{1}{3}}}  \vert \theta_{k,j} \vert  
+C_{\delta}(\sigma)\sum_{j=0}^{2}  \sum_{m=0}^{k-1} \frac{1}{(1+m)^{\frac{1}{3}}}\vert  \theta_{m,j}\vert,
\end{align}
where $C_{\delta}(\sigma)$ is uniformly bounded for $\sigma$ in any compact set of $(\frac{5}{3},2)$.
Observe that we used the convention that for $s=0$ the sum  $\sum_{ \{j<s\}}\big(\dots \big)=0$.

We now define $\| \theta_k \|=\underset{ s\in\{0,1,2\}}{\max}\vert \theta_{k,s}\vert$. Then, separating the term $m=k$ in the first sum in the right hand side of \eqref{eq:theta}, using also that $\frac{1}{(1+k)^{\frac{1}{3}}}\leq 1$, we obtain
\begin{align*}
\vert \theta_{k,s}\vert \leq 
C_{\delta}(\sigma)\sum_{\{j< s\}}\vert  \theta_{k,j}\vert +C_{\delta}(\sigma) \sum_{m=0}^{k-1} \frac{1}{(1+m)^{\frac{1}{3}}}  \| \theta_m \| 
\end{align*}
We set $A_{k-1}=\sum_{m=0}^{k-1} \frac{1}{(1+m)^{\frac{1}{3}}}  \| \theta_m \| $. For any value of $s$ we get
\begin{align*}
&\vert \theta_{k,0}\vert \leq C_{\delta}(\sigma)A_{k-1} \\&
\vert \theta_{k,1}\vert \leq C_{\delta}(\sigma)\big(A_{k-1}+\vert \theta_{k,0}\vert\big) \\&
\vert \theta_{k,2}\vert \leq C_{\delta}(\sigma)\big[A_{k-1}+\big(\vert \theta_{k,0}\vert +\vert \theta_{k,1}\vert \big)\big].
\end{align*}
Therefore, we obtain 
\begin{equation}\label{eq:thetak}
\|\theta_k\|\leq C_{\delta}(\sigma) A_{k-1}\leq C_{\delta}(\sigma) \sum_{m=0}^{k-1} \frac{1}{(1+m)^{\frac{1}{3}}}  \| \theta_m \| 
\end{equation}
where $C_{\delta}(\sigma)$ can change from line to line but  is always independent of $k$.

We now choose $M=M_{\delta}(\sigma)$ such that 
\begin{equation}\label{eq:bardelta}
\frac{C_{\delta}(\sigma)}{(1+M)^{\frac{1}{3}}}\leq \frac{\log(1+\frac{\delta}{3})}{10}=:\bar{\delta}.\end{equation}
Then, for $k\geq M+1$ we obtain
\begin{equation*}
\|\theta_k\|\leq K_{\delta}(\sigma)+\bar{\delta}\sum_{m=M}^{k-1}\| \theta_m \| 
\end{equation*}
where $K_{\delta}(\sigma)=C_{\delta}(\sigma) \sum_{m=0}^{M-1} \frac{1}{(1+m)^{\frac{1}{3}}}  \| \theta_m \|.$
Notice that $K_{\delta}(\sigma) <\infty$ due to \eqref{eq:thetak} as well as the fact that $\|\theta_0\|=1$ since $a_{0,0}=1$ (cf. also \eqref{eq:atotheta}).

Therefore, for any $k\geq 0$
\begin{equation*}
\|\theta_k\|\leq K_{\delta}(\sigma)+\bar{\delta}\sum_{m=0}^{k-1}\| \theta_m \|.
\end{equation*}
It then follows by induction that 
\begin{equation*}
\|\theta_k\|\leq 2 K_{\delta}(\sigma)\exp(5\bar{\delta} k).\quad  
\end{equation*}

Combining this estimate with \eqref{eq:atotheta} and \eqref{eq:bardelta} we have 
\begin{equation*}
\vert a_{k,s}\vert \leq 2 K_{\delta}(\sigma)\left(1+\frac{\delta}{3}\right)^{k}\left(1+\frac{\delta}{3}\right)^{k}\leq 2 K_{\delta}(\sigma)\left(1+\delta\right)^{k} \quad s\in\{0,1,2\}
\end{equation*}
where we used that $\delta<1$. 
Then, we obtained that for any $\delta>0$ and for any $k= 0,1,2,\dots$
\begin{equation}\label{eq:boundcoeff}
\underset{ s\in\{0,1,2\}}{\max}\vert a_{k,s}\vert \leq C_{\delta}(\sigma)  (1+\delta)^k .
\end{equation}

This implies the convergence of the series in \eqref{eq:solLambda} for $|\xi|<1$. Notice that the same argument could be applied for any solution of the form \eqref{eq:solLambda} since the coefficient $a_{k,s}$ would be then determined uniquely by means of the formula \eqref{eq:eqLambda2}. The asymptotics \eqref{eq:astmLam} is a straightforward consequence of the representation formula \eqref{eq:solLambda}.
\end{proof}

We provide some estimates for the function $\Lambda(\xi)$ that will be useful in what follows.

\begin{lemma}\label{lem:bdLamda}
The unique solution of  \eqref{eq:eqLambda} of the form \eqref{eq:solLambda} obtained in Lemma \ref{lem:Frobth} satisfies 
\begin{equation}\label{eq:bdLAM}
0<\Lambda(\xi)\leq \xi^{\sigma-2} \quad \text{for any} \quad \xi\in (0,1).
\end{equation}
\end{lemma}

\begin{lemma}\label{lem:estdiffLam}
The unique solution of  \eqref{eq:eqLambda} of the form \eqref{eq:solLambda} obtained in Lemma \ref{lem:Frobth} satisfies 
\begin{equation}\label{eq:estdiffLam}
 \vert \Lambda(\xi(1-\zeta))-\Lambda(\xi) \vert \leq C(\sigma,R)\,\xi^{\sigma-2} \zeta(1-\zeta)^{\sigma-2} 
 \end{equation}
 for any $0< \zeta< 1$ and $0< \xi\leq R<1$. Here $C(\sigma,R)$ is a constant depending only on $\sigma$ and $R$.
 
Moreover, we define the function $H(\xi)$ as
 \begin{equation}\label{eq:funH}
H(\xi)=\Lambda(\xi)-\xi^{\sigma-2}.
\end{equation}
Then
\begin{equation}\label{eq:bdH}
\vert H(\xi)\vert \leq C(\sigma,R)\xi^{\sigma-2+\frac{1}{3}},
\end{equation}
for any $0<  \xi<1$ 
and 
\begin{equation}\label{eq:bdH1}
\vert H(\xi(1-\zeta))-H(\xi)\vert \leq C(\sigma,R)\xi^{\sigma-2+\frac{1}{3}}\min\{\zeta,(1-\zeta)^{\sigma-2+\frac{1}{3}}\},
\end{equation}
for any $0<\zeta<1$, $0< \xi\leq R<1$.
 \end{lemma}
 
The proofs of Lemmas \ref{lem:bdLamda}, \ref{lem:estdiffLam} are given in Appendix \ref{appendix1}. 

\medskip

We can now conclude the proof of Proposition \ref{prop:existG}.

\medskip

\medskip
\begin{proof}[Proof of Proposition \ref{prop:existG}]
We define the function $G(V,V_0)$ as 
\begin{equation}\label{eq:GLamb}
G(V,V_0)=C_{\ast}V_0^{\sigma-\frac{8}{3}}\Lambda(\xi) \quad \text{with} \quad \xi=\frac{V_0-V}{V_0},\quad \text{for} \quad V<V_0
\end{equation}
and $G(V,V_0)=0$ for $V>V_0$. 
Here $C_{\ast}$ is as in the statement of Proposition \ref{prop:existG}. Since the function $\Lambda(\xi) $ satisfies \eqref{eq:eqLambda} we have that $G(V,V_0)$ satisfies \eqref{eq:homogpb1} with $\mathcal{K}_{\infty}$ as in \eqref{eq:infOmega}.  Moreover, the asymptotics for $\Lambda(\xi)$ given by \eqref{eq:astmLam} implies \eqref{eq:homogpb2}.  The positivity of the function $Q(w)$ in the statement of Proposition \ref{prop:existG} (cf. \eqref{def:Q}) follows from Lemma \ref{lem:posQ} in Appendix \ref{appendix1}. Finally, \eqref{eq:bdG} follows from Lemma \ref{lem:bdLamda}.

\end{proof}

\medskip

In order to prove Theorem \ref{th:genbdpb} we first solve \eqref{eq:genbdpb}-\eqref{eq:genbdpb1} with homogeneous boundary conditions. This is the content of the following Lemma.

\begin{lemma}\label{lem:sol1} 
Let $G(V,V_0)$ be as in Proposition \ref{prop:existG}. Let $\bar{V}>0$ and suppose that
$g\in\left(\bigcap_{\delta>0}W^{1,\infty}[\delta ,\bar{V}-\delta]\right)$. In addition suppose that $\|g\|_{\infty}<\infty$.
We then define  the function
\begin{align}\label{eq:ffisol1}
&u_1(V)=\int_{V}^{\bar{V}} G(V,V_0) g(V_0)dV_0, \quad\, \text{for}\quad 0<V \\& 
u_1(V)=0, \qquad\qquad\text{for}\quad V> \bar{V}. \label{eq:sourcepb1}
\end{align}
Then, \eqref{eq:ffisol1} and \eqref{eq:sourcepb1} give a well defined function $u_1$ which belongs to $L^{\infty}(0,\infty)$ and satisfies
 the following estimates:
 \begin{equation}\label{eq:bdu1}
\vert u_1(V)\vert \leq C(\sigma) \frac{(\bar{V}{-}V)^{\sigma-1}}{\bar{V}^\frac{2}{3}}\|g\|_{\infty}\quad\quad\text{for}\quad 0 < V< \bar{V}.
\end{equation}
 \begin{equation}\label{eq:bdu1_2}
\left\vert\frac{d u_1}{dV}(V)\right\vert\leq C_{\delta}(\sigma)  \left(\|g\|_{\infty} \frac{ (\bar{V}_{\delta}{-}V)^{\sigma-2}}{V^2}(\bar{V}_{\delta})^{\frac{1}{3}}+  
\|g'\|_{\infty,\delta} \frac{ (\bar{V}_{\delta}{-}V)^{\sigma-1}}{V} (\bar{V}_{\delta})^{\frac{1}{3}}\right) \quad\text{for}\quad \delta < V< \bar{V}_{2\delta}
\end{equation} 
where $\bar{V}_{\delta}=\bar{V}-{\delta}$ and we use the notation 
 $\|\cdot\|_{\infty,\delta}=\|\cdot\|_{L^{\infty}([ \delta,\bar{V}_{\delta}])}$. 
Moreover, $\mathcal{K}_{\infty}u_1$ with $\mathcal{K}_{\infty}$ given in \eqref{eq:infOmega} is well defined and  we have
\begin{equation*}
-\mathcal{K}_{\infty}u_1(V)=g(V) \quad\quad\text{for}\quad 0 < V< \bar{V}
\end{equation*}
where $\mathcal{K}_{\infty}$ is as in \eqref{eq:infOmega}.

\end{lemma}

\begin{proof}
We first note that the integral on the right hand side of \eqref{eq:ffisol1} is well defined due to \eqref{eq:homogpb2}. Since $G(V,V_0)=0$ for $V>V_0$ we obtain \eqref{eq:sourcepb1}.
Using \eqref{eq:GLamb} we can rewrite $u_1(V)$ as 
\begin{equation*}
u_1(V)=C_{*}\int_{V}^{\bar{V}} V_0^{\sigma-\frac{8}{3}}\Lambda\left(\frac{V_0-V}{V_0}\right) g(V_0)dV_0 \qquad \text{ for}\quad V< \bar{V}.
\end{equation*}
Applying \eqref{eq:bdLAM} in Lemma \ref{lem:bdLamda} we obtain 
\begin{equation*}
\vert u_1(V)\vert \leq C_{*} \|g\|_{\infty} \int_{V}^{\bar{V}} \frac {V_0^{\sigma-\frac{8}{3}}}{(\bar{V}-V_{0})^{\nu}}\left(\frac{V_0-V}{V_0}\right)^{\sigma-2} dV_0
\leq C(\sigma) \|g\|_{\infty} \frac{ (\bar{V}-V)^{\sigma-1} }{\bar{V}^{\frac{2}{3}}}.
\end{equation*}
This proves \eqref{eq:bdu1}. 

In order to prove \eqref{eq:bdu1_2} we first notice that
\begin{equation*}
u_1(V)=u_{1,1}(V)+u_{1,2}(V) \qquad \text{ for}\quad V< \bar{V},
\end{equation*}
with
\begin{align}
&u_{1,1}(V)=C_{*}\int_{V}^{\bar{V}_{\delta}} V_0^{\sigma-\frac{8}{3}}\Lambda\left(\frac{V_0-V}{V_0}\right) g(V_0)dV_0\\& 
u_{1,2}(V)=C_{*}\int_{\bar{V}_{\delta}}^{\bar{V}}V_0^{\sigma-\frac{8}{3}}\Lambda\left(\frac{V_0-V}{V_0}\right) g(V_0)dV_0 .
\end{align}  
We now replace the variable of integration $V_0$ in the equation for $u_{1,1}(V)$ by means of the change of variable $\xi=\frac{V_0-V}{V_0}$. 
We then obtain 
\begin{equation*}
u_{1,1}(V)=
C_{*}{V}^{\sigma-\frac{5}{3}} \int_{0}^{1-\frac{V}{\bar{V}_{\delta}}} \frac{\Lambda(\xi)  }{(1-\xi)^{\sigma-\frac{2}{3}}}g\left(\frac{V}{1-\xi}\right)d\xi.
\end{equation*}
We now compute the derivative of $u_{1,1}$ with respect to $V$. We have
\begin{align*}
\frac{d u_{1,1}}{dV}(V)=& C_{*}\left(\sigma-\frac{5}{3}\right) {V}^{\sigma-\frac{8}{3}} \int_{0}^{1-\frac{V}{\bar{V}_{\delta}}} \frac{\Lambda(\xi)  }{(1-\xi)^{\sigma-\frac{2}{3}}}g\left(\frac{V}{1-\xi}\right)d\xi
-C_{*}\frac{{\bar{V}_{\delta}}^{\sigma-\frac{5}{3}}}{V^2}\Lambda\left(1-\frac{V}{\bar{V}_{\delta}}\right) g(\bar{V}_{\delta})\\&+ C_{*}{V}^{\sigma-\frac{5}{3}}  \int_{0}^{1-\frac{V}{\bar{V}_{\delta}}} \frac{\Lambda(\xi)  }{(1-\xi)^{\sigma+\frac{1}{3}}}g' \left(\frac{V}{1-\xi}\right)d\xi.
\end{align*}
Therefore, using again \eqref{eq:bdLAM} in Lemma \ref{lem:bdLamda}, we have
\begin{multline}\label{eq:bdderu11}
\left\vert\frac{d u_{1,1}}{dV}(V)\right\vert \leq  
C_{*}\left(\sigma-\frac{5}{3}\right) {V}^{\sigma-\frac{8}{3}}\|g\|_{\infty} \int_{0}^{1-\frac{V}{\bar{V}_{\delta}}} \frac{\xi^{\sigma-2}  }{(1-\xi)^{\sigma-\frac{2}{3}}}d\xi 
+C_{*}\left(1-\frac{V}{\bar{V}_{\delta}}\right)^{\sigma-2}\frac{{\bar{V}_{\delta}}^{\sigma-\frac{5}{3}}}{V^2} \|g\|_{\infty} \\*
+ C_{*}{V}^{\sigma-\frac{5}{3}}  \int_{0}^{1-\frac{V}{\bar{V}_{\delta}}} \frac{\xi^{\sigma-2}  }{(1-\xi)^{\sigma+\frac{1}{3}}}\|g'\|_{\infty,\delta}d\xi  \\*
\leq C_{\delta}(\sigma)\left( \|g\|_{\infty}\frac{  (\bar{V}_{\delta}-V)^{\sigma-1}}{V \, \bar{V}_{\delta}^{\frac{2}{3}}} + 
 \|g\|_{\infty} \frac{ (\bar{V}_{\delta}-V)^{\sigma-2}}{V^2}(\bar{V}_{\delta})^{\frac{1}{3}}+  
\|g'\|_{\infty,\delta} \frac{ (\bar{V}_{\delta}-V)^{\sigma-1}}{V} (\bar{V}_{\delta})^{\frac{1}{3}}\right).
\end{multline}
We now estimate the derivative of $u_{1,2}$ with respect to $V$. We have
\begin{align}\label{eq:bdderu12}
\left\vert\frac{d u_{1,2}}{dV}(V)\right\vert &\leq  C_{\delta}(\sigma)\|g\|_{\infty}\int_{\bar{V}_{\delta}}^{\bar{V}}V_0^{\sigma-\frac{8}{3}-1}\left(\frac{V_0-V}{V_0}\right)^{\sigma-3}dV_0 \nonumber\\&
\leq C_{\delta}(\sigma) \bar{V}^{-\frac{2}{3}} \|g\|_{\infty} (\bar{V}_{\delta}-V)^{\sigma-2}, \quad \text{for}\;\, \delta<V\leq \bar{V}_{2\delta}.
\end{align}
This is due to the fact that the estimate $\vert\Lambda'(\xi)\vert \leq C\xi^{\sigma-3}$, which follows from \eqref{eq:estdiffLam} taking the limit $\zeta\to 0$, is not uniform as $\xi\to 1$.
Assuming that $\delta>0$ is sufficiently small we have that $\bar{V}_{\delta}\geq \frac{\bar{V}}{2}$. Combining  now \eqref{eq:bdderu11} and \eqref{eq:bdderu12} we obtain  \eqref{eq:bdu1_2}.

We observe that 
\begin{equation*}  
 \mathcal{K}_{\infty}u_1(V)=\int_{0}^{\infty}v^{-\sigma}(V^{1/3}+v^{1/3})^2\bigl[u_1(V+v)-u_1(V)\bigr]\dv.
 \end{equation*}
can be defined for 
$ u_1 \in L^{\infty}({\delta},\bar{V}_{2\delta}) \cap W^{1,\infty}([\delta,\bar{V}_{2\delta}])$ for any $\delta>0$. 
Then, $ \mathcal{K}_{\infty} u_1 (V)$ is well defined for any $V\in (0,\bar{V})$. 

Using \eqref{eq:ffisol1} we can rewrite \eqref{eq:infOmega} as 
\begin{align*}  
 \mathcal{K}_{\infty}u_1(V)=&\int_{0}^{\infty}v^{-\sigma}(V^{1/3}+v^{1/3})^2\\&
 \times\left[\chi_{[0, \bar{V}-V]}(v)\int_{V+v}^{\bar{V}} G(V+v,V_0) g(V_0)dV_0 -\int_{V}^{\bar{V}} G(V,V_0) g(V_0)dV_0\right]\dv,
 \end{align*}
where we denoted by $\chi_{[0, \bar{V}-V]}$ the characteristic function of the set $[0, \bar{V}-V]$. We now use a regularization argument for the kernel $v^{-\sigma}$. For any $V\in (0,\bar{V})$ we have
\begin{align*}  
\mathcal{K}_{\infty}u_1(V)=&\lim_{\varepsilon\to 0}\int_{0}^{\infty}(v+\varepsilon)^{-\sigma}(V^{1/3}+v^{1/3})^2\\&
\times \left[\chi_{[0, \bar{V}-V]}(v)\int_{V+v}^{\bar{V}} G(V+v,V_0) g(V_0)dV_0 -\int_{V}^{\bar{V}} G(V,V_0) g(V_0)dV_0\right]\dv.
\end{align*}
Note that this formula holds due to Lebesgue's dominated convergence theorem, as well as \eqref{eq:bdu1_2}.
Applying Fubini to the right hand side of the equation above and using also \eqref{def:changevar}, we obtain
\begin{align}\label{eq:Ominfu1}
\mathcal{K}_{\infty}u_1(V)
=\lim_{\varepsilon\to 0}\left\{\int_{V}^{\bar{V}}dV_0g(V_0)  V_0^{\sigma-\frac{8}{3}} W_{\eps}(V,V_0)\right\},
  \end{align}
  where we set 
  \begin{align*}
W_{\eps}(V,V_0)&= C_{\ast}\int_{0}^{V_0-V}dv (v+\varepsilon)^{-\sigma} (V^{1/3}+v^{1/3})^2 \Lambda\left(\frac{V_0-V-v}{V_0}\right) \\&-C_{\ast}\left(\int_{0}^{\infty}dv (v+\varepsilon)^{-\sigma} (V^{1/3}+v^{1/3})^2 \right) \Lambda\left(\frac{V_0-V}{V_0}\right).
  \end{align*}
 Choosing $0<\delta<\bar{V}-V$ we now write  
 \begin{align}\label{eq:L1L2}
\int_{V}^{\bar{V}}dV_0g(V_0) V_0^{\sigma-\frac{8}{3}}  W_{\eps}(V,V_0)&=\int_{V}^{V+\delta}dV_0g(V_0) V_0^{\sigma-\frac{8}{3}} W_{\eps}(V,V_0)\nonumber\\&
\quad+\int_{V+\delta}^{\bar{V}}dV_0g(V_0)V_0^{\sigma-\frac{8}{3}}  W_{\eps}(V,V_0)=:L_1+L_2.
  \end{align} 
We estimate now $L_2$. We can rewrite $W_{\eps}(V,V_0)$ as 
    \begin{align*}\label{eq:Weps}
W_{\eps}(V,V_0)&= C_{\ast}\int_{0}^{V_0-V}dv (v+\varepsilon)^{-\sigma} (V^{1/3}+v^{1/3})^2 \left(\Lambda\left(\frac{V_0-V-v}{V_0}\right)- \Lambda\left(\frac{V_0-V}{V_0}\right)\right)\\& \quad -C_{\ast}\left(\int_{V_0-V}^{\infty}dv (v+\varepsilon)^{-\sigma} (V^{1/3}+v^{1/3})^2 \right) \Lambda\left(\frac{V_0-V}{V_0}\right).
  \end{align*}
  We notice that in the region where $V_0-V\geq \delta$ the second term on the right hand side of the equation above is immediately uniformly bounded, independently of $\eps$. Concerning the first term, combining \eqref{eq:solLambda} and Lemma \ref{lem:estdiffLam}, we can also estimate it by a constant independent of $\eps$ despite the unboundedness of the integrand, for $V_0-V\geq \delta$.     
  We can then apply Lebesgue's dominated convergence theorem and we obtain
 \begin{align*}
\lim_{\eps\to 0}L_2=\int_{V+\delta}^{\bar{V}}dV_0g(V_0)V_0^{\sigma-\frac{8}{3}}  \lim_{\eps\to 0} W_{\eps}(V,V_0).
\end{align*}
On the other hand 
 \begin{align*}
 \lim_{\eps\to 0} W_{\eps}(V,V_0)&=C_{\ast}\int_{0}^{V_0-V}dv v^{-\sigma} (V^{1/3}+v^{1/3})^2 \left(\Lambda\left(\frac{V_0-V-v}{V_0}\right)- \Lambda\left(\frac{V_0-V}{V_0}\right)\right)\\& \quad -C_{\ast}\left(\int_{V_0-V}^{\infty}dv v^{-\sigma} (V^{1/3}+v^{1/3})^2 \right) \Lambda\left(\frac{V_0-V}{V_0}\right)=0,
\end{align*}
due to \eqref{eq:eqLambda} (cf. also \eqref{eq:OmLam_1}). Therefore,
 \begin{equation}\label{eq:limL2}
\lim_{\eps\to 0}L_2=0.
\end{equation}

The following argument is reminiscent of the argument in the proof of Lemma \ref{lem:posQ}. We now estimate $L_1$. We use the change variable $v=\eps z$, $V_0=V+\eps X$ with $X>0$ and we obtain
  \begin{align*}  
L_1&
= C_{\ast} \int_{0}^{\frac{\delta}{\eps}}dX g(V+\eps X) (V+\eps X)^{\sigma-\frac{8}{3}} \mathcal{U}_{\eps}(X)
  \end{align*}
where
\begin{align*}  
 \mathcal{U}_{\eps}(X)=& \eps^{2-\sigma}\Big[ \int_{0}^{X}dz (z+1)^{-\sigma} (V^{1/3}+(\eps z)^{1/3})^2 \Lambda\left( \frac{\eps(X-z)}{V+\eps X}\right)\\&
-\left(\int_{0}^{\infty}dz (z+1)^{-\sigma} (V^{1/3}+(\eps z)^{1/3})^2 \right) \Lambda\left(\frac{\eps X}{V+\eps X}\right)\Big].
  \end{align*}
We now rewrite $ \mathcal{U}_{\eps}$ in a more convenient form. 
\begin{align*}  
 \mathcal{U}_{\eps}(X)=& \eps^{2-\sigma} \Big[ \int_{0}^{X}dz (z+1)^{-\sigma} (V^{1/3}+(\eps z)^{1/3})^2 \left[ \Lambda\left( \frac{\eps(X-z)}{V+\eps X}\right)-\Lambda\left(\frac{\eps X}{V+\eps X}\right)\right]\\&
-\left(\int_{X}^{\infty}dz (z+1)^{-\sigma} (V^{1/3}+(\eps z)^{1/3})^2 \right) \Lambda\left(\frac{\eps X}{V+\eps X}\right)\Big]\\&
= \int_{0}^{X}dz \big[(z+1)^{-\sigma}- z^{-\sigma}\big](V^{1/3}+(\eps z)^{1/3})^2 \eps^{2-\sigma} \left[ \Lambda\left( \frac{\eps(X-z)}{V+\eps X}\right)-\Lambda\left(\frac{\eps X}{V+\eps X}\right) \right]\\&-\left(\int_{X}^{\infty}dz (z+1)^{-\sigma} (V^{1/3}+(\eps z)^{1/3})^2 \right) \eps^{2-\sigma} \Lambda\left(\frac{\eps X}{V+\eps X}\right)\\&
+\int_{0}^{X}dz z^{-\sigma}(V^{1/3}+(\eps z)^{1/3})^2 \eps^{2-\sigma} \left[ \Lambda\left( \frac{\eps(X-z)}{V+\eps X}\right)-\Lambda\left(\frac{\eps X}{V+\eps X}\right) \right]\\&
:= I_1+I_2+I_3.
  \end{align*}
We first estimate the function $\mathcal{U}_{\eps}(X)$ for $0<X\leq 1$. Using Lemma \ref{lem:estdiffLam}, with $R=\frac 1 2$, we have
\begin{align}  \label{eq:I1I3}
\vert  I_1+I_3 \vert  &\leq C(\sigma)(V+\eps X)^{2-\sigma} (V^{2/3}+\eps^{2/3}) \int_{0}^{X}dz (z+1)^{-\sigma} \left[(X-z)^{\sigma-2}-(X)^{\sigma-2}\right]\\&
\leq C(\sigma)(V+\eps X)^{2-\sigma} (V^{2/3}+\eps^{2/3})  X^{\sigma-1}\leq C(\sigma)(V^{8/3-\sigma}+\eps^{8/3-\sigma})  X^{\sigma-1}
  \end{align}
 where we used that $\frac{\eps X}{V+\eps X}\leq \frac{\eps}{V}$ and this can be chosen smaller than $\frac 1 2$ if $\eps>0$ is sufficiently small.
Moreover, in a similar way we obtain
\begin{align}\label{eq:I2}
\vert  I_2 \vert  &\leq C(\sigma)(V^{8/3-\sigma}+\eps^{8/3-\sigma})  X^{\sigma-2}.
 \end{align}
Combining \eqref{eq:I1I3} and \eqref{eq:I2} we have 
\begin{align}\label{eq:bdUeps}
\vert  \mathcal{U}_{\eps} \vert  &\leq C(\sigma)(V^{8/3-\sigma}+\eps^{8/3-\sigma})  X^{\sigma-2}, \quad \text{for}\;\; 0<X\leq 1.
 \end{align}

  We now estimate the terms $ I_1$, $ I_2$ and $ I_3$ for $1\leq X\leq \frac{ \delta}{\eps}$.
  
 We consider $ I_1$ and we change variables setting $z=X\zeta$. We have
 \begin{align}  \label{eq:bdI1}
I_1 =& X \int_{0}^{1}\dzeta \big[(X\zeta+1)^{-\sigma}- (X\zeta)^{-\sigma}\big](V^{1/3}+(\eps X\zeta)^{1/3})^2 \eps^{2-\sigma}\times\nonumber\\
& \qquad\qquad\qquad\qquad\qquad\qquad\qquad\qquad\qquad\qquad\times\left[ \Lambda\left( \frac{\eps X(1-\zeta)}{V+\eps X}\right)-\Lambda\left(\frac{\eps X}{V+\eps X}\right) \right] \nonumber\\&
\leq  C(\sigma) X^{-1} \int_{0}^{1}\dzeta \vert (\zeta+\frac{1}{X})^{-\sigma}- \zeta^{-\sigma}\vert (V^{1/3}+(\eps X\zeta)^{1/3})^2 (V+\eps X)^{2-\sigma}\min\{\zeta, (1-\zeta)^{\sigma-2}\}\nonumber \\&
\leq  C(\sigma) X^{-1} (\bar{V})^{\frac{8}{3}-\sigma}\int_{0}^{1}\dzeta \zeta^{-\sigma}(X\zeta)^{-(2-\sigma-\eps)}\min\{\zeta, (1-\zeta)^{\sigma-2}\} \nonumber\\&
\leq C_{\eps}(\sigma) X^{-(3-\sigma-\eps)},
  \end{align}
  where we used \eqref{eq:estdiffLam} in Lemma \ref{lem:estdiffLam} and \eqref{eq:eq:bdJ1bis} in the last inequality. This estimate holds for $0<X\leq\frac{\delta}{\eps}$.
   Notice that $\frac{\eps X}{V+\eps X}\leq \frac{\delta}{V}=R$, and this can be chosen smaller than $1$ if $\delta>0$ is sufficiently small for each $V>0$.
  
 We now consider $ I_2+I_3$. We notice that 
  \begin{align*} 
  I_2 =&-\int_{X}^{\infty}dz \big[(z+1)^{-\sigma}-z^{-\sigma}\big] (V^{1/3}+(\eps z)^{1/3})^2  \eps^{2-\sigma} \Lambda\left(\frac{\eps X}{V+\eps X}\right)\\&
-\left(\int_{X}^{\infty}dz z^{-\sigma} (V^{1/3}+(\eps z)^{1/3})^2 \right) \eps^{2-\sigma} \Lambda\left(\frac{\eps X}{V+\eps X}\right)=I_2^{(1)}+I_2^{(2)},
  \end{align*} 
 and we study separately $I_2^{(1)}$ and $I_2^{(2)}+ I_3 $. 
  For $0<X\leq\frac{\delta}{\eps}$ we obtain the following estimate for $I_2^{(1)}$:
    \begin{align}  \label{eq:bdI21}
I_2^{(1)} & \leq -C(\sigma) \int_{X}^{\infty} dz  (V^{1/3}+(\eps z)^{1/3})^2 z^{-(\sigma+1)}\left(\frac{X}{V+\eps X}\right)^{\sigma-2}\nonumber \\&
\leq C(\sigma) X^{\sigma-2}( X^{-\sigma}V^{2/3}+\eps^{2/3}X^{-\sigma+2/3})(V+\delta)^{2-\sigma}\nonumber \\&
\leq C(\sigma) X^{-2} (V^{2/3}+\delta^{2/3})(V+\delta)^{2-\sigma}
\leq C(\sigma) X^{-2} (V^{8/3-\sigma}+\delta^{8/3-\sigma}).
  \end{align} 
We now estimate $ I_2^{(2)}+ I_3 $. 
Using the function $H(\xi)$ defined in \eqref{eq:funH} 
we can rewrite $ I_2^{(2)}+I_3 $ as 
  \begin{equation*} 
  I_2^{(2)}+I_3 = K_1+K_2+K_3+K_4+K_5
 \end{equation*}
 where
  \begin{equation*}  
 K_1=     V^{2/3} (V+\eps X)^{2-\sigma} \Big(  -\int_{X}^{\infty}dz z^{-\sigma} X^{\sigma-2} + \int_{0}^{X}dz z^{-\sigma}\big[(X-z)^{\sigma-2}-X^{\sigma-2}\big] \Big),
 \end{equation*}
   \begin{equation*}  
 K_2=- (V+\eps X)^{2-\sigma}   \int_{X}^{\infty}dz z^{-\sigma}(2(\eps V)^{1/3}z^{1/3}+(\eps z)^{2/3}) X^{\sigma-2},
\end{equation*} 
  \begin{equation*} 
K_3=(V+\eps X)^{2-\sigma}  \int_{0}^{X}dz z^{-\sigma}(2(\eps V)^{1/3}z^{1/3}+(\eps z)^{2/3}) \big[(X-z)^{\sigma-2}-X^{\sigma-2}\big],
 \end{equation*} 
  \begin{equation*} 
 K_4= -\int_{X}^{\infty} dz z^{-\sigma}(V^{1/3}+(\eps z)^{1/3})^2\eps^{2-\sigma} H\left(\frac{\eps X}{V+\eps X}\right),
    \end{equation*} 
 \begin{equation*} 
 K_5=  \int_{0}^{X} dz z^{-\sigma}(V^{1/3}+(\eps z)^{1/3})^2\eps^{2-\sigma}\left[H\left(\frac{\eps (X-z)}{V+\eps X}\right)-H\left(\frac{\eps X}{V+\eps X}\right)\right].
 \end{equation*} 
 We have $K_1=0$ due to \eqref{eq:phisigma}.
Moreover, we can obtain the following estimates for $K_2,\dots,K_5$ for $1\leq X\leq \frac{ \delta}{\eps}$. 
   \begin{equation}  \label{eq:bdK23}
 \vert K_2+K_3\vert \leq C(\sigma)\left(\eps^{1/3}V^{1/3} X^{-2/3}+ \eps^{2/3}X^{-1/3}\right).
 \end{equation} 
 Using \eqref{eq:bdH} we have
  \begin{align}  \label{eq:bdK4}
 \vert  K_4\vert &\leq C(\sigma) \int_{X}^{\infty} dz z^{-\sigma}(V^{2/3}+(\eps z)^{2/3})\eps^{1/3} \left(\frac{ X}{V+\eps X}\right)^{\sigma-2+\frac{1}{3}}\nonumber \\&
\leq C(\sigma) \left( V^{7/3-\sigma} \eps^{1/3} X^{-2/3}+ \eps^{8/3-\sigma} X^{5/3-\sigma}\right),
    \end{align}    
Using now \eqref{eq:bdH1} we obtain
 \begin{align}  \label{eq:bdK5}
 \vert  K_5\vert &\leq C(\sigma) \int_{0}^{X} dz z^{-\sigma}(V^{2/3}+(\eps z)^{2/3})\eps^{1/3} \left(\frac{ X}{V+\eps X}\right)^{\sigma-2+\frac{1}{3}}\min\left\{\frac{z}{X},\left(1-\frac{z}{X}\right)^{\sigma-2+\frac{1}{3}}\right\}\nonumber \\&
 \leq C(\sigma)V^{5/3-\sigma} \eps^{1/3}(V^{2/3}+\delta^{2/3}) X^{\sigma-2+\frac{1}{3}}X^{1-\sigma}   \int_{0}^{1} \dzeta \zeta^{-\sigma}\min\{\zeta,(1-\zeta)^{\sigma-2+\frac{1}{3}}\} \nonumber \\&
 \leq C(\sigma) V^{5/3-\sigma}(V^{2/3}+\delta^{2/3})\eps^{1/3} X^{-2/3}.
  \end{align} 
Combining \eqref{eq:bdI21}, \eqref{eq:bdK23}, \eqref{eq:bdK4} and \eqref{eq:bdK5} we get
\begin{align}\label{eq:bdI2I3}
 \vert  I_2^{(2)}+I_{3} \vert &\leq \tilde{\lambda}_0(V,\delta,\sigma)\eps+\tilde{\lambda}_1(V,\delta,\sigma)\eps^{2/3}X^{-1/3}+\tilde{\lambda}_2(V,\delta,\sigma)\eps^{1/3}X^{-2/3}\nonumber \\&
\leq  \lambda_0(V,\delta,\sigma)\eps +\lambda_1(V,\delta,\sigma)\eps^{1/3}X^{-2/3}, \quad 1\leq X\leq\frac{\delta}{\eps},
 \end{align}
where we used Young's inequality in the last step. 
Here the functions $\tilde{\lambda}_k,$ and $\lambda_k$ with $k=0,1,2$, are bounded for $V$ in any compact set of $(0,\infty)$.

We define 
\begin{equation*}
M_1(X):=\left\{\begin{array}{ll}
 \mathcal{U}_{\eps}(X) \quad\text{for}\quad 0<X\leq 1,&\vspace{3mm} \\
I_1+I_2^{(1)} \quad\text{for}\quad X\geq 1,& \vspace{3mm}
\end{array}\right.
\end{equation*}
and 
\begin{equation*}
M_2(X):=\left\{\begin{array}{ll}
0 \quad\text{for}\qquad 0<X\leq 1,&\vspace{3mm} \\
I_2^{(2)}+I_3 \quad\text{for}\quad X\geq 1. & \vspace{3mm}
\end{array}\right.
\end{equation*}

We can now rewrite $L_1$ as 
 \begin{align}  \label{eq:L1}
L_1&= C_{\ast} \int_{0}^{\frac{\delta}{\eps}}dX g(V+\eps X) (V+\eps X)^{\sigma-\frac{8}{3}}M_1(X)+C_{\ast} \int_{0}^{\frac{\delta}{\eps}}dX g(V+\eps X) (V+\eps X)^{\sigma-\frac{8}{3}}M_2(X)\nonumber \\&
=L_1^{(1)}+L_1^{(2)}.
  \end{align}
 Since \eqref{eq:bdUeps}, \eqref{eq:bdI1} and \eqref{eq:bdI21} hold we  can compute the limit $\lim_{\eps\to 0}L_1^{(1)}$ using Lebesgue's dominated convergence theorem. Then
   \begin{align}  \label{eq:limL1}
\lim_{\eps\to 0}L_1^{(1)}=C_{\ast} \int_{0}^{\infty}dX g(V) V^{\sigma-\frac{8}{3}}\lim_{\eps\to 0}M_1(X)  
\end{align}
where
  \begin{align*} 
  \lim_{\eps\to 0}M_1(X)&=V^{8/3-\sigma}\Big[ \int_{0}^{X}dz \left[(z+1)^{-\sigma}-(z)^{-\sigma}\right] \left[  (X-z)^{\sigma-2}-X^{\sigma-2}\right]\\&
\quad -\left(\int_{X}^{\infty}dz \left[(z+1)^{-\sigma}-(z)^{-\sigma}\right]\right) X^{\sigma-2} \Big].
\end{align*}
The terms which do not contain $(z+1)^{-\sigma}$ can be combined and they give a vanishing contribution due to \eqref{eq:phisigma}. Therefore, 
$ \lim_{\eps\to 0}M_1(X)=-V^{8/3-\sigma}Q(X)$ where $Q(X)$ is as in \eqref{def:Q}. Hence, 
   \begin{equation}  \label{eq:limL1bis}
\lim_{\eps\to 0}L_1^{(1)}=-C_{\ast}g(V) \int_{0}^{\infty} Q(X)dX=-g(V)
\end{equation}

Concerning the term $L_1^{(2)}$, using \eqref{eq:bdI2I3} we have
\begin{equation}  \label{eq:limL12}
\underset{\eps\to 0}{\limsup} \vert L_1^{(2)}\vert \leq C(\sigma,V) \|g\|_{\infty} \delta^{1/3}
\end{equation} 
where $ C(\sigma,V)$ is bounded when $V$ is in a compact set of $(0,\infty)$.

We now combine \eqref{eq:Ominfu1}, \eqref{eq:L1L2}, \eqref{eq:limL2}, \eqref{eq:limL1bis} and \eqref{eq:limL12} and this concludes the proof. 

 \end{proof}

\bigskip

We now introduce another auxiliary function $K$ which will play the role of the fundamental solution for the Dirichlet boundary value problem. Given  any  $\bar{V}>0$ and any $\eta> \bar{V}$, we define 
  \begin{align}\label{eq:defK}
&K(V,\eta;\bar{V})=\int_0^{\bar{V}} dV_0 \, G(V,V_0) (\eta-V_0)^{-\sigma} (V_0^{\frac{1}{3}}+(\eta-V_0)^{\frac{1}{3}})^2\quad \text{for}\;\;  V< \bar{V},\nonumber \\&
K(V,\eta;\bar{V})=0\quad \text{for}\;\;V> \bar{V}.
\end{align}
Notice that the integral in the right hand side of \eqref{eq:defK} is well defined due to \eqref{eq:homogpb2}, since $\sigma>\frac{5}{3}$, as well as the fact that $\eta> \bar{V}>V_0$. 

\medskip

We also introduce the auxiliary function
 \begin{equation}\label{eq:defPalp}
 P_{\alpha}(\theta,\zeta)=\theta^{\alpha} \int_{0}^{1-\theta}\Lambda(z) (\zeta(1-z)-\theta)^{-\alpha}dz \qquad \alpha>1,\;\zeta>1, \;\theta<1,
 \end{equation}
 with $\sigma>1$ (note that $\Lambda$ depends on $\sigma$).
In Lemma \ref{lem:Halpha} presented in Appendix \ref{appendix1} we provide some estimates for the function $P_{\alpha}(\theta,\zeta)$ that we will use later.

The following lemma collects some properties of the function $K(V,\eta;\bar{V})$.
\begin{lemma}\label{lem:K}
Let $K(V,\eta;\bar{V})$ be defined as in \eqref{eq:defK}. The following identity holds:
 \begin{equation}\label{eq:Kselfsim}
K(V,\eta;\bar{V})=\frac{C_{\ast}}{\bar{V}}Y(\theta,\zeta) \quad \text{where} \quad \zeta={\eta}/{\bar{V}}, \;\,\theta={V}/{\bar{V}}
\end{equation}
 with $\zeta>1$, $\theta<1$. Here
  \begin{equation}\label{eq:defY}
  Y(\theta,\zeta)=
\frac{1}{\theta} \sum_{j=0}^{2} c_j P_{\sigma-\frac{j}{3}}(\theta,\zeta), \quad c_0=c_2=1,\;c_1=2,
\end{equation}
where $P_{\alpha}(\theta,\zeta)$ is as in \eqref{eq:defPalp}. 
 Moreover, $K(V,\eta;\bar{V})$ satisfies:
 \begin{equation}\label{eq:bdK}
 0\leq K(V,\eta;\bar{V})\leq \frac{ C(\sigma)}{\bar{V}^{\frac{2}{3}}} \frac{\eta^{\frac{2}{3}}}{(\eta-V)} \frac{( \bar{V}-V)^{\sigma-1}} {(\eta-\bar{V})^{\sigma-1}}, \qquad 0<V<\bar{V}<\eta,
 \end{equation}
 \begin{equation}\label{eq:bdderivK}
 \left\vert \frac{\partial K}{\partial V}(V,\eta;\bar{V})\right\vert \leq \frac{C(\sigma)}{\bar{V}^{\frac{2}{3}}}  \left(\frac{\bar{V}}{V}\right)^2 \frac{(\bar{V}-V)^{(\sigma-2)}\eta^{\frac{2}{3}}}{(\eta-\bar{V})^{\sigma}},  \qquad 0<V<\bar{V}<\eta.
 \end{equation}
 \end{lemma}
The proof of the Lemma above is postponed to Appendix \ref{appendix1}.

\bigskip

\medskip

The second step to prove Theorem \ref{th:genbdpb} consists in proving the following lemma.

\begin{lemma}\label{lem:sol2}
Let $G(V,V_0)$ be as in Proposition \ref{prop:existG}. For any $\bar{V}>0$ let $\Psi \in W^{1,\infty}([\bar{V},\infty])$. We define
\begin{align}\label{eq:ffisol2}
&u_2(V)=\int_{\bar{V}}^{\infty} K(V,\eta;\bar{V}) \Psi(\eta) d\eta, \quad \text{for}\quad 0< V< \bar{V},\\&
u_2(V)=\Psi(V)\quad \text{ for} \;\;V> \bar{V}.\label{eq:ffisol2bis}
\end{align}
  The function $K$ is defined as in \eqref{eq:defK}.
 Then, \eqref{eq:ffisol2} and \eqref{eq:ffisol2bis} give a well defined function $u_2 \in L^{\infty}(0,\bar{V})\cap \left(\bigcap_{\delta>0}W^{1,\infty}[\delta,\bar{V}-\delta]\right)$ for any $\delta>0$ which satisfies the following estimates:
\begin{equation}\label{eq:bdu2}
\vert u_2(V)\vert \leq C(\sigma)  \|\Psi\|_{\infty}, \quad\quad\text{for}\quad 0 < V< \bar{V},\quad 
\end{equation}
\begin{equation}\label{eq:bdu2_2}
\left\vert\frac{d u_2}{dV}(V)\right\vert\leq C_{\delta}(\sigma,\bar{V})\left(\|\Psi\|_{\infty}+\|\Psi'\|_{\infty}\right)\quad\quad\text{for}\quad \delta < V< \bar{V}_{\delta}
\end{equation}
\noindent where the constant $C_{\delta}(\sigma,\bar{V})$ depends on $\delta>0$.
 
Then we have
\begin{align}\label{eq:sourcepb}
-\mathcal{K}_{\infty}&u_2(V)=0 \qquad\qquad\text{for}\quad 0< V< \bar{V}. 
\end{align}
\end{lemma}
\bigskip

\begin{proof} 
We first prove the regularity properties of the function $u_2$. We can rewrite $u_2$ as 
  \begin{equation*}
u_2(V) = u_{2,1}(V)+u_{2,2}(V)
\end{equation*}
with
\begin{align*}
&u_{2,1}(V)= \int_{\bar{V}}^{2\bar{V}} K(V,\eta;\bar{V}) \Psi(\eta) d\eta\\&
u_{2,2}(V)=\int_{2\bar{V}}^{\infty} K(V,\eta;\bar{V}) \Psi(\eta) d\eta.
\end{align*}
Then, from \eqref{eq:ffisol2}, using the splitting above and \eqref{eq:bdK}, we obtain
\begin{align*} 
&\phantom{{}\leq{}}\vert u_2(V)\vert\leq \vert u_{2,1}(V)\vert +\vert u_{2,2}(V)\vert \\& 
\leq C(\sigma)\|\Psi\|_{\infty} \Big((\bar{V}-V)^{\sigma-1}\int_{\bar{V}}^{2\bar{V}}(\eta-V)^{-1}(\eta-\bar{V})^{-\sigma+1} \deta +\frac{(\bar{V}-V)^{\sigma-1}}{\bar{V}^{\frac{2}{3}}}\int_{2\bar{V}}^{\infty}\eta^{\frac{2}{3}-\sigma} \deta \Big) \\&
 \leq C(\sigma) \|\Psi\|_{\infty}\left[1+\frac{(\bar{V}-V)^{\sigma-1}}{\bar{V}^{\sigma-1}}\right]\leq C(\sigma) \|\Psi\|_{\infty},
\end{align*} 
where in the third inequality we used the change of variables  $\eta=\bar{V}+(\bar{V}-V)X$.

We now prove \eqref{eq:bdu2_2}.
We estimate first the derivative of $u_{2,2}$. Using \eqref{eq:bdderivK} we obtain
\begin{equation}\label{eq:bdu22}
\left\vert \frac{d u_{2,2}}{d V}(V)\right\vert \leq C(\sigma)\frac{ \|\Psi\|_{\infty} }{\bar{V}^{\frac{2}{3}}}\int_{2\bar{V}}^{\infty} \left(\frac{\bar{V}}{V}\right)^2 \frac{(\bar{V}-V)^{\sigma-2}\eta^{\frac{2}{3}}}{(\eta-\bar{V})^{\sigma}} d\eta \leq
C(\sigma)\left(\frac{\bar{V}}{V}\right)^2 \frac{(\bar{V}-V)^{\sigma-2}}{\bar{V}^{\sigma-1}} \|\Psi\|_{\infty}. 
\end{equation}
 
We consider now the derivative of $u_{2,1}$.  
\begin{align}\label{eq:u21bis}
u_{2,1}(V)&= \frac{C_{\ast}}{\bar{V}}\int_{\bar{V}}^{2\bar{V}} Y\left(\frac{V}{\bar{V}},\frac{\eta}{\bar{V}}\right) \bar{\Psi}\left(\frac{\eta}{\bar{V}}\right) \deta=C_{\ast}\int_{1}^{2} Y(\theta,\zeta) \bar{\Psi}(\zeta) \dzeta \nonumber \\&
=\frac{C_{\ast}}{\theta} \sum_{j=0}^2 c_{j}\int_{1}^{2}\bar{\Psi}(\zeta) P_{\sigma-\frac{j}{3}}(\theta,\zeta)  \dzeta,
\end{align}
where we used that $\Psi(\eta) =\bar{\Psi}({\eta}/{\bar{V}}) $ in the first identity, the change of variable $\zeta={\eta}/{\bar{V}}$, $\theta={V}/{\bar{V}}$ (cf. \eqref{eq:Kselfsim}) in the second and \eqref{eq:defY} in the third.
Moreover, using \eqref{eq:defPalp} with $\alpha=\sigma-\frac{j}{3}$ the last integral in the right hand side of \eqref{eq:u21bis} becomes
 \begin{align*}
& \int_{1}^{2}\dzeta \bar{\Psi}(\zeta) \int_{0}^{1-\theta} \frac{\Lambda(z)}{\big(\frac{\zeta}{\theta}(1-z)-1\big)^{\alpha}} \dz\\&
= \int_{1}^{2} \dzeta \bar{\Psi}(\zeta) \int_{\theta}^{1} \frac{\Lambda(1-w)}{\big(\frac{\zeta}{\theta}w-1\big)^{\alpha}}\dw
=\theta\int_{1}^{2} \dzeta \bar{\Psi}(\zeta) \int_{1}^{\frac{1}{\theta}} \frac{\Lambda(1-\theta X)}{(\zeta X -1)^{\alpha}}\dX \\&
=\theta\int_{1}^{2} \dzeta \bar{\Psi}(\zeta) \int_{1}^{1+\delta_0} \frac{\Lambda(1-\theta X)}{(\zeta X -1)^{\alpha}}\dX
+\theta\int_{1}^{2} \dzeta \bar{\Psi}(\zeta) \int_{1+\delta_0}^{\frac{1}{\theta}} \frac{\Lambda(1-\theta X)}{(\zeta X -1)^{\alpha}}\dX := M_1+M_2,
\end{align*} 
where we used in the first identity  the change of variable $w=1-z$ and, in the second identity, $X=\frac{w}{\theta}$. We chose  $\delta_0>0$ such that $\delta_0\leq \frac{\delta}{2\bar{V}}$. This implies that in the first integral, for $X\leq 1+\delta_0$, $\theta X\leq \theta(1+\delta_0) \leq 1-\delta_0 $, since $\theta<1-\frac{\delta}{\bar{V}}$. Therefore $1-\theta X\geq \frac{\delta_0}{2}$. 
Hence, we have
\begin{align}\label{eq:bdM1}
\left\vert \frac{d M_1}{d \theta}\right\vert &\leq  C_{\delta_0}(\sigma,\bar{V})\|\Psi\|_{\infty}\int_{1}^{2} \dzeta \int_{1}^{1+\delta_0}  \frac{\dX}{(\zeta X -1)^{\alpha}}\nonumber \\&
\leq  C_{\delta_0}(\sigma,\bar{V})\|\Psi\|_{\infty}\int_{1}^{2} \dzeta \frac{1}{\zeta(\zeta -1)^{\alpha-1}}\leq  C_{\delta_0}(\sigma,\bar{V})\|\Psi\|_{\infty},
\end{align}
where $C_{\delta_0}(\sigma,\bar{V})$ is uniformly bounded for $\delta_0>0$ but it might diverge if $\delta_0\to 0$.
We estimate now the derivative of $M_2$. Performing the change of variables $\zeta=\theta s$ and $w=\theta X$, $M_2$ becomes 
\begin{equation*}
M_2=\theta \int_{\frac{1}{\theta}}^{\frac{2}{\theta}} \ds \bar{\Psi}(\theta s) \int_{(1+\delta_0)\theta}^{1} \frac{\Lambda(1-w)}{(sw -1)^{\alpha}}\dw. 
\end{equation*}
Then, we obtain
\begin{align*}
\frac{d M_2}{d \theta} &= \frac{1}{\theta}M_2+ \Big[-\frac{\bar{\Psi}(2)}{\theta} \int_{(1+\delta_0)\theta}^{1}  \frac{\Lambda(1-w)}{(2\frac{w}{\theta} -1)^{\alpha}}\dw+\frac{\bar{\Psi}(1)}{\theta} \int_{(1+\delta_0)\theta}^{1}  \frac{\Lambda(1-w)}{(\frac{w}{\theta} -1)^{\alpha}}\dw \Big] \\&
+\theta^2  \int_{\frac{1}{\theta}}^{\frac{2}{\theta}} \ds \bar{\Psi}'(\theta s) \int_{(1+\delta_0)\theta}^{1} \frac{\Lambda(1-w)}{(sw -1)^{\alpha}}\dw
\\&
-(1+\delta_0)\theta\int_{\frac{1}{\theta}}^{\frac{2}{\theta}} \ds \bar{\Psi}(\theta s)
\frac{\Lambda(1-(1+\delta_0)\theta)}{((1+\delta_0)\theta s -1)^{\alpha}}=:N_1+N_2+N_3+N_4.
\end{align*}
Using the fact that the terms of the form ${1}/{(\cdot)^{\alpha}}$ can be estimated by  $C_{\delta_0}(\sigma,\bar{V})$.  
in the regions of integration, as well as $\theta(1+\delta_0) \leq 1-\delta_0$ and \eqref{eq:bdLAM} we get
\begin{equation*}
\vert N_1 \vert +\vert N_2 \vert+\vert N_4 \vert \leq C_{\delta_0}(\sigma,\bar{V})\|\Psi\|_{\infty}, \quad \vert N_3 \vert \leq C_{\delta_0}(\sigma,\bar{V})\|\Psi'\|_{\infty}
\end{equation*}
which implies
\begin{equation}\label{eq:bdM2}
\left\vert \frac{d M_2}{d \theta}\right\vert \leq C_{\delta_0}(\sigma,\bar{V})\left(\|\Psi\|_{\infty}+\|\Psi'\|_{\infty}\right) \quad \text{for}\; \delta<V<\bar{V}_\delta.
\end{equation}
Therefore, combining \eqref{eq:u21bis}, \eqref{eq:bdM1} and \eqref{eq:bdM2} we obtain \eqref{eq:bdu2_2}.

It only remains to prove \eqref{eq:sourcepb}. Due to \eqref{eq:bdu2} and \eqref{eq:bdu2_2} we have that $\mathcal{K}_{\infty}u_2(V)$ is well defined for $0<V<\bar{V}.$
We notice that \eqref{eq:bdderivK} in Lemma \ref{lem:K} implies that for every $\eta>\bar{V}\,$ the function $K(\cdot,\eta;\bar{V})\in W^{1,\infty}[\delta,\bar{V}-\delta]$, for any $\delta>0$. 
Then, Lemma \ref{lem:sol1} yields
\begin{equation}\label{eq:OminfK}
- \mathcal{K}_{\infty}K(V,\eta;\bar{V})=(\eta-V)^{\sigma}(V^{\frac 1 3}+(\eta-V)^{\frac 1 3})\quad \text{for}\;\; V\in (0,\bar{V}),
 \end{equation}
 where it is always understood that $ \mathcal{K}_{\infty}$ acts on the variable $V$. 
Multiplying  \eqref{eq:OminfK} by $\Psi(\eta)$, integrating in $\eta$ and using  \eqref{eq:ffisol2bis} we obtain 
\begin{equation*} 
\int_{\bar{V}}^{\infty} (\eta-V)^{-\sigma}(V^{\frac 1 3}+(\eta-V)^{\frac 1 3})u_2(\eta)\deta =-\int_{\bar{V}}^{\infty}  \mathcal{K}_{\infty}K(V,\eta;\bar{V})\Psi(\eta)\deta.
 \end{equation*}
Using \eqref{eq:infOmega} as well as \eqref{eq:defK}  we can rewrite the equation above as 
 \begin{align*} 
&-\int_{\bar{V}}^{\infty}\dxi\Psi(\eta)\int_{0}^{\infty}v^{-\sigma}(V^{1/3}+v^{1/3})^2\bigl[\chi_{[0,\bar{V}-V]}(v)K(V+v,\eta;\bar{V})-K(V,\eta;\bar{V})\bigr]\dv\\&
\quad-\int_{\bar{V}}^{\infty} (\eta-V)^{-\sigma}(V^{\frac 1 3}+(\eta-V)^{\frac 1 3})u_2(\eta)\deta =0. 
 \end{align*}
 We can now use Fubini in the first integral since the function $K(w,\eta;\bar{V})$ is uniformly Lipschitz for $w\in [\delta,\bar{V}-{\delta}]$ and $\eta\geq \bar{V}$. Performing also the change of variable $\eta=V+v$ in the second integral, the previous equation is equivalent to:
  \begin{align*} 
\int_{0}^{\infty}\dv v^{-\sigma}(V^{1/3}+v^{1/3})^2 \left(\chi_{[0,\bar{V}-V]}(v) \int_{\bar{V}}^{\infty}\deta\Psi(\eta) K(V+v,\eta;\bar{V})- \int_{\bar{V}}^{\infty}\deta\Psi(\eta)  K(V,\eta;\bar{V})\right)\\
\quad+ \int_{0}^{\infty} v^{-\sigma}(V^{\frac 1 3}+v^{\frac 1 3})\chi_{[\bar{V}-V,\infty]}(v)u_2(V+v)\dv =0 
 \end{align*}
 Using again \eqref{eq:ffisol2bis} in the first integral we  have
 \begin{align*} 
&\int_{0}^{\infty}\dv v^{-\sigma}(V^{1/3}+v^{1/3})^2 \left(\chi_{[0,\bar{V}-V]}(v) u_2(V+v)-u_2(V) \right)\\&
+ \int_{0}^{\infty} v^{-\sigma}(V^{\frac 1 3}+v^{\frac 1 3})\chi_{[\bar{V}-V,\infty]}(v)u_2(V+v)\dv =0
 \end{align*}
 Rearranging the terms in the equation above we obtain $\,\mathcal{K}_{\infty}u_2(V)=0$. 
\end{proof}

\medskip
\begin{proof}[Proof of Theorem \ref{th:genbdpb}]
The existence of the function $u(V)$, under the regularity assumption of Theorem \ref{th:genbdpb} follows from Lemma \ref{lem:sol1} and Lemma \ref{lem:sol2}. Note that the $L^{\infty}$ bounds obtained in Lemma \ref{lem:sol1} and Lemma \ref{lem:sol2} are independent on $\delta$.
In order to prove uniqueness of  $u(V)$ we will use a maximum principle argument. We can assume without loss of generality that $g=0$ and $\Psi=0$. 
Suppose that $\mathcal{K}_{\infty}u(V)=0 $ for $0<V<\bar{V}$ and $u\in  \big(\bigcap_{\delta>0}W^{1,\infty}[\delta,\bar{V}-\delta]\big)$ satisfying \eqref{eq:normunu}. Then, we define the function
\begin{equation}\label{eq:utilde}
\tilde{u}_{\eps}(V)=\eps G(V,\bar{V})-u(V),\qquad \eps>0.
\end{equation}
It is straightforward to show that 
\begin{equation*}
\mathcal{K}_{\infty}\tilde{u}_{\eps}(V)=0 \quad \text{for}\;0< V<\bar{V},\quad \tilde{u}_{\eps}(V)=0 \quad \text{for} \;V>\bar{V}.
\end{equation*}
We now claim that $\tilde{u}_{\eps}(V)>0$ for $0< V<\bar{V}$. 
First notice that $\tilde{u}_{\eps}(V)>0$ for $V\in(\bar{V}_\delta,\bar{V})$ with $\delta=\delta(\eps)>0$ sufficiently small, due to \eqref{eq:normunu} and \eqref{eq:homogpb2}. Suppose that there exists $V\in(0,\bar{V})$ such that $\tilde{u}_{\eps}(V)<0$.  We define $V_n=\inf\{V\in(0,\bar{V}) \,:\,\tilde{u}_{\eps}(w)>0 \;\text{for}\; w\in (V,\bar{V})\}$.  We then have $V_n\in(0,\bar{V})$ and, due to the continuity of the functions $u$ and $G$ for $V\in(0,\bar{V})$, we have  $\tilde{u}_{\eps}(V_n)=0$ and $\tilde{u}_{\eps}(V)>0$ for $V_n<V<\bar{V}$.
 Therefore,
 \begin{equation*}
\mathcal{K}_{\infty}\tilde{u}_{\eps}(V_n)
=\int_{0}^{\infty}v^{-\sigma}({V_n}^{1/3}+v^{1/3})^2\tilde{u}_{\eps}(V_n+v)\dv
 >0,
\end{equation*}
and this gives a contradiction. Hence, $\tilde{u}_{\eps}(V)>0$ for every $V\in(0,\bar{V})$ and $\eps>0$.
 Then, taking the limit $\eps\to 0$ we obtain that $u(V)\leq 0$ for every $V\in(0,\bar{V})$. Replacing now $u$ by $-u$ we also have $u(V)\geq 0$. It follows then $u(V)= 0$ for every $V\in(0,\bar{V})$.
\end{proof}

\begin{remark}
Note that the unique solution $u$ of \cref{eq:genbdpb,eq:genbdpb1} with $g\in L^{\infty}(0,\bar{V})\cap \left(\bigcap_{\delta>0}W^{1,\infty}[\delta ,\bar{V}-\delta]\right)$ and $\Psi \in W^{1,\infty}([\bar{V},\infty])$ is given 
by 
\begin{align}\label{eq:solu}
u(V)=\int_{V}^{\bar{V}} G(V,V_0) g(V_0)dV_0+\int_{\bar{V}}^{\infty} K(V,\eta;\bar{V}) \Psi(\eta) d\eta, \quad \text{for}\quad 0< V< \bar{V}.
\end{align}

\end{remark}

\bigskip


\subsubsection{\texorpdfstring{On the continuity of the function $\Lambda(\xi)$ for $\xi=1$}{On the continuity of the function Lambda(xi) for xi=1} }
\label{Sss.continuity}
In order to conclude the proof of Theorem \ref{Prop:OmINF} we still need to show that the function $\Lambda(\xi)$ is continuous for $\xi=1$. To this end, we will derive a representation formula for $\Lambda(\xi)$ in terms of a power series of $1-\xi$ (compare with \eqref{eq:solLambda}).
From a heuristic point of view, thanks to \eqref{eq:GLamb} and Lemma \ref{lem:sol1}, we can think that the function $\mathcal{G}(V):=G(V,1)=C_{\ast}\Lambda(1-V)$ solves the following problem:
\begin{equation}\label{eq:solGdelta}
-\mathcal{K}_{\infty} G(V,1)=\delta(V-1),\quad G(V,1)=0,\;\,V>1,
\end{equation}
where $\mathcal{K}_{\infty}$ is given by \eqref{eq:infOmega}.
We observe that it could be possible to give a meaning to the problem above in the sense of distributions but we will not attempt to do this in this paper.

 We will  solve \eqref{eq:solGdelta} using Mellin transform. We  look for a solution of \eqref{eq:solGdelta} of the form:%
\[
\mathcal{G}\left(  V\right)  =\frac{1}{2\pi}\int_{-\infty}^{\infty}g(k)V^{ik}dk.
\]
On the other hand, using that
\[
\delta\left(  V-1\right)  =\frac{1}{2\pi}\int_{-\infty}^{\infty}V^{ik}dk
\]
we can formally rewrite \eqref{eq:solGdelta}, exploiting the fact that all the terms in the equation transform similarly under rescalings, as 
\begin{align*}
-\frac{1}{2\pi}\int_{-\infty}^{\infty}dkg(k)V^{ik}\sum_{j=0}^{2}c_{j}\int_{0}^{\infty}\left[  \left(  1+\eta\right)
^{ik}-1\right]  \frac{d\eta}{\eta^{\sigma-\frac{j}{3}}}  &  =\frac{1}{2\pi
}\int_{-\infty}^{\infty}V^{ik}dk.
\end{align*}
Then, identifying Fourier modes on both sides, we obtain
\[
g(k) =- \frac{1}{M(ik)}
\]
where 
\begin{equation}\label{eq:defMik}
M(z)= \sum_{j=0}^{2}c_{j}\int_{0}^{\infty}\left[
\left(  1+\eta\right)  ^{z}-1\right]  \frac{d\eta}{\eta^{\sigma-\frac{j}{3}}} \quad \text{for}\; \Re(z)<0. 
\end{equation}
This suggests that the function $\mathcal{G}(V)$ should be given by
\begin{equation}\label{eq:formulaG_1}
\mathcal{G}(V)=-\frac{1}{2\pi i} \int_{\mathcal{C}_{\alpha}} \frac{V^{z}}{M(z)}\dz\;\text{for }\,0<V<1, \quad \mathcal{G}(V)=0\;\text{for }\,V>1,
\end{equation}
where
\begin{equation*}
\mathcal{C}_{\alpha}=\{-1+re^{-i\alpha}\,:\; 0<r<\infty\}\cup \{-1+re^{i\alpha}\,:\; 0\leq r<\infty\} \qquad 0<\alpha<\frac{\pi}{2}
\end{equation*}
and the orientation of the contour is as in Figure \ref{fig2}.
\begin{figure}[th]
\centering
\includegraphics [scale=0.35]{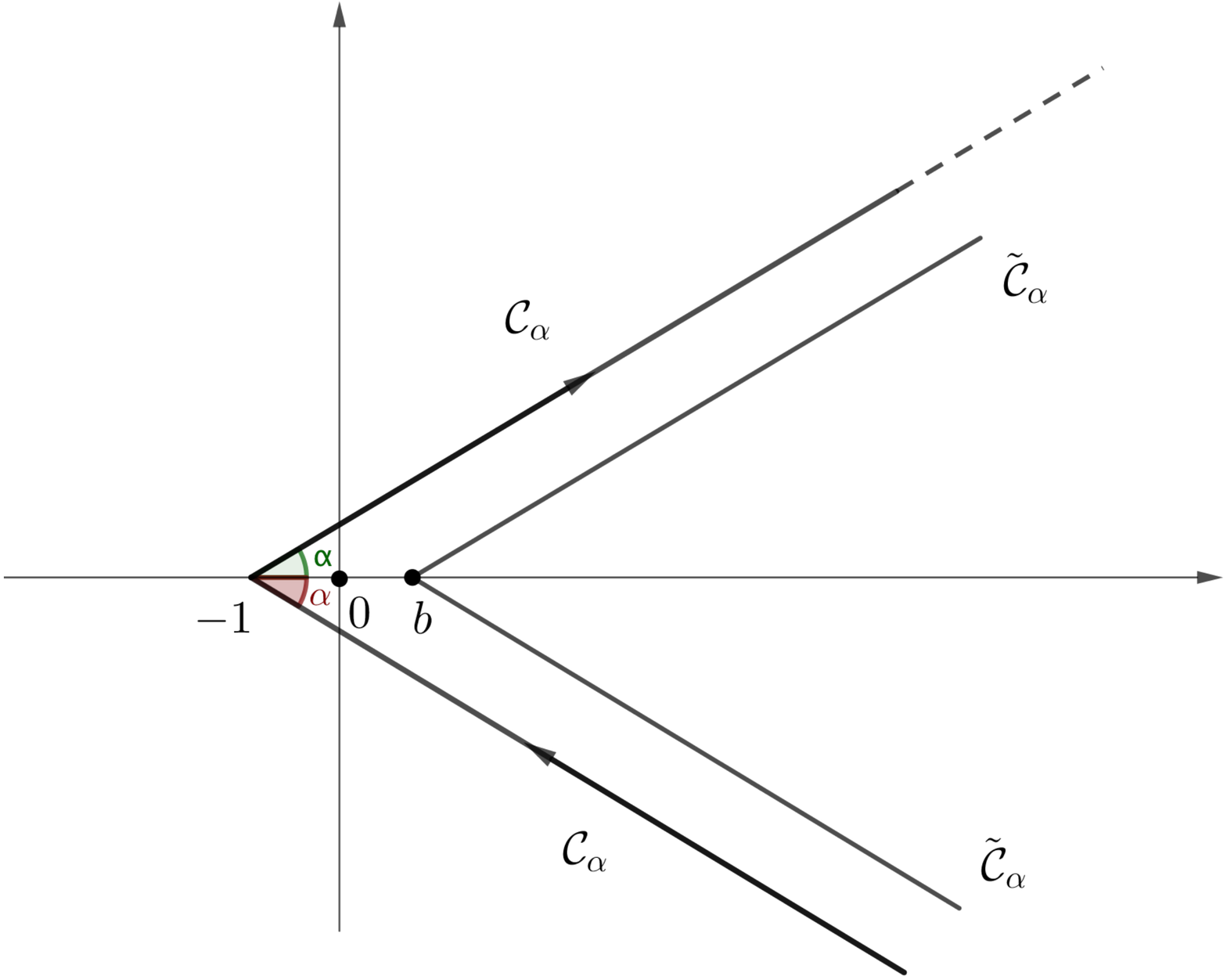}\caption{The orientation of the contour $\mathcal{C}_{\alpha}$. \label{fig2} }%
\end{figure}

The contour $\mathcal{C}_{\alpha}$ has been chosen in such a way that all the roots of $M(z)$ are on the right side of $\mathcal{C}_{\alpha}$.  Notice that in \eqref{eq:formulaG_1} we are assuming that the function $M(z)$ can be extended analytically to the portion of the complex place filled by the contour $\mathcal{C}_{\alpha}$.

In order to show that \eqref{eq:formulaG_1} is well defined we need to study the properties of the function $M(z)$. In order to do this, we state here some technical lemmas whose proofs are given in Appendix \ref{appendix1}. 
\begin{lemma}\label{lem:propM}
The function $M(z)$ defined as in \eqref{eq:defMik} admits a meromorphic extension to $z\in \C$ given by
\begin{equation}\label{eq:extensionM}
M(z)= \frac{z}{\Gamma(1-z)} \sum_{j=0}^{2} c_j \frac{\Gamma(2+\frac{j}{3}-\sigma)\Gamma(\sigma-\frac{j}{3}-1-z)}{\big(\sigma-\frac{j}{3}-1\big)},
\end{equation}
where the coefficients $c_j$ are defined as in \eqref{eq:defY}. 
We denote as $\mathcal{P}_M$ the set of poles for the function $M(z)$ which are given by:
\begin{equation}\label{eq:polesM}
z_{j,n}=\sigma-1-\frac{j}{3}+n\ \ ,\ \ n=0,1,2,\dots \ \ ,\ \ j=0,1,2.
\end{equation}
and as $\mathcal{Z}_M$ the set of zeros for the function $M(z)$ which are given by:
\begin{align}
&\hat{z}_{0,n} = n\quad n=0,1,2,\dots \label{eq:zerosM} \\&
\hat{z}_{1,n}\in\big(\sigma-\frac{5}{3}+n, \sigma-\frac{4}{3}+n\big) \quad n=0,1,2,\dots \qquad  \hat{z}_{1,n}\in (z_{2,n},z_{1,n})\label{eq:zerosM1}\\&
\hat{z}_{2,n} \in\big(\sigma-\frac{4}{3}+n, \sigma-\frac{2}{3}+n\big)\quad n=0,1,2,\dots \qquad  \hat{z}_{2,n}\in (z_{1,n},z_{2,n+1})\label{eq:zerosM2}
\end{align}
where the zeros $\hat{z}_{1,n}$ and $\hat{z}_{2,n}$ are unique in each of the indicated intervals. 

Moreover, for any $\eps_0>0$ we have the following estimate:
\begin{equation}\label{eq:boundMz}
C_1\vert z\vert^{\sigma-\frac{5}{3}} \leq \vert M(z)\vert \leq C_2\vert z\vert^{\sigma-\frac{4}{3}}
\end{equation}
for any $z\in \C$ such that $\text{dist}(z,\mathcal{P}_M\cup \mathcal{Z}_M)\geq \eps_0$, where $C_1,\,C_2$ depend on $ \eps_0$.
\end{lemma}

\begin{lemma}\label{lem:wposasympG} 
The integral in \eqref{eq:formulaG_1} is absolutely convergent for $V\in (0,1)$ and it defines a function $\mathcal{G}\in C[0,1)$ with $\mathcal{G}(0)=\frac{1}{M'(0)}>0$ where 
\begin{equation}\label{eq:Mprime0}
M'(0)= \sum_{j=0}^{2} c_j \frac{\Gamma(2+\frac{j}{3}-\sigma)\Gamma(\sigma-\frac{j}{3}-1)}{\big(\sigma-\frac{j}{3}-1\big)}.
\end{equation}
  The function $\mathcal{G}$ defined in  \eqref{eq:formulaG_1} satisfies 
\begin{equation}\label{eq:solG}
-\mathcal{K}_{\infty} \mathcal{G}(V)=0\quad 0<V<1,\qquad \mathcal{G}(V)=0,\;\,V>1,
\end{equation}
where $\mathcal{K}_{\infty}$ is given by \eqref{eq:infOmega}.

Moreover, the following asymptotics for $\mathcal{G}$ holds:
\begin{equation}\label{eq:asymptoticsG}
\mathcal{G}\left( V\right) = \bar{K}(1-V)^{\sigma-2} = O\left(\vert V-1\vert^{\sigma-\frac{5}{3}} \right) \quad \text{as}\quad V\to 1^{-}
\end{equation}
where
\begin{equation}\label{eq:constantK}
\bar{K}= \frac{(  \sigma-1) \sin(\pi (\sigma-1))  }{\pi }.
\end{equation}
\end{lemma}

\begin{lemma}\label{lem:limLambda}
Let $\Lambda$ be the unique solution of  \eqref{eq:eqLambda} of the form 
\eqref{eq:solLambda} (cf. Lemma \ref{lem:Frobth}). Let $\mathcal{G}$ be as in \eqref{eq:formulaG_1}. Then we have:
\begin{equation}\label{eq:LG}
\Lambda(\xi)=\frac{1}{\bar{K}}\mathcal{G}(1-\xi), \quad 0<\xi<1
\end{equation}
where $\bar{K}$ is as in \eqref{eq:constantK}. Moreover 
\begin{equation}\label{limLam}
\lim_{\xi\to 1^{-}} \Lambda(\xi)=:\Lambda(1^{-})=\big(\bar{K} M'(0)\big)^{-1}>0.
\end{equation}
\end{lemma}

\begin{proof}
Since the function $\Lambda$ solves \eqref{eq:eqLambda}  it follows that the function $\hat{G}(V)$ defined as $ \hat{G}(V):=\Lambda(1-V)$  for $0<V<1$ and $\hat{G}(V):=0$ for $V>1$ solves \eqref{eq:solG}. Therefore,
 $u(V):=\mathcal{G}(V)-\bar{K} \hat{G}(V)$ solves $-\Omega_{\infty}u(V)=0$ for $0<V<1$, $u(V)=0$ for $V>1$ and, using  \eqref{eq:astmLam} and 
\eqref{eq:asymptoticsG}, we have that $u$  satisfies $u(V)=O((1-V)^{\sigma-\frac{5}{3}})$ as $V\to 1^{-}$. Hence, the uniqueness part of Theorem \ref{th:genbdpb} implies \eqref{eq:LG}. Thanks to Lemma \ref{lem:wposasympG} we have \eqref{limLam}.
\end{proof}

\begin{lemma}\label{lem:limitK}
Let $K(V,\eta;\bar{V})$ be defined as in \eqref{eq:defK}. 
 Then we have: 
 \begin{equation}\label{eq:limK}
 \lim_{V\to 0^{+}} K(V,\eta;\bar{V})= C_{\ast}\big(\bar{K} M'(0)\big)^{-1} \sum_{j=0}^2 \frac{c_j}{\zeta(\zeta-1)^{\sigma-\frac{j}{3}}},
 \end{equation} 
 for any $\bar{V}>0$ and $\eta>\bar{V}$ with $\zeta=\frac{\eta}{\bar{V}}$.
 \end{lemma}

\begin{proof}

In order to prove \eqref{eq:limK}, due to Lemma \ref{lem:K}, it is enough to compute the limit $\lim_{\theta\to 0}\frac{1}{\theta}P_{\alpha}(\theta,\zeta)$ for $\zeta>1$ and $\alpha>1$.
We have
\begin{equation*}
\frac{1}{\theta}P_{\alpha}(\theta,\zeta)=\theta^{\alpha-1} \int_{0}^{\frac{1}{2}}\Lambda(z) (\zeta(1-z)-\theta)^{-\alpha}dz +\theta^{\alpha-1} \int_{\frac{1}{2}}^{1-\theta}\Lambda(z) (\zeta(1-z)-\theta)^{-\alpha}dz
\end{equation*}
with $\zeta>1, \;\theta<1$. The first term in the right hand side converges to zero as $\theta\to 0$ for any $\zeta>1$ since $\alpha>1$. Performing the change of variable $z=1-\theta X$  the second  integral becomes
\begin{equation*}
 \int_{1}^{\frac{1}{2\theta}} \Lambda(1-\theta X) (\zeta X-1)^{-\alpha}\dX.
\end{equation*}
Lebesgue's dominated convergence theorem implies 
\begin{equation*}
\lim_{\theta\to 0}\int_{1}^{\frac{1}{2\theta}} \Lambda(1-\theta X) (\zeta X-1)^{-\alpha}\dX=\frac{\Lambda(1)}{(\alpha-1)}\frac{1}{\zeta (\zeta-1)^{\alpha-1}}=\frac{\big(\bar{K}M'(0)\big)^{-1}}{(\alpha-1)}\frac{1}{\zeta (\zeta-1)^{\alpha-1}},
\end{equation*}
where we used Lemma \ref{lem:limLambda}.

From \eqref{eq:defY} using the equation above we have that \eqref{eq:limK} follows.
\end{proof}

\subsubsection{Proof of Theorem \ref{Prop:OmINF}}
Due to Lemma \ref{Lemma1:OmINF}, in order to prove Theorem \ref{Prop:OmINF} it only remains to show that
\begin{itemize}
\item[i)] $\varphi\in C[0,\infty]$;
\item[ii)] $\exists \lim_{V\to 0^{+}}\mathcal{K}_{\infty}\varphi(V)$.
\end{itemize}
Thanks to \eqref{eq:inteqOminf} the condition $ii)$ will follow from $i)$. Therefore, we only need to prove that $\exists \lim_{V\to 0^{+}}\varphi(V)$.

Using Theorem \ref{th:genbdpb} and \eqref{eq:solu} we write $\varphi$ as in \eqref{eq:complsolffi}.
Therefore, the limit $\lim_{V\to 0^{+}}\varphi_2(V)$ exists as a consequence of Lebesgue's dominated convergence theorem, \eqref{eq:bdK} and \eqref{eq:limK}.
 
We now consider $\varphi_1$. Using \eqref{eq:GLamb} we can rewrite $\varphi_1$ as 
\begin{align*}
&\frac{1}{\lambda}\int_0^{\bar{V}} G(V,V_0) (g(V_0)-\varphi(V_0))dV_0=\frac{1}{\lambda}\int_V^{\bar{V}} V_0^{\sigma-\frac{8}{3}} \Lambda\left(\frac{V_0-V}{V_0}\right)(g(V_0)-\varphi(V_0))dV_0\\&
=\frac{1}{\lambda}\int_0^{\bar{V}-V} (X+V)^{\sigma-\frac{8}{3}} \Lambda\left(\frac{X}{X+V}\right)(g((X+V))-\varphi((X+V)))dX
\end{align*}
where we used the change of variable $X=V_0-V$. 
Using \eqref{eq:bdLAM} the integrand above can be bounded by
\begin{align*}
\frac{1}{\lambda} (\|g\|_{\infty}+\|\varphi\|_{\infty}) (X+V)^{\sigma-\frac{8}{3}}\left(\frac{X}{X+V}\right)^{\sigma-2} \leq 
\frac{1}{\lambda} (\|g\|_{\infty}+\|\varphi\|_{\infty})X^{\sigma-\frac{8}{3}}
\end{align*}
which is in $L^1(0,\bar{V})$. Then, the existence of $\lim_{V\to 0^{+}}\varphi_1(V)$ follows from Lebesgue's dominated convergence theorem as well as the continuity of $\varphi(V)$ for  $V>0$ and Lemma \ref{lem:limLambda} which yields the continuity at zero. 
This concludes the proof of Theorem \ref{Prop:OmINF}.

\bigskip

\section{Existence, uniqueness and stability of self-similar profiles}\label{sec:proofs}

\subsection{Proof of Theorem \ref{th:wellpos}} \label{ss:wellposproof}

We first recall the definition of Markov semigroup (cf. \cite{Li}, Definition 1.4).
\begin{definition}
A family of operators $\{S(t), t\geq 0\}$ in $C([0,\infty])$ is a Markov semigroup if $S(0)=I$, the mapping $t\to S(t)f $ from $[0,\infty)$ to $C([0,\infty])$ is right continuous for every $f\in C([0,\infty])$, 
$S(t+s)f=S(t)S(s)f$ for all $f\in C([0,\infty])$ and $s,t\geq 0$,  $S(t)1=1$ for all $t\geq 0$, $S(t)f\geq 0$ for all $f\in C([0,\infty])$ such that $f\geq 0$.
\end{definition}
\begin{proposition}\label{prop:sgrgen}
There exist two Markov semigroups $S_T(t)$, $S_{\infty}(t)$ whose generators are the operators $\bar{\Omega}_T$, $\bar{\Omega}_{\infty}$ which are the closure of the operators ${\Omega}_T$, ${\Omega}_{\infty}$ defined in \eqref{def:opOmegaT} and \eqref{def:opOmegainf} respectively. If $\varphi_0\in  \D(\bar{\Omega}_{\star})$ with $\star=\{T,\infty\}$
we define $\varphi(t)=S_{\star}(t)\varphi_0$. Then, 
\begin{equation}\label{eq:genOm}
\frac{d\varphi}{dt}(t)=\bar{\Omega}_{\star}\varphi(t)\quad \forall \; t\in [0,\infty).
\end{equation}
For every $M>0$ we have $S_{\star}(t)X_M\subset X_M$. Moreover, suppose that $f_0\in \mathcal{M}_{+}([0,\infty])$. We then define $f\in C([0,T];\mathcal{M}_{+}([0,\infty]))$  by means of the duality formula \eqref{duality_volumes} for any $\varphi_0\in C([0,\infty])$ where
$\varphi(\cdot,t)=S_{\star}(t)\varphi_0(\cdot)$. Then $f$ is a weak solution of \eqref{eq:lin:coag0bis} or 
\eqref{eq:lin:coag1} in the sense of Definition \ref{def:fweaksol}.
\end{proposition}

\begin{proof}
We note that the existence of the Markov semigroups $S_{T}(t),$ $S_{\infty}(t)$, as well as \eqref{eq:genOm}, is a consequence of Hille-Yosida theorem
(see for instance Theorem 2.9 in \cite{Li}). 
Furthermore, to prove the invariance of the space $X_M$ under the action of $S_{\star}(t)$ it is sufficient to prove that 
for any $f\in X_M$ then $(I-\lambda \bar{\Omega}_{\star})^{-1}f\in X_M$. This follows from the argument in the proof of Lemma \ref{Lemma1:OmINF}, which shows that $\varphi\in X_M$ if $g\in X_M$, as 
well as the fact that $\bar{\Omega}_{\star}$ is the closure of $\Omega_{\star}$. Combined with 
$S(t)f=\lim_{n\to\infty} (I-\frac{t}{n}\Omega)^{-n}f$ for $f\in C([0,\infty]),\;t\geq 0$
(cf. Theorem 2.9 in \cite{Li}) the invariance of $X_M$ under $S_{\star}(t)$ follows.

It only remains to prove that $f$ is a weak solution of \eqref{eq:lin:coag0bis} or 
\eqref{eq:lin:coag1} in the sense of Definition \ref{def:fweaksol}. In order to do this we consider first test functions $\varphi \in C^{1}([0,T]; C^{\infty}([0,\infty]))$ such that $\varphi\in X_M$ and $\varphi$ constant for $0\leq V\leq \frac{1}{M}$ with $M>1$ large. Notice that the considered test functions are in $\D(\Omega_{\infty})$, which is defined as in \eqref{def:DOmInf}. Therefore, for these test functions we have 
$\bar{\Omega}_{\star}\varphi(V,t)=\mathcal{K}_{\star}\varphi(V,t)$ where $\mathcal{K}_{\star}$ is as in \eqref{eq:infOmega} or \eqref{eq:TOmega} respectively. Then, for each of these test functions  $\varphi$ we define
\begin{equation}\label{eq:testfc}
\xi(V,t)=\partial_{t}\varphi(V,t)+\mathcal{K}_{\star}\varphi(V,t),\qquad \star=\{T,\infty\}.
\end{equation}
Using now \eqref{eq:genOm}, as well as the fact that $\bar{\Omega}_{\star}\varphi(V,t)=\mathcal{K}_{\star}\varphi(V,t)$ and $S_{\star}(t)\bar{\Omega}_{\star}=\bar{\Omega}_{\star}S_{\star}(t)$, for the given test function $\varphi$, we get
\begin{equation}\label{eq:equationS}
S_{\star}(t)\xi(V,t)=\partial_{t}\left(S_{\star}(t)\varphi(V,t)\right).
\end{equation}
The formula $\bar{\Omega}_{\star}\varphi(V,t)=\mathcal{K}_{\star}\varphi(V,t)$ allows us to rewrite \eqref{eq:fweak} as 
\begin{align}\label{eq:dualBis1}
0=\int_{[0,\infty]} f(V,T)\varphi(V,T) dV - \int_{[0,\infty]} f(V,0)\varphi(V,0) \dV 
-\int_{0}^{T}\int_{[0,\infty]}  f(V,t)\xi(V,t)  \dV\dt.
\end{align}
Using the definition of $f(V,t)$ in \eqref{duality_volumes}, the right hand side of \eqref{eq:dualBis1} becomes
\begin{align}\label{eq:dualBis}
\int_{[0,\infty]} f_0(V)\left[S_{\star}(T)\varphi(V,T) - \varphi(V,0) 
-\int_{0}^{T} S_{\star}(t)\xi(V,t) \dt \right] \dV=:J.
\end{align} 
Thanks to \eqref{eq:equationS}, it then follows that $J=0$, whence \eqref{eq:dualBis1} holds for this class of test functions.  A standard density argument yields \eqref{eq:fweak} for arbitrary test functions in $ C^1([0,T]; C^{\infty}([0,\infty])\cap \mathcal{X})$. 
\end{proof}

\begin{proof}[Proof of Theorem \ref{th:wellpos}]
The existence of a weak solution for \eqref{eq:lin:coag0bis} or 
\eqref{eq:lin:coag1} in in the sense of Definition \ref{def:fweaksol} follows from Proposition \ref{prop:sgrgen}. To prove uniqueness we use a standard duality argument. Suppose that $f_1$, $f_2$ are two solutions of \eqref{eq:lin:coag0bis} or 
\eqref{eq:lin:coag1} with the same initial data. Then the difference $g=f_1-f_2$ satisfies
\begin{equation*}
\int_{[0,\infty]} g(V,T)\varphi(V,T) dV =
\int_{0}^{T}\int_{[0,\infty]}  g(V,t)\left(\partial_{t}\varphi(V,t)+\mathcal{K}_{\star}\varphi(V,t)\right)   \dV\dt. 
\end{equation*} 
If the test function $\varphi(\cdot,t)\in \D(\bar{\Omega}_{\star})\cap X_M$ it follows from a density argument
\begin{equation*}
\int_{[0,\infty]} g(V,T)\varphi(V,T) dV =
\int_{0}^{T}\int_{[0,\infty]}  g(V,t)\left(\partial_{t}\varphi(V,t)+\bar{\Omega}_{\star}\varphi(V,t)\right)   \dV\dt. 
\end{equation*} 
If  $g(V,t)\neq 0$, we have $\int_{[0,\infty]} g(V,T)\varphi_0(V) dV \neq 0$ for an appropriate test function $\varphi_0(V)\in \D(\bar{\Omega}_{\star})\cap X_M$.
We then define $\varphi(V,t)=S_{\star}(T-t)\varphi_0(V)$. Using then \eqref{eq:genOm} we obtain $\partial_t \varphi(V,t)+\bar{\Omega}_{\star}\varphi(V,t)=0$ with $\varphi(V,t)\in \D(\bar{\Omega}_{\star})\cap X_M$.  Then, $\int_{[0,\infty]} g(V,T)\varphi_0(V) dV = 0$ and the contradiction implies $g=0$.

To prove that $f(\{\infty\},t)=0$ for $t>0$ if $f_0(\{\infty\})=0$  we will show that for any $\eps_0>0$ there exists $V_0$ sufficiently large such that 
\begin{equation}\label{eq:finfty}
\int_{[V_0,\infty]} f(V,t)\dV\leq \int_{[\frac{V_0}{2},\infty]} f_0(V)\dV+\eps_0.
\end{equation}
The fact that $f(\{\infty\},t)=0$ then follows taking the limit $V_0\to \infty$ and later $\eps_0\to 0$. We will assume without loss of generality that $\int_{[0,\infty]} f_0=1.$ 
In order to prove \eqref{eq:finfty} we will construct for any $T>0$ any $\eps_0>0$, and any $V_0$ sufficiently 
large a function $\varphi \in C^{1}([0,T);W^{1,\infty}[0,\infty]\cap \mathcal{X})$ such that 
\begin{equation}\label{eq:cdtffi1}
\varphi(V,T)\geq \chi_{[V_0,\infty]}(V), \quad \varphi(V,0)\leq \chi_{[\frac{V_0}{2},\infty]}(V)+\eps_0
\end{equation}
 and 
\begin{equation}\label{eq:cdtffi2}
\partial_{t}\varphi(V,t) +\mathcal{L}\varphi(V,t):=\partial_{t}\varphi(V,t) +\int_{[0,\infty]} G_{\star}(v) (V^{\frac{1}{3}}+v^{\frac{1}{3}})^2\big(\varphi(V+v,t)-\varphi(V,t) \big)  dv \leq 0,
\end{equation}
for a.e. $(V,t)\in (0,\infty) \times (0,T)$.
Notice that a density argument implies that \eqref{eq:fweak} holds for this class of test functions and together with the assumptions for the function $\varphi$ it implies
\begin{equation*}
\int_{[V_0,\infty]} f(V,t)\dV\leq\int_{[0,\infty]} f(V,t)\varphi(V,T)\dV\leq \int_{[0,\infty]} f_0(V)\varphi(V,T)\dV\leq \int_{[\frac{V_0}{2},\infty]} f_0(V)\dV+\eps_0,
\end{equation*} 
whence \eqref{eq:finfty} follows. 
We now claim that the following function satisfies \eqref{eq:cdtffi1} and \eqref{eq:cdtffi2}. We define
\begin{align}\label{eq:solineq}
&{\varphi}\left(  V,t\right)  =1-e^{-\lambda s} Q\left(\xi\right),\quad \xi=1-W-\lambda s
\end{align}
with $s=(T-t)V_0^{-\sigma+\frac 5 3}$, $W=\frac{V}{V_0}$,and  $Q$ given by
\begin{equation}\label{def:auxQ}
Q(\xi):=\left\{\begin{array}{ll}
0\quad\text{if}\quad \xi< 0,&\vspace{3mm} \\
4\xi \quad\text{if}\quad 0\leq \xi\leq  \frac 1 4,& \vspace{3mm}\\
1\quad\text{if}\quad \xi\geq  \frac 1 4
\end{array}\right.
\end{equation}
where $\lambda>0$ is a constant independent on $V_0$ that will be fixed later. 
In order to prove \eqref{eq:cdtffi1} we observe that since $s\leq T{V_{0}^{-\sigma+\frac{5}{3}}}$ if $V_0>0$ is sufficiently large then $s$ can be assumed to take arbitrary small values. Using then \eqref{eq:solineq} we obtain that $\varphi\leq1$, $\varphi\left(  V,T\right)\geq\chi_{[V_0,\infty]}(V) $ and $\varphi\left(  V,0\right)\leq \eps_0$ for $V\leq \frac{V_0}{2}$. Therefore \eqref{eq:cdtffi1} follows.
It now remains to show that \eqref{eq:cdtffi2} holds.  
Suppose first that $W>1-\lambda s$. Then $\varphi(V,t)=\varphi(V+v,t)=1$, $v>0$. Hence $\partial_t \varphi(V,t)=0$, $\mathcal{L}\varphi(V,t)=0$ and \eqref{eq:cdtffi2} follows. We now consider the case $W<1-\lambda s$. Then $\xi>0$ and
\begin{equation}\label{eq:dtffi}
\partial_t \varphi(V,t)=-\lambda e^{-\lambda s}V_0^{-(\sigma-\frac{5}{3})}\big(Q(\xi)+Q'(\xi)\big)
\end{equation}
and 
\begin{equation*}
\mathcal{L}\varphi(V,t)=e^{-\lambda s}\int_{[0,\infty]}G_{\star}(v)(V^{\frac{1}{3}}+v^{\frac{1}{3}})^2[Q(\xi)-Q(\xi-\frac{v}{V_0}) ]\dv.
\end{equation*}
Using \eqref{ass:G} and the change of variables $V=V_0 W$ and $v=V_0 w$ and the fact that $W<1-\lambda s\leq 1$ we obtain 
\begin{align}\label{eq:Lffi}
\mathcal{L}\varphi(V,t)\leq& C_0 V_0^{-(\sigma-\frac{5}{3})} e^{-\lambda s}\int_{[0,\infty]}(w)^{\sigma}(1+w^{\frac{1}{3}})^2 [Q(\xi)-Q(\xi-w) ]\dw\nonumber  \\&
\leq C_0 V_0^{-(\sigma-\frac{5}{3})} e^{-\lambda s} \left[ 4\int_{[0,1]}(w)^{\sigma}(1+w^{\frac{1}{3}})^2 w \dw+\int_{[1,\infty]}(w)^{\sigma}(1+w^{\frac{1}{3}})^2 \dw \right]\nonumber \\&
=\tilde{C}\,C_0 V_0^{-(\sigma-\frac{5}{3})} e^{-\lambda s} 
\end{align}
where  we used that $0\leq [Q(\xi)-Q(\xi-w) ] \leq \min\{4w,1\}$, and $\tilde{C}$ depends only on $\sigma$. Combining now  \eqref{eq:dtffi} and \eqref{eq:Lffi} and using that $\big(Q(\xi)+Q'(\xi)\big)
\geq 1$ for any $\xi>0$ we obtain
\begin{equation}
\partial_t \varphi(V,t)+\mathcal{L}\varphi(V,t)\leq \big(-\lambda+\tilde{C}\,C_0  \big)V_0^{-(\sigma-\frac{5}{3})} e^{-\lambda s} \leq 0
\end{equation}
if we choose $\lambda>0$ sufficiently large independent on $V_0$. Then \eqref{eq:cdtffi2} follows and this concludes the proof. 
\end{proof}

\subsection{Existence and uniqueness of self-similar profiles } \label{Ss.ex}

In this subsection we will prove Theorem \ref{th:existencesssol}. 
Some of the methods used in this section have been extensively used in the analysis of the nonlinear Smoluchowski equation. For instance the existence of self-similar profiles using  fixed point methods can be found in \cite{ EMR, GPV, KV16, NTV, NV, JNV}.
The idea of relying on the analysis of the dual  problem to obtain estimates for the  self-similar solutions of the coagulation equations has been introduced in \cite{NV}.

We consider the evolution equation~\eqref{eq:lin:coag} in self-similar variables which reads by means of~\eqref{eq:self:sim:2} as
\begin{equation}\label{eq:self:sim:evolution}
 \del_{\tau}F(\xi,\tau)-\mu\del_{\xi}\bigl(\xi F(\xi,\tau)\bigr)+\del_{\xi}\biggl(\int_{0}^{\xi}\int_{\xi-x}^{\infty}y^{-\sigma}(x^{1/3}+y^{1/3})^2 F(x,\tau)\dy\dx\biggr)=0.
\end{equation}
Our goal is to prove the existence and uniqueness of a stationary solution of \eqref{eq:self:sim:evolution}. 

In order to do this we first introduce the concept of weak solution of \eqref{eq:self:sim:evolution}.
\begin{definition}[Definition of weak solutions]\label{def:weaksol}
We will say that $F\in C([0,T],\M_{+}[0,\infty])$ is a weak solution of \eqref{eq:self:sim:evolution} with initial value $F(\cdot,0)=F_0\in\M_{+}[0,\infty]$ if for any 
test function $\varphi \in C([0,T], C^1[0,\infty])\cap \mathcal{X}$ 
 we have
\begin{align}\label{eq:self:sim:evolution:weak}
&\int_{[0,\infty]} F(\xi, T)\varphi(\xi,T )\dxi-\int_{[0,\infty]} F_0(\xi)\varphi(\xi,0)\dxi-\int_{0}^{T}\dt \int_{[0,\infty]}F(\xi,t)\partial_t\varphi(\xi,t)\dxi \nonumber \\&
 =-\mu\int_{0}^{T}\dt \int_{[0,\infty]}\xi F(\xi,t)\partial_{\xi}\varphi(\xi,t)\dxi \nonumber \\&\quad +\int_{0}^{T}dt \int_{0}^{\infty}\int_{0}^{\infty}y^{-\sigma}(x^{1/3}+y^{1/3})^2F(x,t)\bigl[\varphi(x+y,t)-\varphi(x,t)\bigr]\dx\dy.
\end{align}
\end{definition}
\medskip

The well-posedness Theorem \ref{th:wellpos} allows us to prove a well-posedness result for \eqref{eq:self:sim:evolution}. More precisely, we have
\begin{theorem}\label{th:wellpos:selfsim}
For any $F_0(\cdot)\in \mathcal{M}_{+}([0,\infty])$ there exists a unique solution of \eqref{eq:self:sim:evolution} in the sense of Definition \ref{def:weaksol}. Moreover, if $F_0(\{\infty\})=0$ then $F(\{\infty\},\tau)=0$ for any $\tau\geq 0$.
\end{theorem}
\begin{proof}
Due to the invariance under translations in time of \eqref{eq:lin:coag1} we can apply Theorem \ref{th:wellpos} to find a weak solution of \eqref{eq:lin:coag1} with initial datum $f(V,0)=f_0(V)=F_0(V)$ for any 
$f_0(\cdot)\in \mathcal{M}_{+}([0,\infty])$ in the sense of Definition \ref{def:fweaksol}. Let us denote this solution as $f(\cdot,\cdot)\in C([1,\infty);\mathcal{M}_{+}([0,\infty]))$. We then define 
\begin{equation}\label{eq:Ff}
F(\xi,\tau)=e^{\mu\tau}f(\xi e^{\mu\tau},e^{\tau}-1), \quad \tau>0.
\end{equation} Notice that  $F(\cdot,\cdot)\in C([0,\infty);\mathcal{M}_{+}([0,\infty]))$. A direct computation shows that  $F$ solves \eqref{eq:self:sim:evolution} in the sense of Definition \ref{def:weaksol}. Reciprocally, given $F$ satisfying \eqref{eq:self:sim:evolution} in the sense of Definition \ref{def:weaksol} we can define $f$, by means of $f(V,t)=\frac{1}{t^{\mu}}F\left(\frac{V}{t^{\mu}},\log(t+1)\right)$, which solves \eqref{eq:lin:coag1} in the sense of Definition \ref{def:fweaksol}. Therefore, the uniqueness statement for the solution of \eqref{eq:self:sim:evolution} follows from the corresponding uniqueness result in Theorem \ref{th:wellpos}.
\end{proof}

 \begin{remark}\label{def:opT}
 We now define a family of operators $\{T(\tau)\}_{\tau\geq 0}$ with $T(\tau):\M_{+}([0,\infty])\to \M_{+}([0,\infty])$ as follows. 
 Given $f_0\in \M_{+}([0,\infty])$ we define $f(V,t)$ as in \eqref{duality_volumes}. We then define $F(\xi,\tau)$ as in \eqref{eq:Ff} and we set $T(\tau)f_0=F(\cdot,\tau)$. 
We note that $T(0)=I$, $T(\tau+s)F=T(\tau)T(s)F$ for all $F\in \M_{+}([0,\infty])$ and $s,\tau\geq 0$,  $T(\tau)1=1$ for all $\tau\geq 0$. Moreover, 
 the family of operators $\{T(\tau)\}_{\tau\geq 0}$ is obtained combining the definition of $f$ given in \eqref{duality_volumes} with the definition of $F$ 
 in \eqref{eq:Ff}. Using then the properties of the semigroup $S_{\infty}(t)$ given in Proposition \ref{prop:sgrgen} it follows that the operators $T(\tau)$ are continuous in the $\star-$weak topology of $\M_{+}([0,\infty])$ for every $\tau\geq 0$ and the mapping $\tau \to T(\tau)F $ is continuous for every $F\in \M_{+}([0,\infty])$ and $\tau\in [0,\infty)$.
 \end{remark}


We now want to prove the boundedness for some moments of the solution $F$ of  \eqref{eq:self:sim:evolution:weak}.
 We first observe that, choosing $\varphi\equiv 1$ in \eqref{eq:self:sim:evolution:weak}, we obtain immediately that the zeroth moment is conserved, i.e.\@
\begin{equation}\label{eq:conservation:of:mass}
 \int_{[0,\infty]}F(\xi,\tau)\dxi=\int_{[0,\infty]}F_{0}(\xi)\dxi=1\quad \text{for all }\tau>0.
\end{equation}

\begin{lemma} 
For any  $\beta\in(0,\sigma-5/3)$, $\gamma>0$ and any $R>0$ we define 
\begin{equation}\label{eq:defSRgb}
S_R(\beta,\gamma):=\left\{H \in \M_{+}([0,\infty])\;: \; \max\left\{ \int_{0}^{\infty}(1+V)^{\beta}  H(V) \dV ,\int_{0}^{\infty}V^{-\gamma} H(V) \dV\right\} \leq R\right\}.
\end{equation}
Then, there exists $R_0(\beta,\gamma)$ such that if $R\geq R_0(\beta,\gamma)$ we have $T(\tau)S_R(\beta,\gamma) \subset S_R(\beta,\gamma)$ for every $\tau\in [0,\infty)$.
\end{lemma}
\begin{proof}

We consider the weak formulation \eqref{eq:self:sim:evolution:weak} and we choose as test function 
\begin{equation*}
 \varphi(x)=(1+x)^{\beta}, \quad \beta>0,
\end{equation*}
and define
\begin{equation}\label{eq:defMbeta}
 M_{\beta}(\tau)\vcc=\int_{0}^{\infty}(1+\xi)^{\beta} F(\xi,\tau)\dxi .
\end{equation}

Using \eqref{eq:self:sim:evolution:weak} and the fact that $ \xi \varphi'(\xi)=\beta \xi (1+\xi)^{\beta-1}\leq \beta (1+\xi)^{\beta}=\beta\varphi(\xi)$ we have that  $M_{\beta}=M_{\beta}(\tau)$ satisfies
\begin{align}\label{eq:estMbeta}
 \del_{\tau}M_{\beta}+\beta\mu M_{\beta}&\leq \int_{0}^{\infty}\int_{0}^{\infty}y^{-\sigma}(x^{1/3}+y^{1/3})^2 F(x)\bigl[(1+x+y)^{\beta}-(1+x)^{\beta}\bigr]\dx\dy \nonumber
\\& \leq \int_{0}^{\infty}\int_{0}^{\infty}y^{-\sigma}\bigl((1+x)^{1/3}+y^{1/3}\bigr)^2 F(x)(1+x)^{\beta}\biggl[\Bigl(1+\frac{y}{1+x}\Bigr)^{\beta}-1\biggr]\dy\dx.
\end{align}  
We now change variables $y\mapsto (1+x)y$. Then, \eqref{eq:estMbeta} becomes 
\begin{equation}\label{eq:moment:est:1}
 \del_{\tau}M_{\beta}+\beta\mu M_{\beta}\leq \int_{0}^{\infty}(1+x)^{5/3-\sigma+\beta} F(x)\dx\int_{0}^{\infty}y^{-\sigma}\bigl[(1+y)^{\beta}-1\bigr](1+y^{1/3})^2\dy.
\end{equation}
Since $5/3<\sigma<2$, we can estimate for each $\beta\in (0,\sigma-5/3)$ the integral in $y$ on the right-hand side of \eqref{eq:moment:est:1} by a constant $C_{\beta,\sigma}$, i.e.\@
$ \int_{0}^{\infty}y^{-\sigma}\bigl[(1+y)^{\beta}-1\bigr](1+y^{1/3})^2\dy\leq C_{\beta,\sigma}.$
Similarly, for $\beta\in (0,\sigma-5/3)$ it holds $(1+x)^{5/3-\sigma+\beta}\leq 1$. Then
$
 \int_{0}^{\infty}(1+x)^{5/3-\sigma+\beta} F(x)\dx\leq \int_{0}^{\infty} F(x)\dx=1.
$
Together with~\eqref{eq:moment:est:1} this implies
\begin{equation*}
 \del_{\tau}M_{\beta}+\beta\mu M_{\beta}\leq C_{\beta,\sigma}.
\end{equation*}
Therefore, the region $M_{\beta}\leq \frac{C_{\beta,\sigma}}{\beta\mu}$ is invariant under the operator $T(\tau)$.

We now choose as test function $\varphi(\xi)=(\xi+\eps)^{-\gamma}$ with  $\gamma>0$ and $\eps>0$, in \eqref{eq:self:sim:evolution:weak}. We define  
\begin{equation}\label{eq:defmgamma}
 m_{\gamma}(\tau)\vcc=\int_{0}^{\infty}(\xi+\eps)^{-\gamma} F(\xi,\tau)\dxi, \qquad \gamma>0,\;\;\eps>0.
\end{equation}
Using \eqref{eq:self:sim:evolution:weak} we get:
\begin{align}\label{eq:moment:01}
\del_{\tau}m_{\gamma}&
\leq \gamma\mu\,m_{\gamma}-\int_{0}^{\infty}\dx F(x) \int_{0}^{\infty}y^{-\sigma}(x^{1/3}+y^{1/3})^2\bigl[(x+\eps)^{-\gamma}-(x+y+\eps)^{-\gamma}\bigr]\dy,
\end{align}
where we used that $\frac{\xi}{\xi+\eps}\leq 1$. 
Performing the change of variables $y\mapsto (x+\eps)y$ we get 
\begin{equation}\label{eq:negmom}
\del_{\tau}m_{\gamma}\leq\gamma\mu\, m_{\gamma}-\int_{0}^{\infty} F(x)(x+\eps)^{5/3-\sigma-\gamma}\dx\int_{0}^{\infty}y^{-\sigma}
\Big(\Big(\frac{x}{x+\eps}\Big)^{1/3}+y^{1/3}\Big)^2\bigl[1-(1+y)^{-\gamma}\bigr]\dy.
\end{equation}
We set
\begin{equation*}
 C_{0}\vcc=\int_{0}^{\infty}y^{-\sigma} y^{2/3} \bigl[1-(1+y)^{-\gamma}\bigr]\dy.
\end{equation*}
Then  
\begin{equation*}
\int_{0}^{\infty}y^{-\sigma}
\Big(\Big(\frac{x}{x+\eps}\Big)^{1/3}+y^{1/3}\Big)^2\bigl[1-(1+y)^{-\gamma}\bigr]\dy \geq C_0.
\end{equation*}
Therefore \eqref{eq:negmom} becomes
\begin{equation*}
  \del_{\tau}m_{\gamma}\leq \gamma \mu \, m_{\gamma}-C_{0}\int_{0}^{\infty} F(x)(x+\eps)^{5/3-\sigma-\gamma}\dx.
\end{equation*}
Splitting the integral on the right-hand side of the equation above we get
\begin{equation*}
 \begin{split}
  \del_{\tau}m_{\gamma}&\leq\gamma\mu\, m_{\gamma}-C_{0}\int_{0}^{\rho} F(x)(x+\eps)^{5/3-\sigma-\gamma}\dx-C_{0}\int_{\rho}^{\infty} F(x)(x+\eps)^{5/3-\sigma-\gamma}\dx\\
  &\leq\gamma\mu\, m_{\gamma}-C_{0}(\rho+\eps)^{5/3-\sigma}\int_{0}^{\rho} F(x)(x+\eps)^{-\gamma}\dx,
 \end{split}
\end{equation*}
where $\rho>0$ will be determined later and where we used that $\int_{\rho}^{\infty} F(x)(x+\eps)^{5/3-\sigma-\gamma}\dx \geq 0$. Moreover,  the integral in the last inequality can be estimated as  
\begin{equation*}\begin{split}
&(\rho+\eps)^{5/3-\sigma}\int_{0}^{\infty} F(x)(x+\eps)^{-\gamma}\dx-(\rho+\eps)^{5/3-\sigma}\int_{\rho}^{\infty}F(x)(x+\eps)^{-\gamma}\dx\\&
\geq (\rho+\eps)^{-(\sigma-5/3)}m_{\gamma}-(\rho+\eps)^{-(\sigma-5/3+\gamma)},
 \end{split}\end{equation*}
since $(x+\eps)^{-\gamma}\leq (\rho+\eps)^{-\gamma}$ and $m_{0}=1$. Hence, we obtain
\begin{equation*} 
\del_{\tau}m_{\gamma}
 \leq \gamma \mu \, m_{\gamma}-C_{0}(\rho+\eps)^{-(\sigma-5/3)} m_{\gamma}+C_0(\rho+\eps)^{-(\sigma-5/3+\gamma)}.
\end{equation*}

Since $\sigma>5/3$, 
we can take $\rho>0$ sufficiently small, but independent on $\eps$, such that 
$\gamma\mu-C_{0}(\rho+\eps)^{5/3-\sigma}<-C_{0}/2$ which leads to the estimate
\begin{equation*}
 \del_{\tau}m_{\gamma}\leq -\frac{C_{0}}{2}m_{\gamma}+C_{0}(\rho+\eps)^{5/3-\sigma-\gamma}.
\end{equation*}
Then, the region $m_{\gamma}\leq 2(\rho+\eps)^{5/3-\sigma-\gamma}$ is invariant under the operator $T(\tau)$. Taking the limit $\eps\to 0$ in this inequality, the result follows.
\end{proof}

\bigskip 

\begin{proof}[Proof of Theorem \ref{th:existencesssol}]
The existence of a self-similar profile $F$, as indicated in Theorem \ref{th:existencesssol}, follows from the Tychonoff fixed-point theorem 
applied on the set $S_R(\beta,\gamma)$ defined in \eqref{eq:defSRgb}.

The fact that $F\in C^{\infty}(0,\infty)$ follows from a bootstrap argument using that the left hand side of \eqref{eq:weakss2} contains one weak derivative of $F$ and the right hand side contains $\sigma-1<1$ derivatives. It is then possible to improve inductively the regularity of $F$ following the strategy proposed in \cite{KV16}.  To prove uniqueness we notice that, given any self-similar profile as stated in Theorem \ref{th:existencesssol}, the measure $f(V,t)=\frac{1}{t^{\mu}}F(\frac{V}{t^{\mu}})$ is a weak solution of \eqref{eq:lin:coag1}  with initial data $f(V,0)=\delta(V)$ in the sense of Definition \ref{def:fweaksol}. 
Theorem \ref{th:wellpos} implies that such a solution is unique and this implies the uniqueness of the self-similar profile.

\end{proof}

\subsection{Stability of self-similar profiles} \label{Ss.stability}

In this section we prove Theorem \ref{th:stabilityssprof}. 
To this end we will use the Trotter-Kurtz theorem (cf. Theorem~2.12 in~\cite{Li}).
\begin{theorem}\label{Thm:Trotter:Kurtz}
 Suppose that $\bar{\Omega}_{T}$ and $\bar{\Omega}_{\infty}$ are the generators of the Markov semigroups $S_{T}(t)$ and $S(t)$ respectively. If there is a core $\D_0$ for $\bar{\Omega}_{\infty}$ such that $\D_0\subset \D(\Omega_{T})$ for all $T$ and $\bar{\Omega}_{T} \varphi \to \bar{\Omega}_{\infty} \varphi$ for all $\varphi \in \D_0$, then
 \begin{equation*}
  S_{T}(t) \varphi\to S(t) \varphi
 \end{equation*}
 for all $\varphi\in C([0,\infty])$ uniformly for $t$ in compact sets.
\end{theorem}

\begin{proof}[Proof of Theorem \ref{th:stabilityssprof}]
We first claim that $\D_0$ defined by means of $\D_0:=\mathcal{X} \cap C^{\infty}([0,\infty])$ is a core for the operator  $\bar{\Omega}_{\infty}$. Indeed, the closure of $\D_0$ in the uniform topology is $C([0,\infty])$. On the other hand we claim $\D_0\subset \D(\Omega_{\infty})\subset \D(\bar{\Omega}_{\infty})$ where $\D(\Omega_{\infty})$ is as in \eqref{def:DOmInf}. Indeed, the only property that we need to check is that there exists $\lim_{V\to 0}(\Omega_{\infty}\varphi)(V)$. 
Indeed, using that $\abs*{v^{-\sigma}(V^{1/3}+v^{1/3})^2[\varphi(V+v)-\varphi(V)]}\leq C_{\varphi} v^{-\sigma}(V^{1/3}+v^{1/3})^2\min\{v,1\}$ we can use Lebesgue's theorem in  \eqref{eq:infOmega}, \eqref{def:opOmegainf} to obtain
 \begin{equation*}
  \lim_{V\to 0}\Omega_{\infty}\varphi(V)=\int_{0}^{\infty}v^{2/3-\sigma}[\varphi(v)-\varphi(0)]\dv .
 \end{equation*}
 
 We now claim that  
  \begin{equation}\label{eq:conver}
  \bar{\Omega}_{T} \varphi \to \bar{\Omega}_{\infty} \varphi,\quad \forall \;\varphi \in \D_0.
  \end{equation} 
as $T\to\infty$.  Suppose that $\varphi \in X_M$. In order to prove \eqref{eq:conver}, we consider
 \begin{align*}
\big(\Omega_{T} \varphi- \Omega_{\infty} \varphi\big)(V)&=\int_0^{\infty}\big(G_{T}(v)-v^{-\sigma}\big)(V^{1/3}+v^{1/3})^2\bigl[\varphi(V+v)-\varphi(V)\bigr]\dv \\&
=\int_{V}^{\infty} \dxi \varphi'(\xi)\int_{\xi-V}^{\infty} \big(G_{T}(v)-v^{-\sigma}\big)(V^{1/3}+v^{1/3})^2 \dv.
\end{align*}
Note that we used $\varphi(V+v)=-\varphi(V)=\int_{V}^{V+v}\varphi'(\xi)\dxi$ together with Fubini's theorem in the second line. Then, since 
$$\int_{\xi-V}^{\infty} \big(G_{T}(v)-v^{-\sigma}\big)(V^{1/3}+v^{1/3})^2 \dv\leq \frac{C}{(\xi-V)^{\sigma-1}}\left(\chi_{\{\xi-V\leq \delta(T) \}}\right) +\varepsilon (T),$$
with $\delta(T)\to 0$ and $\varepsilon(T)\to 0$ as $T\to\infty$,  
we get 
\begin{align*}
\vert \big(\Omega_{T} \varphi- \Omega_{\infty} \varphi\big)(V)\vert &\leq
\|\varphi' \|_{\infty} \int_{V}^{M} \dxi \varphi'(\xi)\frac{C}{(\xi-V)^{\sigma-1}}\left(\chi_{\{\xi-V\leq \delta(T) \}}\right) +\varepsilon (T)\\&
\leq C_{M}\|\varphi' \|_{\infty}\big[ \delta(T)^{2-\sigma}+\varepsilon(T)\big]\rightarrow 0 \quad \text{as}\quad T\to\infty,
\end{align*}
for $V\in [0,M]$, $0<M<\infty$.

We now prove \eqref{eq:ssconv}. 
The solution $f(V,t)$ of \eqref{eq:lin:coag0bis} obtained in Theorem \ref{th:wellpos} is given by the duality formula \eqref{duality_volumes} where $\varphi(V,t)=S(t)\varphi_0(t)$ with $S(t)=S_1(t)$.

We define a new measure $F_T\in \mathcal{M}_{+}([0,\infty])$ by means of $F_T(W,\tau):=T^{\mu}f(T^{\mu}W,T \tau)$ for $\tau\geq 0,\; W\geq 0$. Suppose that we consider test functions $\varphi_0$ in \eqref{duality_volumes} of the form $\varphi_0(V)=\tilde{\varphi_0}\left(\frac{V}{T^{\mu}}\right)$. Then the left hand side of \eqref{duality_volumes} becomes $ \int_{[0,\infty]}\tilde{\varphi_0}(W)F_T(W,\tau)dW $. On the other hand, using the definition of $\mathcal{K}_{T}$ that appears in \eqref{eq:TOmega} we obtain
$\varphi(T^{\mu}W,T \tau)=S_{T}(\tau)\tilde{\varphi_0}(W)$. We also define $F_{0,T}(W):=T^{\mu}f_0(T^{\mu}W)$. Therefore,  \eqref{duality_volumes} becomes
\begin{equation}\label{eq:rescaledduality}
\int_{[0,\infty]}\tilde{\varphi_0}(W)F_T(W,\tau)dW= \int_{[0,\infty]}S_{T}(\tau)\tilde{\varphi_0}(W)F_{0,T}(W)dW .
\end{equation}
We first have that $F_{0,T}(W)\rightharpoonup \delta(W) $ as $T\to\infty$ in the weak topology of measures. On the other hand, due to \eqref{eq:conver}, we can apply the Trotter-Kurtz theorem (see. Theorem \ref{Thm:Trotter:Kurtz}) which implies that $S_{T}(\tau)\tilde{\varphi_0}(W)\to S_{\infty}(\tau)\tilde{\varphi_0}(W)$ as $T\to\infty$ in $C([0,\infty])$. Then, the right hand side of \eqref{eq:rescaledduality} converges to 
$\int_{[0,\infty]} S_{\infty}(\tau)\tilde{\varphi_0}(W)\delta(W) dW$. Using now the duality \eqref{duality_volumes} for the equation \eqref{eq:lin:coag1} we get
$$\int_{[0,\infty]}\tilde{\varphi_0}(W)F_T(W,\tau)dW\to \int_{[0,\infty]} F_{\infty}(W,\tau)\tilde{\varphi_0}(W)dW\quad \text{as}\;\; T\to \infty$$
where $F_{\infty}(W,\tau)$ is the unique solution of \eqref{eq:lin:coag1} in the sense of Definition \ref{def:fweaksol} with the property $F_{\infty}(W,0)=\delta(W)$. 
Due to Theorem \ref{th:existencesssol} we have that $F_{\infty}(W,\tau)=\frac{1}{\tau^{\mu}}F\left(\frac{W}{\tau^{\mu}}\right)$ with $F$ the self-similar profile solution of \eqref{eq:weakss2}. 
It then follows that 
$$T^{\mu}f(T^{\mu}W,T)=F_T(W,1) \to F(W)\quad \text{as}\;\; T\to \infty$$
and this yields the result.

\end{proof}

\bigskip

\textbf{Acknowledgements.} B.N, A.N. and J.J.L.V. acknowledge support through the
CRC 1060 \textit{The mathematics of emergent effects }of the University of
Bonn that is funded through the German Science Foundation (DFG). S.T.\@ has been supported by a Lichtenberg Professorship grant (awarded to C. Kühn) of the VolkswagenStiftung.

\bigskip


\appendix
\section{Technical Lemmas} \label{appendix1}

In this Appendix we collect the proofs of some technical results which are used in the proofs of the
main results of the paper in Sections \ref{Sss.fundamental}, \ref{Sss.continuity}.
 We first prove the two computational lemmas (cf. Lemma \ref{lem:phiab} and Lemma \ref{lem:omega}) 
used in the construction of the solution of the form \eqref{eq:solLambda}.

\begin{proof}[Proof of Lemma \ref{lem:phiab}]
In order to prove \eqref{eq:betafc0} we first observe that the formula holds for $Re(\beta)>0$. 
This can be seen using the definition \eqref{def:Dalpha}, the fact that $\eta^{-(\alpha+1)}=-\frac{1}{\alpha}\frac{d}{d\eta}(\eta^{-\alpha})$ and integrating by parts. Furthermore, the validity of \eqref{eq:betafc0} for $\text{Re}(\beta)>-1$ is due to the fact that the function
\begin{equation*}
\Phi_{\alpha}(\beta)=\int_{0}^{1} \eta^{-(\alpha+1)}\bigl[(1-\eta)^{\beta}-1\bigr]d\eta
\end{equation*}
is analytic for $\text{Re}(\beta)>-1$. Using that $\beta \Gamma(\beta)$ is analytic for $\text{Re}(\beta)>-1$ and $\frac{1}{\Gamma(1-\alpha+\beta)}$ is analytic in the whole complex plane, we obtain that  \eqref{eq:betafc0} holds for $\text{Re}(\beta)>-1$.

Finally, \eqref{eq:phisigma} follows from the fact that $\frac{1}{\Gamma(1-\alpha-\beta)}=0$ if $\alpha=\sigma-1$ and $\beta=\sigma-2$.
\end{proof}

\medskip

\begin{proof}[Proof of Lemma \ref{lem:omega}]
Using that 
$$B(x,y)=\frac{\Gamma(x)\Gamma(y)}{\Gamma(x+y)}$$
and \eqref{eq:omega} we have 
\begin{equation*}
\omega(\ell,j,m;\sigma)
= \frac{\big(\sigma-2+\frac{j}{3}+m\big)}{\big(\sigma-1-\frac{\ell}{3}\big)}\frac{\Gamma(2-\sigma+\frac{\ell}{3})\Gamma(\sigma-2+\frac{j}{3}+m)}{\Gamma(\frac{\ell}{3}+\frac{j}{3}+m)}.
 \end{equation*}

Using that $\Gamma(z)$ is decreasing for $z\in (0,1)$, 
as well as $\ell\in\{0,1,2\}$, we get
\begin{equation*}
\frac{\Gamma\big(\frac{5}{3}-\sigma\big)}{\big(\sigma-1\big)}\frac{\Gamma(\sigma-1+\frac{j}{3}+m)}{\Gamma(\frac{\ell}{3}+\frac{j}{3}+m)}\leq\omega(\ell,j,m;\sigma) \leq 
\frac{\Gamma(2-\sigma)}{\big(\sigma-\frac{5}{3}\big)}\frac{\Gamma(\sigma-1+\frac{j}{3}+m)}{\Gamma(\frac{\ell}{3}+\frac{j}{3}+m)},
 \end{equation*}
 for $ \ell+j+m\geq 1$. 
Then, \eqref{eq:boundomega} follows from Stirling's formula.
\end{proof}

\bigskip

\bigskip 

We now present the lemma used in the proof of Proposition \ref{prop:existG} concerning the positivity of the auxiliary function $Q(w)$ defined by \eqref{def:Q}. 
 
Furthermore, we prove the two lemmas which provide some estimates for the function $\Lambda(\xi)$ obtained in Lemma \ref{lem:Frobth} (cf. Lemma \ref{lem:bdLamda} and Lemma \ref{lem:estdiffLam}) 
which are used in order to prove Proposition \ref{prop:existG} and Theorem \ref{th:genbdpb}.

\begin{lemma}\label{lem:posQ}
Let $Q(w)$ be the function defined as in \eqref{def:Q} with $\sigma\in (1,2)$.  Then $Q(w)>0$ for any $w>0$.  Moreover, 
\begin{equation}\label{eq:bdQ}
0<Q(w) \leq \frac{1}{(\sigma-1)}\frac{1}{w^{\sigma-2}},\quad \text{for}\;\; w\leq 1
\end{equation}
and for every $0<\eps<2-\sigma$, there exists a constant $\bar{C}_{\eps}(\sigma)>0$  such that
\begin{equation}\label{eq:bdQ_1}
0<Q(w) \leq \frac{\bar{C}_{\eps}(\sigma)}{w^{3-\sigma-\eps}}, \quad \text{for}\;\; w\geq 1.
\end{equation}
\end{lemma}
\begin{remark}
Note that \eqref{eq:bdQ}, \eqref{eq:bdQ_1} imply that $Q(\cdot)\in L^1(0,\infty)$.
\end{remark}

\begin{proof}[Proof of Lemma \ref{lem:posQ}]
We observe that the upper bound in \eqref{eq:bdQ} follows immediately from \eqref{def:Q} using the negativity of the last term in the formula.  
We consider
\begin{align}\label{eq:Q1}
Q(w)&=  \frac{1}{(\sigma-1)}  w^{\sigma-2}-\int_0^{w} (z+1)^{-\sigma} \left(w-z\right)^{\sigma-2} \dz\nonumber\\&
 =\frac{1}{(\sigma-1)}  w^{\sigma-2}-\int_0^{w} (z+1)^{-\sigma}{w}^{\sigma-2}\dz
 -\int_0^{w} (z+1)^{-\sigma} \big[\left(w-z\right)^{\sigma-2}-w^{\sigma-2}\big] \dz \nonumber\\&
 =\frac{w^{\sigma-2}}{(\sigma-1)(w+1)^{\sigma-1}} - \int_0^{w} (z+1)^{-\sigma} \big[\left(w-z\right)^{\sigma-2}-w^{\sigma-2}\big] \dz.\nonumber 
\end{align}

Using the change of variables $z=w \xi$ we obtain
\begin{align*}
Q(w)= \frac{1}{w} q\left(\frac{1}{w}\right)
\end{align*}
with
\begin{equation}\label{def:q}
q(y)=\frac{1}{(\sigma-1)}\frac{1}{(1+y)^{\sigma-1}} - \int_0^{1} (\xi+y)^{-\sigma} \big[(1-\xi)^{\sigma-2}-1\big] d\xi,
\end{equation}
where $y=\frac{1}{w}$. 
We claim that
\begin{equation}\label{eq:ineq}
(\xi+y)^{-\sigma}<\xi^{-\sigma}(1+y)^{-(\sigma-1)} \quad \text{for}\quad y>0,\quad 0\leq \xi\leq 1
\end{equation}
which follows from the fact that 
$\xi (1+y)^{\beta}< \xi(1+y) \leq \xi+y$ for $\beta=\frac{\sigma-1}{\sigma}$, $y>0$ and $0\leq \xi\leq 1$.

Using the inequality \eqref{eq:ineq} with $\sigma<2$, we obtain that
\begin{equation*}
q(y)> \frac{1}{(1+y)^{(\sigma-1)}} \left(\frac{1}{(\sigma-1)}-\int_0^1 \xi^{-\sigma} \big[ (1-\xi)^{\sigma-2}-1\big]d\xi \right)=0
\end{equation*}
where the last equality follows from  \eqref{eq:phisigma}. 

We now prove \eqref{eq:bdQ_1}. 
We have
 \begin{align*}
Q(w)&=\int_0^{\infty} (z+1)^{-\sigma}  w^{\sigma-2}\dz -\int_0^{w} (z+1)^{-\sigma} \left(w-z\right)^{\sigma-2} \dz\\&
= \int_0^{w} (z+1)^{-\sigma}\big(w^{\sigma-2}-\left(w-z\right)^{\sigma-2} \big)\dz +w^{\sigma-2}\left(\int_{w}^{\infty} (z+1)^{-\sigma} \dz\right) \\&
= \int_0^{w} \left(\frac{1}{(z+1)^{\sigma}}-\frac{1}{z^{\sigma}}\right)\left(w^{\sigma-2}-\left(w-z\right)^{\sigma-2} \right)\dz \\&+ \int_0^{w} \frac{1}{z^{\sigma}}
\left(  w^{\sigma-2}-\left(w-z\right)^{\sigma-2} \right)+\frac{1}{(\sigma-1)}\frac{w^{\sigma-2}}{\left(w+1\right)^{\sigma-1}}\\&=
J_1+J_2+J_3,
\end{align*} 
where we used the change of variable $z=w\zeta$ and
\begin{align*}
& J_1:=\frac{1}{w} \int_{0}^{1}\left(\frac{1}{(\zeta+w^{-1})^{\sigma}}-\frac{1}{\zeta^{\sigma}}\right)\left( 1-(1-\zeta)^{\sigma-2}\right)  \dzeta\\&
J_2:=\frac{1}{w} \left[\int_{0}^{1} \frac{1}{\zeta^{\sigma}}\left( 1-(1-\zeta)^{\sigma-2}\right)\dzeta +\frac{1}{(\sigma-1)}\right]\\&
J_3:=\frac{1}{(\sigma-1)}\frac{1}{w}\left[ \frac{1}{(1+w^{-1})^{\sigma-1}}-1\right].
\end{align*}
We notice that $J_2=0$ thanks to \eqref{eq:phisigma} and using Taylor's theorem we have that  $|J_3|\leq \frac{\bar{C}(\sigma)}{w^2}$ for $w\geq 1$.
Finally, using again Taylor's theorem we obtain
\begin{align}\label{eq:bdJ1}
|J_1|&\leq \frac{\bar{C}(\sigma)}{w} \int_{0}^{1} \frac{(1-\zeta)^{\sigma-2}}{\zeta^{\sigma-1}} \left\vert \frac{1}{(1+(\zeta w)^{-1})^{\sigma}}-1 \right\vert \dzeta.
 \end{align}

We now claim that the following inequality holds for any $0<\eps<2-\sigma$ and $X>0$:
\begin{equation}\label{eq:eq:bdJ1bis}
\left\vert 1-\frac{1}{(1+(X)^{-1})^{\sigma}} \right\vert \leq \frac{\bar{C}(\sigma)}{X^{2-\sigma-\eps}}.
\end{equation}
This inequality can be proved studying separately the cases $X>1$ and $X\leq 1$. 
Therefore \eqref{eq:bdJ1} yields
\begin{align*}
|J_1|&\leq \frac{\bar{C}(\sigma)}{w^{3-\sigma-\eps}} \int_{0}^{1} \frac{(1-\zeta)^{\sigma-2}}{\zeta^{1-\eps}} \dzeta =\frac{\bar{C}_{\eps}(\sigma)}{w^{3-\sigma-\eps}}.
 \end{align*}
 Then \eqref{eq:bdQ_1} follows.
\end{proof}

\medskip

\begin{proof}[Proof of Lemma \ref{lem:bdLamda}]
Due to the representation formula \eqref{eq:solLambda} we have that $\Lambda(\cdot)\in C^{\infty}(0,1)$.
We now claim that $\Lambda(\xi)$ is strictly positive for $\xi\in (0,1).$ This follows from a maximum principle argument. Due to \eqref{eq:astmLam} we have that $\Lambda(\xi)>0$ for any $\xi \in (0,\varepsilon_0),$ with $\varepsilon_0>0$.

Suppose that $\Lambda(\xi)$ is not strictly positive for $\xi\in(0,1)$. Then, due to the continuity of $\Lambda$ in $(0,1)$, as well as  \eqref{eq:astmLam}, there exists $\xi_0\in(0,1)$ such that $\Lambda(\xi_0)=0$ and $\Lambda(\xi)>0$ for $0<\xi<\xi_0$. 
Therefore, \eqref{eq:eqLambda} implies that
\begin{equation}\label{eq:Lambda_3}
D^{\sigma-1}_{+}\Lambda(\xi_0)+\sum_{\ell=1}^{2}F_{\ell}(\xi_0)D^{\sigma-1-\frac{\ell}{3}}_{+}\Lambda(\xi_0)=\sum_{\ell=0}^{2}F_{\ell}(\xi_0)D^{\sigma-1-\frac{\ell}{3}}_{+}\Lambda(\xi_0)=:\mathcal{J}=0,
\end{equation}
where we set for the sake of convenience $F_0(\xi)=1$. On the other hand, using \eqref{def:Dalpha}, as well as $\Lambda(\xi_0)=0$, 
we have
$$
\mathcal{J}= \sum_{\ell=0}^{2}F_{\ell}(\xi_0)\int_{0}^{\xi_0}\eta^{-(\sigma-\frac{\ell}{3})}\Lambda(\xi_0-\eta)\deta.$$
Notice that the right hand side of the formula above is well defined since $\Lambda(\cdot)\in C^{\infty}(0,1)$ and $\Lambda(\xi_0)=0$.  Using that $F_{\ell}(\xi)>0$ for any $\xi\in(0,1)$ and $\Lambda(\xi)>0$ for any $\xi \in (0,\xi_0)$ it then follows that $\mathcal{J}>0$. This contradicts \eqref{eq:Lambda_3}. Hence
\begin{equation}\label{eq:liminf}
\Lambda(\xi)>0\quad \text{for any}\quad \xi\in (0,1)
\end{equation}

In order to prove that $\underset{\xi\to 1^{-}}{\limsup} \,\Lambda(\xi)<\infty$ we construct a suitable supersolution for \eqref{eq:eqLambda}.
We look for a supersolution of the form
\begin{equation}\label{eq:supersol}
\tilde{\Lambda}(\xi)= (1+\varepsilon)\xi^{\sigma-2}.
\end{equation} 
with $\varepsilon>0$ arbitrary.
More precisely we will check that $\mathcal{L}(\tilde{\Lambda})(\xi)\geq 0$ for any $\xi\in (0,1)$ where $\mathcal{L}(\tilde{\Lambda})$ is defined as in \eqref{eq:defLLambda}. 

First notice that \eqref{eq:phisigma} implies 
$$-D^{\sigma-1}_{+}\tilde{\Lambda}(\xi)+ \frac{\xi^{-(\sigma-1)}}{(\sigma-1)}\tilde{\Lambda}(\xi)=0, \quad \xi>0$$
then
\begin{equation*}
\mathcal{L}(\tilde{\Lambda})=-(1+\varepsilon)\sum_{\ell=1}^{2}F_{\ell}(\xi)\left[ D^{\sigma-1-\frac{\ell}{3}}_{+}\big(\xi^{\sigma-2}\big)-\frac{\xi^{-(\sigma-1-\frac{\ell}{3})}}{(\sigma-1-\frac{\ell}{3})}\big(\xi^{\sigma-2}\big)\right], \quad \xi\in (0,1).
\end{equation*}
By using \eqref{eq:betafc} with $\alpha=\sigma-1-\frac{\ell}{3}$, $\beta=\sigma-2$, and the identity $z\Gamma(z)=\Gamma(z+1)$, we obtain
\begin{align*}
D^{\sigma-1-\frac{\ell}{3}}_{+}\big(\xi^{\sigma-2}\big)&=\frac{1}{\big(\sigma-1-\frac{\ell}{3}\big)}\left[ 1-\frac{\Gamma(2-\sigma+\frac{\ell}{3})\Gamma(\sigma-1)}{\Gamma\big(\frac{\ell}{3}\big)}\right]\xi^{\frac{\ell}{3}-1}, \quad \xi>0.
\end{align*} 
Then
 \begin{align*}
\mathcal{L}(\tilde{\Lambda})
=(1+\varepsilon)\sum_{\ell=1}^{2}\frac{\Gamma(2-\sigma+\frac{\ell}{3})\Gamma(\sigma-1)}{(\sigma-1-\ell/3)\Gamma\big(\frac{\ell}{3}\big)}F_{\ell}(\xi) \xi^{\frac{\ell}{3}-1},\quad \xi\in (0,1). 
\end{align*}

Moreover, since $\frac{5}{3}<\sigma<2$ we have that $2-\sigma+\frac{\ell}{3}>0$ and $\sigma-1>0$. Using that $\Gamma(z)>0$ for $z>0$ as well as $F_{\ell}(\xi)>0$ for $\xi\in (0,1)$ we obtain 
\begin{equation}\label{eq:LtildeLam}
\mathcal{L}(\tilde{\Lambda})(\xi)>0\quad\text{for} \quad \xi\in (0,1).
\end{equation} 
We now argue again using the maximum principle. We set 
$$\mathcal{W}(\xi)=\tilde{\Lambda}(\xi)-\Lambda(\xi).$$
Using \eqref{eq:eqLambda} and \eqref{eq:LtildeLam} we have that $\mathcal{L}(\mathcal{W})(\xi)>0$ for $\xi\in (0,1)$. Moreover, due to \eqref{eq:astmLam} and \eqref{eq:supersol} it follows that $\mathcal{W}(\xi)>0$ for $\xi\in(0,\varepsilon_0)$ with $\varepsilon_0>0.$ Then, following the same strategy as in the proof of the positivity of $\Lambda(\xi)$ we will obtain that   $\mathcal{W}(\xi)>0$ for $\xi\in(0,1)$. Otherwise there would exist a $\xi_0\in (0,1)$ such that $\mathcal{W}(\xi_0)=0$ and $\mathcal{W}(\xi)>0$ for $\xi\in(0,\xi_0)$. Then the inequality $\mathcal{L}(\mathcal{W})(\xi_0)>0$ would imply
$$\mathcal{J}=\sum_{\ell=0}^{2}F_{\ell}(\xi_0)D^{\sigma-1-\frac{\ell}{3}}_{+}\mathcal{W}(\xi_0)<0.$$
However, the definition of the operator $D^{\sigma-1-\frac{\ell}{3}}_{+}$ given by \eqref{def:Dalpha} yields $\mathcal{J}>0$. This contradiction implies that $\mathcal{W}(\xi)>0$  for any $\xi\in (0,1)$. Hence $ \Lambda(\xi)\leq \tilde{\Lambda}(\xi)$ for any $\xi\in (0,1)$. Therefore
\begin{equation}\label{eq:limsup}
 \Lambda(\xi) \leq (1+\varepsilon) \xi^{\sigma-2}.
\end{equation}
Combining \eqref{eq:liminf} and \eqref{eq:limsup} and taking the limit $\varepsilon\to 0$ we obtain that $\Lambda(\xi)$ satisfies \eqref{eq:bdLAM}. 
\end{proof}

\medskip

 \begin{proof}[Proof of Lemma \ref{lem:estdiffLam}]
We now prove \eqref{eq:estdiffLam}. 
Suppose first that $0<\zeta\leq 1/2$. 
Using \eqref{eq:solLambda} we can write 
\begin{equation*}
\Lambda(\xi(1-\zeta))-\Lambda(\xi)=\sum_{j=0}^{2}\sum_{m=0}^{\infty} a_{m,j} \,\xi^{\sigma-2+\frac{j}{3}+m} \big[(1-\zeta)^{\sigma-2+\frac{j}{3}+m}-1\big].
\end{equation*} 
Since 
\begin{equation*}
\big\vert(1-\zeta)^{\sigma-2+\frac{j}{3}+m}-1\big\vert=\left(1-e^{(\sigma-2+\frac{j}{3}+m)\log(1-\zeta)}\right) \leq C(\sigma)\, m\zeta  \quad \text{for}\;\; m\geq 1,
\end{equation*} 
and using \eqref{eq:boundcoeff} with
 $\delta>0$ such that $\theta:=R(1+\delta)<1$, we then obtain
 \begin{align}\label{eq:diffLam1}
\big\vert\Lambda(\xi(1-\zeta))-\Lambda(\xi)\big\vert&\leq  \xi^{\sigma-2}\left(\sum_{m=1}^{\infty} \sum_{j=0}^{2} C(1+\delta)^m R^{\frac{j}{3}+m} \abs[\big]{(1-\zeta)^{\sigma-2+\frac{j}{3}+m}-1} +C\zeta\right)\nonumber \\&
\leq  \xi^{\sigma-2} \left(C\zeta+3\sum_{m=1}^{\infty} Cm(1+\delta)^m R^{m}\zeta\right)\leq  C(R,\sigma) \xi^{\sigma-2}\zeta.
\end{align}

We now consider the case $\frac{1}{2}\leq \zeta<1$. Using \eqref{eq:bdLAM} we have:
\begin{equation}\label{eq:diffLam2}
\vert \Lambda(\xi(1-\zeta))-\Lambda(\xi)\vert \leq C(\sigma)\xi^{\sigma-2} \big[1+(1-\zeta)^{\sigma-2}\big],\quad 0< \xi<1.
\end{equation} 
\eqref{eq:estdiffLam} follows combining \eqref{eq:diffLam1} and \eqref{eq:diffLam2}.

The proof of \eqref{eq:bdH}, \eqref{eq:bdH1} is similar. Note that
\begin{equation*}
H(\xi)=\sum_{j=0}^{2}\sum_{m=0}^{\infty}\tilde{a}_{m,j} \,\xi^{\sigma-2+\frac{j}{3}+m}, \quad a_{0,0}=1.
\end{equation*}
with $\tilde{a}_{m,j}=0$ for $m=j=0$ and $\tilde{a}_{m,j}={a}_{m,j}$ otherwise. 
Then \eqref{eq:bdH} follows using the series representation above and \eqref{eq:bdLAM}. The proof of \eqref{eq:bdH1} can be made analogously to the proof of \eqref{eq:estdiffLam}.
\end{proof}

\bigskip

\bigskip

In the following Lemma we provide some estimates for the function $P_{\alpha}(\theta,\zeta)$ defined by \eqref{eq:defPalp}. This auxiliary function has been introduced in Lemma \ref{lem:K} 
in order to prove Lemma \ref{lem:sol2} which is the second step in the proof of Theorem \ref{th:genbdpb}.

\begin{lemma}\label{lem:Halpha}
Let the function $P_{\alpha}(\theta,\zeta)$ be defined as in \eqref{eq:defPalp}. 
The following estimates hold:
\begin{align}\label{eq:bdPalpha}
0\leq P_{\alpha} (\theta,\zeta) \leq \frac{C\theta}{(\zeta-\theta)}  \frac{(1-\theta)^{\sigma-1}}{(\zeta-1)^{\alpha-1}}  
\end{align}
and 
\begin{align}\label{eq:bdderivPalpha}
\left\vert \frac{\partial P_{\alpha}}{\partial \theta} (\theta,\zeta)\right\vert \leq  \frac{C
}{\theta}\frac{(1-\theta)^{\sigma-2}}{(\zeta-1)^{\alpha}}  \quad \text{for}\;\, \theta<1.
\end{align}

\end{lemma}

\begin{proof}
We have
\begin{align}\label{eq:bdHalp_bis}
0\leq P_{\alpha}(\theta,\zeta)&\leq \theta^{\alpha} \int_{0}^{1-\theta} z^{\sigma-2} \big(\zeta-\theta-z\zeta \big)^{-\alpha}dz \nonumber \\&
= \theta^{\alpha} \left(\frac{(\zeta-\theta)^{\sigma-\alpha-1} }{\zeta^{\sigma-1} } \right) \int_{0}^{\frac{\zeta(1-\theta)}{(\zeta-\theta)}} X^{\sigma-2}(1-X)^{-\alpha} dX
\end{align}
where we used the change of variable $z=\frac{(\zeta-\theta)}{\zeta}  X$.

 Since 
\begin{equation}\label{eq:bdw}
\int_{0}^{w} X^{\sigma-2}(1-X)^{-\alpha} dX\leq \frac{w^{\sigma-1}}{(1-w)^{\alpha-1}}, 
\end{equation}
for $0<w<1$. We can now use \eqref{eq:bdw} in \eqref{eq:bdHalp_bis} and \eqref{eq:bdPalpha} follows. 
Moreover, we have
 \begin{align*}
\frac{\partial  P_{\alpha}}{\partial \theta}(\theta,\zeta)&=-\theta^{\alpha}\Lambda(1-\theta)\big(\zeta\theta-\theta \big)^{-\alpha}+\frac{\alpha }{\theta^{2}} \int_{0}^{1-\theta}  \Lambda(z) \zeta(1-z) \big(\frac{\zeta}{\theta}(1-z)-1 \big)^{-(\alpha+1)} dz. 
\end{align*}
We notice that we can estimate the last integral as $\frac{\alpha }{\theta^{2}} \zeta P_{\alpha+1}(\theta,\zeta)$.
Therefore, 
\begin{align*}
\left\vert\frac{\partial  P_{\alpha}}{\partial \theta}(\theta,\zeta)\right\vert &\leq \Lambda(1-\theta)\big(\zeta-1\big)^{-\alpha}+\frac{\alpha }{\theta^{2}} \zeta P_{\alpha+1}(\theta,\zeta) \\&
\leq \frac{C (1-\theta)^{\sigma-2}}{\big(\zeta-1\big)^{\alpha}} \left[ 1+\frac{\zeta}{\theta}\frac{(1-\theta)}{(\zeta-\theta)}\right]\leq  \frac{C (1-\theta)^{\sigma-2}}{\theta \big(\zeta-1\big)^{\alpha}} \end{align*}
which yields \eqref{eq:bdderivPalpha}.
\end{proof}

\medskip

We now prove Lemma \ref{lem:K} which collects some properties of the function $K(V,\eta;\bar{V})$.

\begin{proof}[Proof of Lemma \ref{lem:K}]
From \eqref{eq:defK}, using \eqref{eq:GLamb},
we have
  \begin{align}\label{eq:K_bis}
K(V,\eta;\bar{V})&=C_{\ast}\int_{V}^{\bar{V}}  \, V_0^{\sigma-\frac{8}{3}}\Lambda\left(\frac{V_0-V}{V_0}\right) (\eta-V_0)^{-\sigma} (V_0^{\frac{1}{3}}+(\eta-V_0)^{\frac{1}{3}})^2 dV_0\quad \nonumber \\&
=C_{\ast}V^{\sigma-\frac{5}{3}} \int_{0}^{1-\frac{V}{\bar{V}}}  \Lambda(z)\big(\eta(1-z)-V\big)^{-\sigma}\left[V^{\frac 1 3}+\big(\eta(1-z)-V\big)^{\frac{1}{3}}\right]^2 dz
\end{align}
where we used the change of variables $z=1-\frac{V}{V_0}$, $dz=\frac{V}{V_0^2}dV_0$. 
Using the variables $\zeta,\,\theta$ defined in \eqref{eq:Kselfsim}, 
then \eqref{eq:K_bis} becomes
  \begin{align}\label{eq:K_bis2}
K(V,\eta;\bar{V})&=\frac{C_{\ast}}{\bar{V}} \theta^{\sigma-\frac{5}{3}} \int_{0}^{1-\theta}  \Lambda(z)\big(\zeta(1-z)-\theta \big)^{-\sigma}\left[\theta^{\frac 1 3}+\big(\zeta(1-z)-\theta \big)^{\frac{1}{3}}\right]^2 dz \nonumber \\&
=:\frac{C_{\ast}}{\bar{V}} Y(\theta,\zeta)
\end{align} 
for $V< \bar{V}<\eta$. This yields \eqref{eq:Kselfsim}. 

Using the function $P_{\alpha}(\theta,\zeta)$ defined as in \eqref{eq:defPalp} we can rewrite $Y(\theta,\zeta)$ and this yields \eqref{eq:defY}. Moreover, 
 using \eqref{eq:bdPalpha} from \eqref{eq:defY}, we obtain 
\begin{equation}\label{eq:bdY_bis}
0\leq Y(\theta,\zeta)\leq \frac{C}{(\zeta-\theta)}\left[ \frac{(1-\theta)^{\sigma-1}}{(\zeta-1)^{\sigma-1}}+\frac{(1-\theta)^{\sigma-1}}{(\zeta-1)^{\sigma-\frac{4}{3}}}+\frac{(1-\theta)^{\sigma-1}}{(\zeta-1)^{\sigma-\frac{5}{3}}}\right] 
\leq C\frac{\zeta^{\frac{2}{3}}}{(\zeta-\theta)} \frac{(1-\theta)^{\sigma-1}} {(\zeta-1)^{\sigma-1}}.
\end{equation}
Then, combining \eqref{eq:K_bis2} with \eqref{eq:bdY_bis} we get \eqref{eq:bdK}.

Furthermore, thanks to \eqref{eq:bdderivPalpha} we get
\begin{equation*}
\left\vert \frac{\partial Y}{\partial \theta} (\theta,\zeta)\right\vert \leq \frac{C}{\theta^2} \frac{(1-\theta)^{\sigma-2}\zeta^{\frac{2}{3}}}{(\zeta-1)^{\sigma}}.
\end{equation*}
Combining \eqref{eq:Kselfsim} with the estimate above, \eqref{eq:bdderivK} follows. 

\end{proof}

\bigskip

In what follows, we present the proof of Lemma \ref{lem:propM} and Lemma \ref{lem:wposasympG} which collect some properties of the function $M(z)$  defined as in \eqref{eq:defMik}. These technical lemmas were used to prove the continuity of the function $\Lambda(\xi)$ at $\xi=1$  in order to conclude the proof of Theorem \ref{Prop:OmINF}.

\begin{proof}[Proof of Lemma \ref{lem:propM}]
For $\Re(z)\leq 0$ we can perform an integration by parts in \eqref{eq:defMik} and we obtain 
\begin{align*} 
M(z)&=
z\sum_{j=0}^{2} \frac{c_j}{\big(\sigma-\frac{j}{3}-1\big)} \int_{0}^{\infty}  \frac{ \left(  1+\eta \right)^{z-1}}{\eta^{\sigma-\frac{j}{3}-1}}\deta
=z\sum_{j=0}^{2} \frac{c_j}{\big(\sigma-\frac{j}{3}-1\big)}\int_{0}^{1} \frac{\dx}{x^{z+2+\frac{j}{3}-\sigma} (1-x)^{\sigma-\frac{j}{3}-1} }\\&
= z\sum_{j=0}^{2} \frac{c_j}{\big(\sigma-\frac{j}{3}-1\big)} B(\sigma-\frac{j}{3}-1-z,2+\frac{j}{3}-\sigma)\\&
= \frac{z}{\Gamma(1-z)} \sum_{j=0}^{2} c_j \frac{\Gamma(2+\frac{j}{3}-\sigma)\Gamma(\sigma-\frac{j}{3}-1-z)}{\sigma-\frac{j}{3}-1}
\end{align*}
where we used the change of variables $1+\eta=1/x$ in the second equality. This proves \eqref{eq:extensionM}.

The poles of $M(z)$ are at the positions of the poles of the functions $\Gamma(\sigma-\frac{j}{3}-1-z)$ for $j=0,1,2$. This yields \eqref{eq:polesM}. In order to obtain the zeros of $M(z)$ we can rewrite $M(z)=\frac{z}{\Gamma(1-z)}\Gamma\left(  \sigma-1-z\right) Q(z)$ where $Q(z)$ is defined by means of
\begin{equation}\label{eq:Qanalfc}
Q\left(  z\right)  =\sum_{j=0}^{2}\frac{c_{j}\Gamma\left(  2+\frac{j}%
{3}-\sigma\right)  }{\sigma-\frac{j}{3}-1} \frac{\Gamma\left(  \sigma-1-\frac{j}{3}-z\right)  }{\Gamma\left(
\sigma-1-z\right)}.
\end{equation} 
Since $\Gamma(\sigma-1-z)$ does not have zeros for $z\in \C$ all the remaining zeroes of the function $M(z)$ are either the zeros of $\frac{z}{\Gamma(1-z)}$ or the zeros of the function $Q(z)$ defined as in \eqref{eq:Qanalfc}. The family of zeros $\hat{z}_{0,n}$ is just the family of zeros of the function $\frac{z}{\Gamma(1-z)}$. Then, \eqref{eq:zerosM} follows. 
In order to characterize the zeros of the function $Q(z)$, we begin analysing the zeros of $Q(z)$ in the real line. We observe that for $z<0$ there are no zeros, indeed $Q(z)>0$ since $\frac{c_{j}\Gamma\left(  2+\frac{j}%
{3}-\sigma\right)  }{\sigma-\frac{j}{3}-1}>0$ for $j=0,1,2$.  We now rewrite $Q(z)$ using that 
\begin{equation}\label{eq:identityGamma}
\Gamma(z)\Gamma(1-z)=\frac{\pi}{\sin(\pi z)}
\end{equation}
(see for instance formula 6.1.17 in \cite{AS}).
 We get 
 \begin{equation}\label{eq:Qanalfc_bis}
Q\left(  z\right)  =\sum_{j=0}^{2}\frac{c_{j}\Gamma\left(  2+\frac{j}{3}-\sigma\right)  }{\sigma-\frac{j}{3}-1}
 \frac{\Gamma\left( 2- \sigma+z\right)  }{\Gamma\left(
2-\sigma+\frac{j}{3}+z\right)} \frac{\sin(\pi (\sigma-1-z))}{\sin\big(\pi \big(\sigma-\frac{j}{3}-1-z\big)\big)}.
\end{equation}

The poles on $\C$ of $Q$ are given by $z_{j,n}=\sigma-\frac{j}{3}-1+n$ for $j=1,2$ and $n=0,1,2,\dots$ . Moreover, $z_{j,n}>0$ and real. We now compute the residue of the function $Q(z)$ at the points $z_{j,n}$. We have:
 \begin{align}\label{eq:residueQ}
\text{Res}(Q(z),z=z_{j,n})&=-\frac{c_{j}\Gamma\left(  2+\frac{j}{3}-\sigma\right)}{\pi\big(\sigma-\frac{j}{3}-1\big)}\frac{\Gamma\left( 2- \sigma+z_{j,n}\right)  }{\Gamma\left(
2-\sigma+\frac{j}{3}+z_{j,n}\right)}\sin\left(\frac{\pi j}{3}\right)\\& 
=-\frac{c_{j}\Gamma(2+\frac{j}{3}-\sigma)}{\pi\big(\sigma-\frac{j}{3}-1\big)}\frac{\Gamma(1-j/3+n)  }{\Gamma(
1+n)}\sin\left(\frac{\pi j}{3}\right) <0,
\end{align}
for $j=1,2,\,n=0,1,2,\dots$.

\begin{figure}[th]
\centering
\includegraphics [scale=0.32]{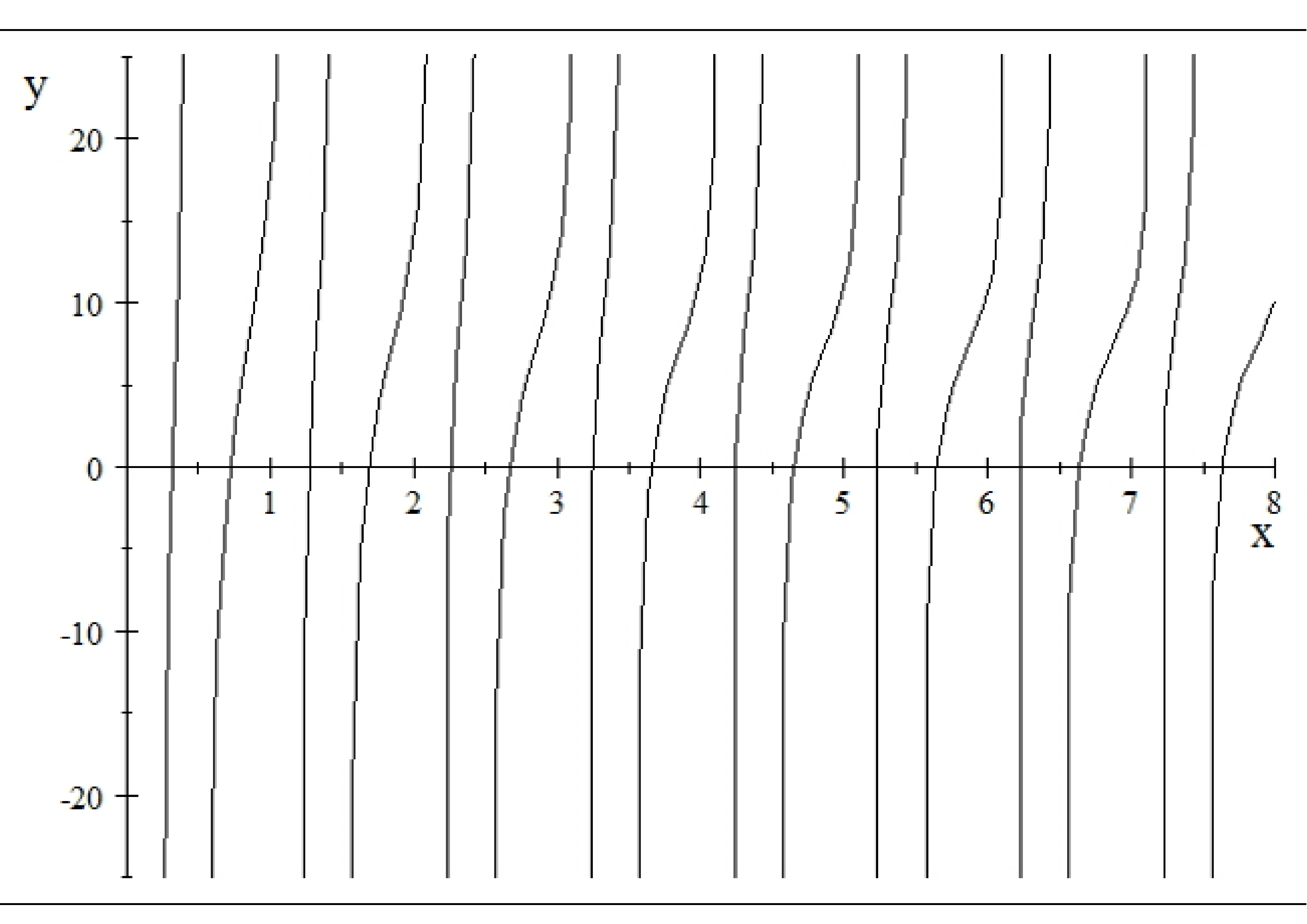}\caption{The zeros and the poles of the function $Q(z)$. \label{fig3} } 
\end{figure}

Then, since $Q(z)$ is real for $z\in\R$ it follows that there exists at least one zero in each interval between two consecutive poles. 
This shows that there exists a family of zeros of the function $Q(z)$ (and also $M(z)$) as in \eqref{eq:zerosM1}, \eqref{eq:zerosM2}.  We will denote as $\hat{z}_{1,n},\, \hat{z}_{2,n}$  the roots of $Q$ in the intervals indicated in \eqref{eq:zerosM1}, \eqref{eq:zerosM2} which are the closest to $z_{2,n}$, $z_{1,n}$ respectively.

It only remains to prove that the zeros $\hat{z}_{1,n},\, \hat{z}_{2,n}$ are unique in $(z_{2,n},z_{1,n})$ and $(z_{1,n},z_{2,n+1})$ respectively and that there are no other zeros of the function $Q(z)$ outside the real line. 
To this end we will apply the Argument Principle.  
Using Stirling's formula in \eqref{eq:Qanalfc} we obtain
\begin{equation}\label{eq:stirling}
\left\vert Q(z)-\frac{c_0\Gamma(2-\sigma)}{\Gamma(\sigma-1)}\right\vert \leq \frac{C}{\vert z\vert^{\frac{1}{3}}},
\end{equation}
for $\vert z\vert\geq 1$ and $\vert \text{Arg}(z)\vert \geq \eps_0>0$.  
Using again Stirling's formula in \eqref{eq:Qanalfc_bis} we have that the inequality \eqref{eq:stirling} holds if 
$\text{dist}(z,\bigcup_{j=1}^{2}\bigcup_{n=0}^{\infty}\{z_{j,n}\})\geq \eps_0>0$. 
It now follows from the  Argument Principle that in any ball $B_R(0)$, with $R>0$ sufficiently large, such that $\text{dist}(\partial B_R(0),\bigcup_{j=1}^{2}\bigcup_{n=0}^{\infty}\{z_{j,n}\})\geq \eps_0>0$ the number of poles and zeros of the function $Q(z)$ is the same.

Notice that $\text{Res}(Q(z),z=z_{j,n})$ in \eqref{eq:residueQ} converges to zero as $n\to\infty$ due to Stirling's formula. Therefore, a simple argument involving Rouch\'e 's theorem implies
\begin{equation}\label{eq:asymzjn}
\hat{z}_{j,n}\sim {z}_{j,n}- \frac{\text{Res}(Q(z),z=z_{j,n})}{\Gamma(2-\sigma)} \quad \text{as}\, n\to \infty.
\end{equation}

For the sake of simplicity let us assume that the radius $R\in (z_{1,n},z_{2,n})$. Hence, since $\text{Res}(Q(z),z=z_{j,n})<0$, it follows from \eqref{eq:asymzjn} that 
$$
z_{1,n}<\hat{z}_{1,n}<R-\eps_0<R.
$$
Therefore, the couples of poles and zeros $\{ z_{j,n},\hat{z}_{j,n}\}$ are contained in the ball $B_R(0)$ for $R>0$ sufficiently large. Since in $B_R(0)$ the number of zeros and poles is the same we have that  $\hat{z}_{j,n}$ are the only zeros of the function $Q(z)$.  

The bound \eqref{eq:boundMz} is a consequence of Stirling's formula as well as the identity \eqref{eq:identityGamma}. Arguing as in the proof of \eqref{eq:stirling} then the result follows.

\end{proof}

\medskip

\begin{proof}[Proof of Lemma \ref{lem:wposasympG}]

We first notice that the convergence of the integral in \eqref{eq:formulaG_1} follows from \eqref{eq:boundMz} and the fact that $\vert V^{z}\vert $ decreases exponentially along the contour $\mathcal{C}_{\alpha}$ as $|z|\to \infty$.
 
In order to prove the continuity of $\mathcal{G}$ it is sufficient to perform a deformation of the contour of integration $\mathcal{C}_{\alpha}$ to a new contour $\tilde{\mathcal{C}}_{\alpha}$ given by 
 \begin{equation*}
\tilde{\mathcal{C}}_{\alpha}=\{b+re^{-i\alpha}\,:\; 0<r<\infty\}\cup \{b+re^{i\alpha}\,:\; 0\leq r<\infty\} \qquad 0<\alpha<\frac{\pi}{2}
\end{equation*}
 where $b>0$ is chosen so small that the only zero of the function $M(z)$  in the interval $[-1,b]$ is at $z=0$ (see Figure \ref{fig2}). 
Hence, using the Residue theorem,  as well as $\Re(z)\geq a \vert z \vert$ for every $z\in \tilde{\mathcal{C}}_{\alpha}$, we have
\begin{equation}
\left\vert \mathcal{G}(V) -\frac{1}{M'(0)}\right\vert \leq CV^{\gamma} \int_{\tilde{\mathcal{C}}_{\alpha}}e^{a\vert z\vert  \log(V)} \vert dz\vert  
, \qquad 0\leq V \leq \frac 1 2
\end{equation}
with $C$ independent of $V$. 
It follows that $\mathcal{G}\in C[0,1)$ and $\mathcal{G}(0)=\frac{1}{M'(0)}>0$ with $M'(0)$ given by \eqref{eq:Mprime0}.

We now show that $\mathcal{G}$ solves \eqref{eq:solG} for $0<V<1$. We first observe that 
 for $\beta<0$ such that $\sigma-1+\beta>0$ the following holds:
\begin{equation}\label{eq:formulaG_2}
\mathcal{G}\left(  V\right)  =
-\frac{1}{2\pi i}\lim_{R\to\infty}\int_{\beta-R i}^{\beta+R i} \frac{V^{z}}{M\left(  z\right)  }\dz \qquad 0<V<1 \quad \text{or}\quad V>1,
\end{equation}
as it may be seen using contour deformation towards regions with $\Re(z)>0$ or $\Re(z)<0$ respectively.

We now choose a test function $\phi(V)$ such that $\supp(\phi)\subset (0,1)$ and, combining \eqref{eq:infOmega} with \eqref{eq:formulaG_2}, we obtain
\begin{align*}
&-\int_0^1 \phi(V)\mathcal{K}_{\infty}\mathcal{G}(V) \dV \\&
=\frac{1}{2\pi i} \int_0^1 \phi(V)  \int_{0}^{\infty}
 v^{-\sigma}(V^{1/3}+v^{1/3})^2\times \\
 &\qquad\qquad\qquad\qquad \times\left( \lim_{R\to\infty}\int_{\beta-R i}^{\beta+R i} 
 \frac{(V+v)^{z}}{M\left(  z\right)  }\dz -\lim_{R\to\infty}\int_{\beta-R i}^{\beta+R i} \frac{V^{z}}{M\left(  z\right)  }\dz \right) \dv\dV\\&
=\frac{1}{2\pi i}  \lim_{R\to\infty} \int_0^1 \phi(V) \int_{\beta-R i}^{\beta+R i}\frac{V^{z+\frac{5}{3}-\sigma}}{M\left(  z\right)  }   
\int_{0}^{\infty} \frac{(1+\eta^{1/3})^2}{\eta^{\sigma}}[ (1+\eta)^{z}-1] \deta\dz\dV\\&
=\frac{1}{2\pi i}  \lim_{R\to\infty} \int_0^1 \phi(V)V^{\frac{5}{3}-\sigma} \int_{\beta-R i}^{\beta+R i}V^{z} \dz=0, 
\end{align*}
where we used the change of variables  $v=V\eta$ in the second equality and  \eqref{eq:defMik} in the third equality.  
Notice that we can exchange the limit and the integral and apply Fubini's theorem in the first and second equality because we have the following uniform estimates. Note that $\frac{8}{3}-\sigma<1$.

For $V<1-\delta_1$, $V+v<1-\delta_2$ with $\delta_1,\,\delta_2>0$
\begin{align}\label{estimate1}
&\left\vert \int_{\beta-iR}^{\beta+iR}\frac{\left(  V+v\right)
^{z}}{M(z)}\dz-\int_{\mathcal{C}_{\alpha}}\frac{V^{z}}{M(z)}\dz \right\vert \leq  C_{\delta_1,\delta_2}v.
\end{align}
For $1-\delta_2 \leq V+v<1$, or $V+v>1$ with $\delta_2>0$
\begin{align}\label{estimate2}
&\left\vert \int_{\beta-iR}^{\beta+iR}\frac{\left(  V+v\right)
^{z}}{M(z)}\dz\right\vert  \leq  \frac{C}{\vert 1-(V+v)\vert ^{\frac{8}{3}-\sigma}}.
\end{align}
We also have
\begin{align}\label{estimate3}
&\left\vert \int_{\beta-iR}^{\beta+iR}\frac{V^{z}}{M(z)  }\dz\right\vert  \leq  C_{\delta_1}.
\end{align}
In the cases \eqref{estimate1}, \eqref{estimate3}, as well as  \eqref{estimate2} for $V+v<1$ the estimates are proved replacing the contour of integration $[\beta-Ri,\beta+Ri]$ by the contour in Figure \ref{fig4} (case 1).

\begin{figure}[th]
\centering
\includegraphics [scale=0.35]{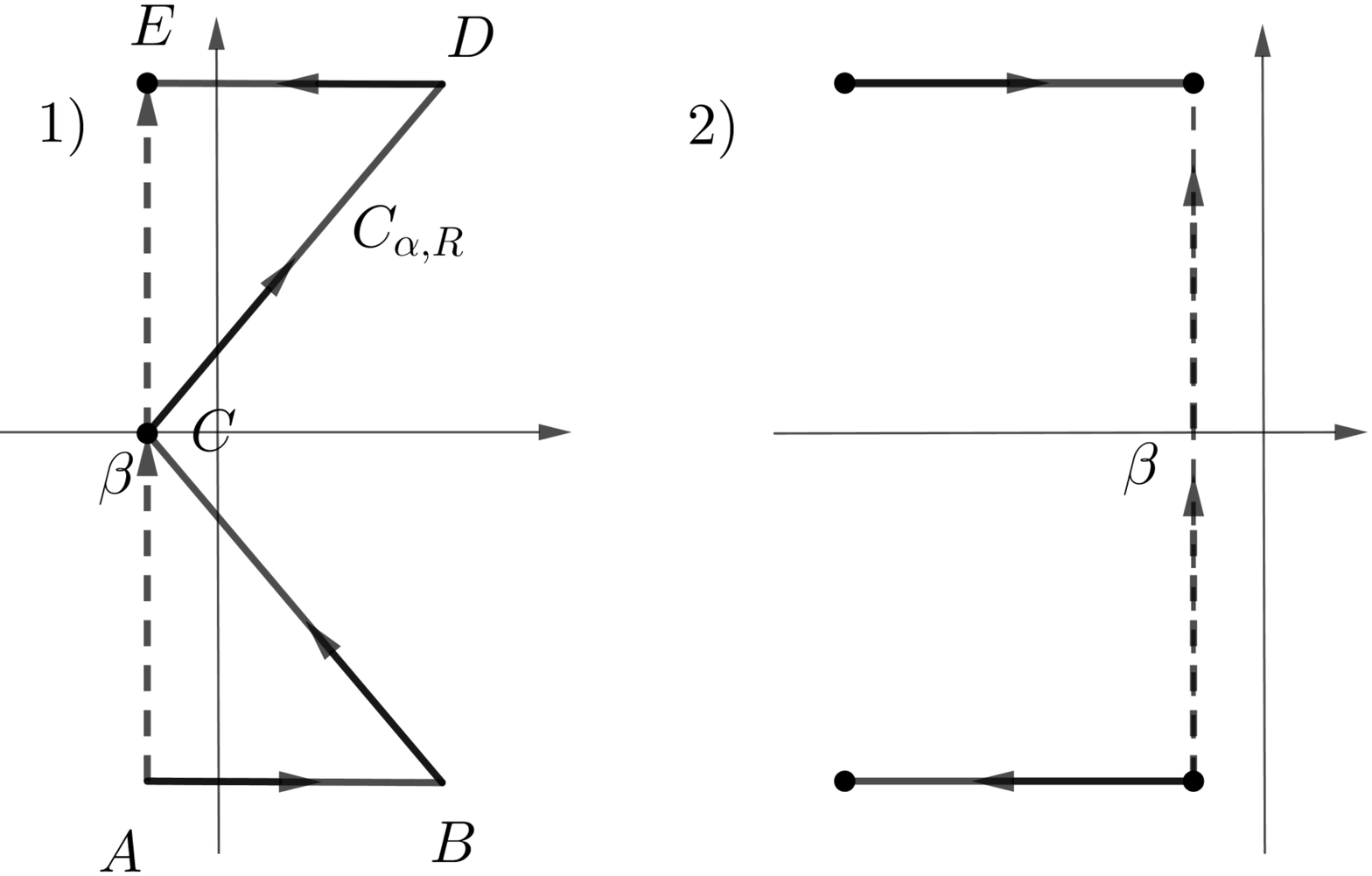}\caption{The dashed lines are the segment $[\beta-iR,\beta+iR]$. \label{fig4} }%
\end{figure}

The contributions due to the horizontal segments can be estimated, using \eqref{eq:boundMz}, as follows:
\begin{align*}
\left\vert \int_{\bar{ED}}\frac{\left(
V+v\right)  ^{z}}{M(z)  }dz\right\vert  
& \leq\frac
{C}{R^{\sigma-\frac{5}{3}}}\int_{0}^{R}e^{\log\left(  V+v\right)  s}ds \\
& \leq\frac{C}{\left(  \left\vert \log\left(  V+v\right)  \right\vert
R\right)  ^{\sigma-\frac{5}{3}}}\frac{\left(  1-e^{-\left\vert \log\left(
V+v\right)  \right\vert R}\right)  }{\left\vert \log\left(  V+v\right)
\right\vert ^{\frac{8}{3}-\sigma}}\\
& \leq\frac{C}{\left\vert \log\left(  V+v\right)  \right\vert ^{\frac{8}{3}-\sigma}%
}\leq\frac{C}{\left\vert 1-\left(  V+v\right)  \right\vert ^{\frac{8}{3}-\sigma}},
\end{align*}
where the segment $\bar{ED}$ is as in Figure \ref{fig4}. 

On the other hand, the contribution due to the portion of the contour containing $\mathcal{C}_{\alpha}$, that we will denote as $\mathcal{C}_{\alpha,R}$ (the contour $BCD$ in Figure \ref{fig4} case 1.), can be bounded, using \eqref{eq:boundMz} again, by:
\begin{align*}
& \left\vert \int_{\mathcal{C}_{\alpha,R}}\frac{\left(  V+v\right)
^{z}}{M\left(  z\right)  }dz\right\vert 
 \leq \int_{\mathcal{C}_{\alpha,R}}%
\frac{\left\vert \exp\left(  \operatorname{Re}\left(  z\right)  \log\left(
V+v\right)  \right)  \right\vert }{\left\vert M\left(  z\right)  \right\vert
}\left\vert dz\right\vert \\
& \leq C \int_{\mathcal{C}_{\alpha,R}}%
\frac{\left\vert \exp\left(  \operatorname{Re}\left(  z\right)  \log\left(
V+v\right)  \right)  \right\vert }{\left\vert z\right\vert ^{\sigma-\frac{5}{3}}%
+1}\left\vert dz\right\vert \\
& \leq\frac{C}{\left\vert \log\left(
V+v\right)  \right\vert ^{\frac{8}{3}-\sigma}}\leq\frac{C}{\left\vert 1-\left(
V+v\right)  \right\vert ^{\frac{8}{3}-\sigma}}.
\end{align*}

The proof of \eqref{estimate2} in the case $V+v>1$ is made similarly, replacing the contour of integration with the one in Figure \ref{fig4} (case 2.)

We finally prove the asymptotics of the function $\mathcal{G}$ given by \eqref{eq:asymptoticsG}. In order to do this we compute the asymptotics of $M\left(  z\right)$, defined as in \eqref{eq:extensionM}, as $\left\vert
z\right\vert \rightarrow\infty.$   
Using this formula, it follows that
\[
M\left(  z\right)  =\frac{\Gamma\left(  2-\sigma\right)  }{\left(
\sigma-1\right)  }z\left(  -z\right)  ^{\sigma-2}\left(  1+O\left(  \frac
{1}{\left\vert z\right\vert ^{\frac{1}{3}}}\right)  \right)  \ \ \text{as
}\left\vert z\right\vert \rightarrow\infty.
\]
We then compute the asymptotics of $\mathcal{G}\left(  V\right)  $ as $V\rightarrow
1^{-}.$ Using \eqref{eq:formulaG_2}, we have:
\begin{equation}\label{eq:asymG_bis}
\mathcal{G}\left(  V\right)  
=-\frac{\left(  \sigma-1\right)  }
{2\pi\Gamma\left(  2-\sigma\right)  i}\int_{\mathcal{C}_{\alpha}}\frac{\left(
-z\right)  ^{2-\sigma}V^{z}}{z}dz+\int_{\mathcal{C}_{\alpha}}O\left(  \frac{\left(
-z\right)  ^{2-\sigma}}{\left\vert z\right\vert ^{1+\frac{1}{3}}}\right)
\left\vert V^{z}\right\vert \left\vert \dz\right\vert
\end{equation}
Using the change of variable $X=\vert \log(V)\vert z$ and deforming the contour from $\frac{\mathcal{C}_{\alpha}}{\log(V)}$ to $\mathcal{C}_{\alpha}$ we can rewrite the first term in the right hand of the equation above as  $ \bar{K}  \vert \log(V)\vert^{\sigma-2} $ where
\begin{equation*}
 \bar{K}=-\frac{\left(  \sigma-1\right) }{2\pi\Gamma\left(
2-\sigma\right)  i}\int_{\mathcal{C}_{\alpha}} ( -X)^{1-\sigma
} e^{-X} \dX= \frac{(  \sigma-1) \sin(\pi (\sigma-1))  }{\pi }.
\end{equation*}
To estimate the remainder in \eqref{eq:asymG_bis} we argue similarly. Hence, we can estimate $\mathcal{G}(V)$ as $O\left(\vert \log(V)\vert^{\sigma-\frac{5}{3}} \right)$. Using that
$ \log(V)=(V-1)+\dots$, as $V\to 1^{-}$, \eqref{eq:asymptoticsG} follows.

\end{proof}

\end{document}